\documentclass[leqno]{amsart}

\usepackage{amssymb,amsfonts,amsthm}
\usepackage[dvipsnames]{xcolor}
\usepackage{graphicx}
\usepackage{subfigure}
\usepackage{amsthm}
\usepackage{tabularx}
\usepackage{amsmath}
\usepackage{bm}             
\usepackage[mathscr]{eucal}
\usepackage{color} 
\usepackage{cancel}
\usepackage{enumerate}
\usepackage{psfrag}
\usepackage{bbm}
\usepackage{appendix}
\usepackage[colorlinks=true,linkcolor=red,citecolor=blue]{hyperref}
\usepackage{dsfont}

\usepackage[normalem]{ulem}
\usepackage{cleveref}

\newcommand{\A}{\mathbb{A}}

\newcommand{\R}{\mathbb{R}}
\newcommand{\bfR}{\mathbb{R}}
\newcommand{\N}{\mathbb{N}}

\newcommand{\mcA}{\mathcal{A}}
\newcommand{\mcE}{\mathcal{E}}
\newcommand{\mcK}{\mathcal{K}}
\newcommand{\mcM}{\mathcal{M}}
\newcommand{\mcS}{\mathcal{S}}
\newcommand{\mcD}{\mathcal{D}}
\newcommand{\mcG}{\mathcal{G}}
\newcommand{\mcH}{\mathcal{H}}

\newcommand{\mcP}{\mathcal{P}}
\newcommand{\mcR}{\mathcal{R}}
\newcommand{\mcPa}{\mathcal{P}^\alpha}

\newcommand{\mcXa}{\mathcal{X}^\alpha}
\newcommand{\mcV}{\mathcal{V}^\alpha}
\newcommand{\mcY}{\mathcal{Y}^\alpha}

\newcommand{\mcU}{\mathcal{U}}
\newcommand{\mcQ}{\mathcal{Q}}

\newcommand{\mfG}{\mathfrak{G}}

\newcommand{\bS}{\mcS^\alpha}

\newcommand{\infnorm}[1]{\left\| #1 \right\|_{\infty}} 
\newcommand{\rndP}[1]{\left(#1 \right)} 

\newcommand{\tildes}[1]{\left|{#1}\right|} 

\newtheorem{theorem}{Theorem}[section]
\newtheorem*{theorem*}{Theorem}

\newtheorem{lemma}[theorem]{Lemma}

\theoremstyle{definition}

\newtheorem{remark}{Remark}[section]

\usepackage{fullpage}

\title{Sharp Time Decay and Scattering of Small Data Solutions to the Relativistic Vlasov-Maxwell System}
\author{Grace Mattingly, Stephen Pankavich}
\address{Department of Applied Mathematics and Statistics, Colorado School of Mines, Golden, CO 80401.}
\email{gmattingly@mines.edu}
\email{pankavic@mines.edu}

\author{Jonathan Ben-Artzi}
\address{Dipartimento di Matematica, Universit\`a degli Studi di Roma ``Tor Vergata'', Via della Ricerca Scientifica 1, 00133 Rome, Italy.}
\email{benartzi@mat.uniroma2.it}

\date{\today}
\thanks{The first author was supported by an NSF Graduate Research Fellowship. The second author was supported in part by NSF grant DMS-2107938. The last author acknowledges the MUR Excellence Department Project MatMod@TOV awarded to the Department of Mathematics, University of Rome Tor Vergata, CUP E83C23000330006}

\begin{document}

%
%
%
%
%
%
%

\maketitle

\begin{abstract}
A multispecies, collisionless plasma is modeled by the relativistic Vlasov-Maxwell system. 
Assuming the plasma is globally neutral, we show that solutions can be constructed with arbitrarily fast, polynomial rates of decay, depending upon the cancellation properties of the limiting spatial averages, as well as the masses and charges of the particles.
In doing so, we establish a countably infinite number of asymptotic profiles for the charge and current density, electric and magnetic fields, and their derivatives, each of which is necessarily realized by a sufficiently smooth solution.
Our approach relies upon a refined field representation that features self-similar behavior away from the boundary of the light cone.
In each case we establish a linear scattering result in $L^\infty$ for every particle distribution function, namely we show that they converge as $t \to \infty$ along the transported spatial characteristics at increasingly faster rates.
In the case that charge cancellation does not occur, our methods can be further extended to demonstrate polyhomogeneous expansions with logarithmic corrections.
\end{abstract}

%
%
%
%
%
%
%
%
%
%
%

\section{Introduction}
\label{sec:intro}

We consider the motion of a collisionless plasma, namely a completely ionized gas that is sufficiently dilute to neglect collisional effects and sufficiently hot to consider relativistic velocities, that is comprised of a large number of charged particles. If there are $N\in\N$ distinct species of particles within the plasma, then those of the $\alpha$th species (for $\alpha = 1, ..., N$)
possess a charge $e_\alpha \in \mathbb{R}$, mass $m_\alpha > 0$, and are distributed in phase space at time $t \geq 0$ according to the function $f^\alpha(t,x,v)$ where $x \in \bfR^3$ represents particle position and $v\in \bfR^3$ particle velocity. 
The time evolution of the plasma is described by the multispecies, relativistic Vlasov-Maxwell system.
In particular, the motion of particles is governed by the Vlasov equation
\begin{equation}
\label{Vlasov}
\tag{RVM$_1$}
\partial_{t}f^\alpha+v_\alpha (p)\cdot\nabla_{x}f^\alpha + e_\alpha \left ( E + \frac{v_\alpha(p)}{c} \times B \right ) \cdot\nabla_{p}f^\alpha=0,
\end{equation}
where $c$ is the speed of light and the relativistic velocity of the $\alpha$th species is given by
\begin{equation}
\label{Rel_vel-c}
\tag{RVM$_2$}
v_\alpha(p) = \frac{p}{\sqrt{m_\alpha^2 + |p|^2/c^2}}.
\end{equation}
We note that this function is smooth and invertible for $|q| < c$, with inverse function
\begin{equation}
\label{Rel_vel_inv}
v^{-1}_\alpha(q) = \frac{m_\alpha q}{\sqrt{1- |q|^2/c^2}}.
\end{equation}
The distribution function gives rise to charge and current densities, defined by
\begin{equation}
\label{rhoj}
\tag{RVM$_3$}
\rho(t,x) =4\pi \sum_{\alpha=1}^N e_\alpha  \int_{\bfR^3} f^\alpha(t,x,p) \,dp \qquad \mathrm{and} \qquad j(t,x) = 4\pi \sum_{\alpha=1}^N e_\alpha \int_{\bfR^3} v_\alpha(p) f^\alpha(t,x,p)  \ dp,
\end{equation}
and these densities generate electromagnetic fields $E(t,x)$ and $B(t,x)$ according to Maxwell's equations
\begin{equation}
\label{Maxwell}
\tag{RVM$_4$}
\begin{aligned}
& \partial_t E - c\nabla \times B = - j, \hspace{.76cm} \nabla \cdot E = \rho\\
& \partial_t B + c\nabla \times E = 0, \hspace{1cm} \nabla \cdot B = 0.
\end{aligned}
\end{equation}
The equations \eqref{Vlasov}-\eqref{Maxwell}, which we will group together and denote by (\hyperref[Vlasov]{RVM}) are then supplemented by the initial conditions 
$f^\alpha(0,x,p) = f^\alpha_0(x,p)\geq0$ for each $\alpha = 1, ..., N$,
and $E(0,x) = E_0(x)$, $B(0,x) = B_0(x)$,
which must further satisfy the compatibility conditions
\begin{equation}
\label{compat}
\nabla \cdot E_0(x) = \rho(0,x), \qquad \nabla \cdot B_0(x) = 0, \qquad \mathrm{and} \qquad \int_{\bfR^3} \rho(0,x) \ dx = 0.
\end{equation}
For each species, the particle number is conserved, namely
$$ \iint_{\bfR^3\times\bfR^3} f^\alpha(t,x,p) \ dpdx =  \iint_{\bfR^3\times\bfR^3} f^\alpha_0(x,p) \ dpdx =: \mcM^\alpha$$
for all $t \geq 0$, with the overall net charge of the system given by
$$\mcM_{\mathrm{net}} = \sum_{\alpha = 1}^N e_\alpha \mcM^\alpha = 0 $$
due to the global neutrality assumption within the last equality of \eqref{compat}.

While it is well-known that given sufficiently small and smooth initial data with compact support in phase space, \eqref{Vlasov}-\eqref{Maxwell} possesses a global-in-time smooth solution \cite{BigorgneRVM,GS}, the large time behavior of solutions has only recently become better understood \cite{Bigorgne, Pan-BA2025}. Indeed, \cite{Bigorgne} used vector field methods to determine the asymptotic behavior of small data solutions to \eqref{Vlasov}-\eqref{Maxwell} from sufficiently regular initial data without the need for compact support, while \cite{Pan-BA2025} used more classical PDE estimates to arrive at a similar result for $C^2$ data, with improved decay rates
in the case that the limiting charge density vanishes.
Aspects of the proof in \cite{Pan-BA2025}, specifically the representation of the electric and magnetic fields, were then simplified within \cite{Breton}.
Additionally, the more recent study \cite{BigorgneScattering} further constructed a small data scattering map to connect solutions as $t \to -\infty$ with initial data at $t=0$ and the forward limiting behavior of quantities within the system as $t \to \infty$.
The analogous result for Vlasov-Poisson was constructed within \cite{Flynn}.
Finally, \cite{Breton2} recently proved that small data solutions of , \eqref{Vlasov}-\eqref{Maxwell} do not verify linear scattering in $L^1$.
Of course, the global-in-time existence and large time asymptotic behavior of classical solutions to \eqref{Vlasov}-\eqref{Maxwell} for data of arbitrary size also remains a crucial unsolved problem in the field.

Much more is known for the Vlasov-Poisson system, which is the classical limit of the relativistic Vlasov-Maxwell system. 
In particular, solutions are known to decay in special cases, including small data \cite{BD, HRV, Ionescu, Pankavich2022, Smulevici}, monocharged and spherically-symmetric data \cite{Horst, Pankavich2020}, and lower-dimensional settings \cite{BKR, BMP, GPS, GPS2, GPS4, GPS5, Sch}.
An understanding of the intermediate asymptotic behavior of spherically-symmetric solutions of the relativistic Vlasov-Maxwell system was obtained in \cite{BCP2} with the analogous result for Vlasov-Poisson occurring in \cite{BCP1}, namely that there are solutions for which the $L^\infty$ norms of the charge density and electric field can be made arbitrarily large at some later, finite time regardless of their initial size.
For general background concerning \eqref{Vlasov}-\eqref{Maxwell} and associated kinetic equations, we refer the reader to \cite{Glassey}. 

Because of the dispersive properties imparted upon the system by the relativistic transport operator $\partial_t + v_\alpha(p) \cdot \nabla_x$ and the momentum averages inherent to observable quantities, one generally expects that the charge and current densities and electromagnetic fields decay to zero like $t^{-3}$ and $t^{-2}$, respectively, as $t \to \infty$ for all smooth solutions.
That being said, previous results for the Vlasov-Poisson system \cite{Mattingly2025,Pankavich2021}, and more recently \cite{Pan-BA2025} for the relativistic Vlasov-Maxwell system, have demonstrated that neutral plasmas ($\mcM = 0$) can sustain decay rates that are faster than those generated by dispersion should the limiting charge density vanish due to charge cancellation.
Conversely, non-neutral plasmas ($\mcM \neq 0$) cannot experience decay that is faster than the rates provided by the dispersive mechanisms in the system \cite{Pankavich2021}.
Within the current paper, we significantly extend this idea to small data solutions of \eqref{Vlasov}-\eqref{Maxwell}, demonstrating that faster decay rates of any polynomial order can be attained by solutions should successive limiting quantities within the system vanish.
In this way, we demonstrate that the relativistic Vlasov-Maxwell system allows for infinitely many distinct regimes of asymptotic behavior, though the scattering behavior of the distribution function can only assume two distinct states (linear and modified), albeit with accelerated rates of convergence.

\subsection{Notation and Approach}
For simplicity, we normalize the speed of light $c = 1$ throughout. Doing so does not affect our results in any way.
In particular, this simplifies the relativistic velocity function, its gradient, and inverse determinant to
\begin{equation}
\label{Rel_vel-c=1}
\tag{RVM$_2$}
v_\alpha(p) = \frac{p}{\sqrt{m_\alpha^2 + |p|^2}},
\end{equation}
$$\mathbb{A}_\alpha(p) := \nabla v_\alpha(p),$$
and 
$$\mcD_\alpha(p) := \left | \det \mathbb{A}_\alpha(p) \right |^{-1}$$
for $\alpha = 1, ..., N$, respectively.
With this, the characteristics associated to (\hyperref[Vlasov]{RVM}) are defined by $\mcXa(t, \tau, x, p)$ and $\mcPa(t, \tau, x, p)$ for all $t, \tau \geq 0$ and $x,p \in \bfR^3$ as solutions of the system
\begin{equation}
\label{char}
\left \{
\begin{aligned}
&\dot{\mcXa}(t, \tau, x, p)= v_\alpha \left ( \mcPa(t, \tau, x, p) \right )\\
&\dot{\mcPa}(t, \tau, x, p)=  K^\alpha \left (t,\mcXa(t, \tau, x, p), \mcPa(t, \tau, x, p) \right)
\end{aligned}
\right.
\end{equation}
with initial conditions
$\mcXa(\tau, \tau, x, p) = x$ and
$\mcPa(\tau, \tau, x, p) = p,$
where the Lorentz force on the $\alpha$th species is defined by
$$K^\alpha(t,x,p) = e_\alpha \left (E(t, x) +\frac{v_\alpha(p)}{c}\times B(t, x) \right ).$$

Due to the appearance of the relativistic transport operator, our approach first recasts the system along the associated free flow via the transformation
$$g^\alpha (t, x, p) = f^\alpha(t, x + v_\alpha(p)t, p)$$
for $\alpha = 1, ..., N$
so that the Vlasov equation becomes

	\begin{equation}
	\tag{RVM$_g$}
	\label{RVMg}
	\partial_{t}g^\alpha -  \frac{e_\alpha}{m_\alpha}t \mathbb{A}_\alpha(p)K^\alpha(t,x+v_\alpha(p)t, p)\cdot \nabla_{x}g^\alpha+\frac{e_\alpha}{m_\alpha}K^\alpha(t,x+v_\alpha(p)t, p) \cdot\nabla_{p}g^\alpha=0.
	\end{equation}
where we denote the operator acting on $g^\alpha$
	\begin{equation}
	\label{eq:vlasov-op-g}
	\mcV_g 
	=
	\partial_{t} + \frac{e_\alpha}{m_\alpha}  K^\alpha(t,x+v_\alpha(p)t, p) \cdot \left( - t \A_\alpha(p) \nabla_{x} + \nabla_{p} \right).
	\end{equation}

Additionally, we transform Maxwell's equations into linear wave equations with source dependence on derivatives of the charge and current density, which yields, for $i=1,2,3$,
	\begin{equation}
	\label{eq:wave-e}
	\Box_{t,x} E^i(t,x) = - \partial_{x_i} \rho(t,x) - \partial_t j^i(t,x)
	\end{equation}
	\begin{equation}
	\label{eq:wave-b}
	\Box_{t,x} B^i(t,x) = \left (\nabla \times j \right)^i(t,x),
	\end{equation}
where for the \emph{box operator} we employ the convention
	\[
	\Box_{t,x}=\partial_t^2-\Delta_x.
	\]
Upon changing variables (via $p \mapsto y = x + v_\alpha(p)t$) in this reference frame, the charge density is expressed as
	\begin{equation}
	\label{rhop}
	\rho(t,x) 
	= 4\pi t^{-3}\sum_{\alpha=1}^N e_\alpha \int g^\alpha \left (t, y, v_\alpha^{-1} \left (\frac{x-y}{t}  \right ) \right ) \mcD_\alpha\left (v_\alpha^{-1} \left (\frac{x-y}{t}  \right ) \right ) dy.
	\end{equation}
From this representation, the large time behavior becomes more transparent, as one expects
$$\rho(t,x) \sim 4\pi t^{-3} \sum_{\alpha=1}^N e_\alpha \mcD_\alpha\left (v_\alpha^{-1} \left (\frac{x}{t}  \right ) \right ) \int g^\alpha \left (t, y, v_\alpha^{-1} \left (\frac{x}{t}  \right ) \right )  dy$$
as $t \to \infty$, under the assumption that the spatial support of $g^\alpha$ remains uniformly bounded in time.
In order to better understand the higher-order corrections to this limit, we merely implement a multi-dimensional Taylor expansion up to order $\ell$ of the integrand around the point $\frac xt$, assuming that the distribution function is sufficiently smooth. To simplify notation, we denote
	\begin{equation}
	\label{Hrho}
	\mcH_\rho(t,y,q)
	= 
	\sum_{\alpha = 1}^N e_\alpha \mcD_\alpha \left (  v_\alpha^{-1}(q) \right ) g^\alpha \left (t, y, v_\alpha^{-1}(q) \right ) 
	\end{equation}
so that
$$\rho(t,x) =  4\pi t^{-3}\int \mcH_\rho \left (t,y, \frac{x-y}{t}\right ) dy.$$
Then, expanding to remove the $y$ dependence within the last argument, we find
$$\mcH_\rho \left(t,y, \frac{x-y}{t} \right)
=
\sum_{m=0}^{\ell} \sum_{|\delta| = m}  \frac{(-y)^\delta}{t^{m}\delta!} D^\delta_q\,\mcH_\rho\left(t,y,\frac{x}{t}  \right) + \sum_{|\delta| = \ell+1}  \frac{(-y)^\delta}{t^{\ell+1}\delta!}  D^\delta_q\,\mcH_\rho\left(t,y,\frac{x-\theta y}{t} \right)$$
for some $\theta \in [0,1]$. Here $\delta \in \mathbb{N}_0^3$ is a multi-index and we use the conventions that $\delta!=\delta_1!\delta_2!\delta_3!$ and $z^\delta=z_1^{\delta_1}z_2^{\delta_2}z_3^{\delta_3}$ for $z\in\R^3$.
Therefore, for any $\ell \in \mathbb{N}_0$ we arrive at 
	\begin{align*}
	t^3 \rho(t,x) 
	=&
	4\pi  \sum_{m=0}^{\ell} t^{-m} \sum_{|\delta| = m}  \frac{1}{\delta!} \int (-y)^\delta D^\delta_q\,\mcH_\rho\left(t,y,\frac{x}{t} \right) dy\\
	 &+
	4\pi t^{-{\ell-1}} \sum_{|\delta| = \ell+1}  \frac{1}{\delta!} \int (-y)^\delta D^\delta_q\,\mcH_\rho\left(t,y,\frac{x - \theta y}{t} \right) dy
	\end{align*}
for some $\theta \in [0,1]$.
Define for any multi-index $\delta \in \mathbb{N}_0^3$, the quantities
\begin{equation}
\label{eq:F-al}
F^{\alpha,\delta}(t, p) = \int (-y)^\delta g^\alpha(t,y,p) \ dy
 \end{equation}
and
\begin{equation}
\label{eq:A-l}
\mcA_{\delta}(t,q) = \sum_{\alpha=1}^N e_\alpha \mcD_\alpha \left ( v_\alpha^{-1}(q) \right ) F^{\alpha,\delta} \left (t, v_\alpha^{-1}(q) \right ),
\end{equation}
the latter of which mimics the structure of $\mcH_\rho(t,y,q)$ but incorporates spatial moments.
Moving the spatial powers and integration inside of velocity derivatives, this simplifies the expression for $\rho(t,x)$ to
\[
t^3 \rho(t,x) 
=
4\pi \sum_{m=0}^{\ell} t^{-m} \sum_{|\delta| = m}  \frac{1}{\delta!} D^\delta_q \mcA_\delta \left(t,\frac{x}{t} \right) +
4\pi t^{-{\ell-1}} \sum_{|\delta| = \ell+1}  \frac{1}{\delta!} \int (-y)^\delta D^\delta_q\,\mcH_\rho\left(t,y,\frac{x - \theta y}{t} \right) dy.
\]
Upon defining
\begin{equation}
\label{eq:rho-l}
\rho_m(t,x)
=
4\pi \sum_{|\delta| = m}  \frac{1}{\delta!} D^\delta_q \mcA_\delta \left(t,\frac{x}{t} \right)
=
4\pi \sum_{|\delta| = m}  \frac{1}{\delta!} D^\delta_q \biggl ( \mcA_\delta \left(t,q \right) \biggr ) \biggr |_{q = \frac{x}{t}}
\end{equation}
this becomes
\begin{equation}
\label{rho_exp}
t^3 \rho(t,x) 
=
\sum_{m=0}^{\ell} t^{-m} \rho_m(t,x) +
4\pi t^{-{\ell-1}} \sum_{|\delta| = \ell+1}  \frac{1}{\delta!} \int (-y)^\delta D^\delta_q\,\mcH_\rho\left(t,y,\frac{x - \theta y}{t} \right) dy
\end{equation}
with the final term representing the error in the approximation of the charge density expansion.

Analogously, the current density can be represented as
	\begin{equation}
	\label{jp}
	\begin{split}
	j(t,x) 
	&=
	 4\pi \sum_{\alpha=1}^N e_\alpha  \int_{\bfR^3} v_\alpha(p)f^\alpha(t,x,p) \,dp\\
	& =
	 4\pi t^{-3}\sum_{\alpha=1}^N e_\alpha \int \frac{x-y}{t} g^\alpha \left (t, y, v_\alpha^{-1} \left (\frac{x-y}{t}  \right ) \right ) \mcD_\alpha\left (v_\alpha^{-1} \left (\frac{x-y}{t}  \right ) \right ) dy.
	\end{split}
	\end{equation}
Hence, proceeding in the same fashion we define
	\begin{equation}
	\label{Hj}
	\mcH_j(t,y,q)
	= q \mcH_\rho(t,y,q) = 
	\sum_{\alpha = 1}^N e_\alpha q \mcD_\alpha \left (  v_\alpha^{-1}(q) \right ) g^\alpha \left (t, y, v_\alpha^{-1}(q) \right ) 
	\end{equation}
so that
$$j(t,x) =  4\pi t^{-3}\int \mcH_j \left (t,y, \frac{x-y}{t}\right ) dy.$$
Ultimately, this yields the expansion (for some $\theta \in [0,1]$)
$$
t^3 j(t,x) 
=
\sum_{m=0}^{\ell} t^{-m} j_m(t,x) +
4\pi t^{-{\ell-1}} \sum_{|\delta| = \ell+1}  \frac{1}{\delta!} \int (-y)^\delta D^\delta_q\,\mcH_j\left(t,y,\frac{x - \theta y}{t} \right) dy
$$
for any $\ell \in \mathbb{N}_0$ where
\begin{equation}
\label{eq:j-l}
j_m(t,x)
=
4\pi \sum_{|\delta| = m}  \frac{1}{\delta!} D^\delta_q \biggl ( q \mcA_\delta \left(t,q \right) \biggr ) \biggr |_{q = \frac{x}{t}}.
\end{equation}

Finally, an expansion of the electric and magnetic fields is also necessary.
For brevity, we focus on the electric field and note that the construction for the magnetic field is identical upon altering the associated kernel.
Recalling the framework for the decomposition of the fields in \cite{GS2} and small data solutions in \cite{GS}, we let $\beta > 0$ represent the maximal momentum on the support of all distribution functions.
As in \cite{Pan-BA2025}   we denote
	\begin{equation}
	\label{eq:gamma}
	\gamma = \max \left \{ \frac{1}{2}, \frac{2\beta}{\sqrt{1 + 4\beta^2}} \right \}
	\end{equation}
to capture the maximal velocity of particle propagation inside the light cone.
Below, in \eqref{eq:glassey-strauss-decomp}, we shall see that the electric field has the convenient decomposition $E(t,x) = E_T(t,x) + E_S(t,x) + E_\mathrm{data}(t,x)$, where the so-called \emph{tangential} part $E_T$ remains dominant for large times. Its precise expression appears in \eqref{ET_GS}. Considering that expression, we define its kernel as
$$\kappa_T^E(\omega,q)=\frac{(\omega + q)(1-|q|^2)}{(1+q\cdot\omega)^2}$$
so that for $t$ sufficiently large and $|x| \leq \gamma t$
$$t^2 E(t,x) \sim \int_{|u|\leq \mcU} \int \mcH_T^E \left (t,u,z, \frac{x+ut-z}{t(1-|u|)} \right ) \, dz \, \frac{du}{|u|^2(1-|u|)^3}$$
where we use the standard notation for a unit-length vector $\hat{u}=\frac{u}{|u|}$,  where the constant $\mcU<1$ is given by
	\begin{equation}
	\label{eq:mcU}
	\mcU = \frac{1}{2} + \frac{\gamma}{\gamma + 1},
	\end{equation}
and
$$\mcH_T^E(t,u,z,q) = \kappa^E_T{\left(\hat{u},q\right)} \mcH_\rho(t(1-|u|), z, q).$$
Expanding to remove the $z$ dependence within the last argument of $\mcH_T^E$, we find
\begin{align*}
\mcH_T^E \left (t,u,z, \frac{x+ut-z}{t(1-|u|)} \right )
& =
\sum_{m=0}^{\ell} \frac{1}{[t(1 -|u|)]^m}\sum_{|\delta| = m}  \frac{(-z)^\delta}{\delta!} D^\delta_q\,\mcH_T^E \left (t,u,z, \frac{x+ut}{t(1-|u|)} \right )\\
& \qquad + \frac{1}{[t(1 -|u|)]^{\ell+1}}\sum_{|\delta| = \ell+1}  \frac{(-z)^\delta}{\delta!}  D^\delta_q\,\mcH_T^E \left (t,u,z, \frac{x+ut-\theta z}{t(1-|u|)} \right )
\end{align*}
for some $\theta \in [0,1]$.
From this, we construct for any $\ell \in \mathbb{N}_0$
\begin{equation}
\label{eq:W-l}
W^E_\ell(t,u, q)= \frac{1}{(1 -|u|)^\ell}\sum_{|\delta| = \ell} \frac{1}{\delta!} D_q^\delta \biggl (\kappa_T^E(\hat{u},q) \mcA_\delta (t(1 -|u|),q) \biggr )
\end{equation}
for $t \geq 0$, $|u| \leq \mcU < 1$, and $|q| < 1$, which mimics the structure of $\mcH_T^E(t, u, z, q)$ and, as before, moves the spatial moments from the expansion within the $\mcA_\delta$ term. 
Finally, this induces an approximating electric field via
\begin{equation}
\label{eq:E-l}
E_\ell(t,x) =  \int_{|u|\leq \mcU} W^E_\ell\rndP{t, u, \frac{x + ut}{t(1 - |u|)}} \, \frac{du}{|u|^2\rndP{1-|u|}^3}
\end{equation}
with a similar construction for the magnetic field.
The $\ell$th approximation of the electric field is then
$$
t^2 E(t,x) 
=
\sum_{m=0}^{\ell} t^{-m} E_m(t,x)
$$
with an error term that can be expressed in terms of previously defined quantities and is $\mathcal{O}\left (t^{-\ell+1}\right )$.

It will be shown that $F^{\alpha, \delta}(t, p)$ converges as $t \to \infty$ to
	\begin{equation}
	\label{eq:F-al-inf}
	F^{\alpha, \delta}_\infty(p)
	=  \int (-y)^\delta f^\alpha_\infty(y,p) \ dy,
	 \end{equation}
where $f^\alpha_\infty$ is smooth with compact support, and thus
$\mcA_\delta(t,q)$ will converge as $t \to \infty$ to
	\begin{equation}
	\label{eq:A-l-inf}
	\mcA_{\delta,\infty}(q)
	=
	\sum_{\alpha=1}^N e_\alpha \mcD_\alpha \left ( v_\alpha^{-1}(q) \right ) F^{\alpha,\delta}_\infty \left (v_\alpha^{-1}(q) \right )
	\end{equation}
for $|q| < 1$, as the functions $\mcD_\alpha(p)$ and $v_\alpha^{-1}(q)$ are well-behaved.
Upon controlling derivatives of these quantities, we will show that $\rho_\ell(t,x)$, $j_\ell(t,x)$, $W^E_\ell(t,u,q)$, and $E_\ell(t,x)$ converge, respectively, to	
	\begin{equation}
	\label{eq:rho-l-inf}
	\rho_{\ell,\infty}(q)
	=
	4\pi \sum_{|\delta| = \ell}  \frac{1}{\delta!} D^\delta_q \mcA_{\delta,\infty} \left(q\right),
	\end{equation}
	\begin{equation}
	\label{eq:j-l-inf}
	j_{\ell,\infty}(q)
	=
	4\pi \sum_{|\delta| = \ell}  \frac{1}{\delta!} D^\delta_q \biggl (q\mcA_{\delta,\infty} \left(q\right) \biggr ),
	\end{equation}
	\begin{equation}
	\label{eq:W-l-inf}
	W^E_{\ell,\infty}(u, q)= \frac{1}{(1 -|u|)^\ell}\sum_{|\delta| = \ell} \frac{1}{\delta!} D_q^\delta \left [\kappa_T^E(\hat{u},q) \mcA_{\delta,\infty} (q) \right ],
	\end{equation}
	\begin{equation}
	\label{eq:E-l-inf}
	E_{\ell,\infty}\rndP{q}=  \int_{|u|\leq \mcU} W^E_{\ell,\infty}\rndP{u, \frac{q + u}{1 - |u|}} \, \frac{du}{|u|^2\rndP{1-|u|}^3},
	\end{equation}
	\begin{equation}
	\label{eq:WB-l-inf}
	W^B_{\ell,\infty}(u, q)= \frac{1}{(1 -|u|)^\ell}\sum_{|\delta| = \ell} \frac{1}{\delta!} D_q^\delta \left [\kappa_T^B(\hat{u},q) \mcA_{\delta,\infty} (q) \right ],
	\end{equation}
and
	\begin{equation}
	\label{eq:B-l-inf}
	B_{\ell,\infty}\rndP{q}=  \int_{|u|\leq \mcU} W^B_{\ell,\infty}\rndP{u, \frac{q + u}{1 - |u|}} \, \frac{du}{|u|^2\rndP{1-|u|}^3}
	\end{equation}
with these limits occurring along $q = \frac{x}{t}$ and where $f_\infty^\alpha$ are the limiting functions introduced in Theorems \ref{oldT1} and \ref{oldT2}.

As the small data solutions we will study remain bounded on any finite time interval, we are essentially concerned only with large time estimates, and thus we use the notation
$$A(t) \lesssim B(t)$$
to represent the statement that there is $C > 0$, independent of $t$,
such that
$A(t) \leq CB(t)$
for $t$ sufficiently large.
When a specific constant is desired, for instance when there is $\gamma >0$  such that $|x| \leq \gamma t$ for $t$ sufficiently large, we will include the constant in this notation, namely
\begin{equation}
\label{xlesssim}
|x| \lesssim \gamma t.
\end{equation}
Additionally, $C$ will denote a positive constant (independent of the solution) that may depend upon initial data and can change from line to line. Finally, the notation 
	\[A(t) \sim B(t)\]
means that both $A(t) \lesssim B(t)$ and $B(t) \lesssim A(t)$ hold.

\subsection{Previous Results}

\subsubsection{Small Data Solutions}
We first summarize some previous properties of small data solutions, which were obtained in \cite{GS}. We denote by $\Gamma_L$ the open ball of radius $L$ in $\R^3$.

\begin{theorem}
\label{T0}
For any $L > 0$ there exist $\epsilon_0 > 0$ and $\beta > 0$ with the following property.
Let $f^\alpha_0 \in C^1$, $\alpha=1,\dots,N$, be non-negative functions supported on $\overline{\Gamma}_L \times \overline{\Gamma}_L$.
Let $E_0, B_0 \in C^2$ be supported on $\overline{\Gamma}_L$ satisfying the compatibility conditions \eqref{compat}.
If the initial data satisfy
$$\sum_{\alpha = 1}^N \Vert f^\alpha_0 \Vert_{C^1} + \Vert E_0 \Vert_{C^2} + \Vert B_0 \Vert_{C^2} \leq \epsilon_0,$$
then there exists a unique classical solution of (\hyperref[Vlasov]{RVM}) for all $x, p \in \bfR^3$ and $t \geq 0$ such that
$$f^\alpha(t,x,p) = 0 \qquad \mathrm{for} \qquad |p| \geq \beta$$
for all $\alpha = 1, ..., N$, $t\geq 0$, and $x \in \bfR^3$.
Furthermore, 
$|E(t,x)| + |B(t,x)| = 0$ for $|x| > t + L$
and we have the estimates
$$s - |\mcXa(s,t,x,p)| + 2L \geq \frac{s + L}{2 (1+\beta^2)}$$
for all $s, t \geq 0$,  $x,p \in \bfR^3$
and

$$|E(t,x)| + |B(t,x)| \leq \frac{C\epsilon_0}{(t+1)\left (t - |x | + 2L \right )},$$
$$|\nabla_x E(t,x)| + |\nabla_x B(t,x)| \leq \frac{C\epsilon_0 \ln(t+2)}{(t+1)\left (t - |x | + 2L \right )^2},$$
$$|\nabla^2_x E(t,x)| + |\nabla^2_x B(t,x)| \leq \frac{C\epsilon_0 \ln(t+2)}{(t+1)\left (t - |x | + 2L \right )^3},$$
$$\Vert \nabla_x f^\alpha(t) \Vert_\infty \leq C \qquad \mathrm{and} \qquad \Vert \nabla_p f^\alpha(t) \Vert_\infty \leq C(1+t),$$
and
$$|\rho(t,x)|  + | j(t,x)|  \leq C\epsilon_0 (1+t)^{-3}$$
for some $C>0$ and all $t \geq 0$, $x \in \bfR^3$.
\end{theorem}

Furthermore, for more regular initial data the decay of the fields can be extended to higher derivatives, as in \cite{Bigorgne}. This can also be displayed by the new field decomposition (Lemma \ref{Krep}) occurring later herein.
\begin{lemma}
\label{lem:glassey-strauss}
Assuming the initial distributions and fields are sufficiently smooth, the Glassey-Strauss theorem \cite{GS} further implies
	\begin{equation}
	|\nabla^k_x E(t,x)| + |\nabla^k_x B(t,x)| \lesssim \frac{ \ln(t+2)}{(t+1)\left (t - |x | + 2L \right )^{k+1}}.
	\end{equation}
for any $k \in \mathbb{N}$.
\end{lemma}

\subsubsection{Modified Scattering}
Next, we recall the large time convergence results of small data solutions, which were obtained separately within \cite{Bigorgne} and \cite{Pan-BA2025}, with a shorter proof of the latter theorem refined in \cite{Breton}.
In particular, we state the main results of those studies in our current notation.
First, the scaled densities and fields converge to limits based upon the limiting spatial average, namely

\begin{theorem}
\label{oldT1}
Under the conditions of Theorem \ref{T0}, the small data solutions satisfy
\begin{enumerate}[(a)]

\item 
For every $\alpha = 1, ..., N$
there exist $F^{\alpha,0}_\infty \in C_c^1(\bfR^3)$ 
such that the spatial average
$$F^{\alpha,0}(t,p) = \int f^\alpha(t,x, p) \ dx$$
satisfies 
$F^\alpha(t,p) \to F^\alpha_\infty(p)$ uniformly as $t \to \infty$, namely
$$\| F^{\alpha,0}(t) - F^{\alpha,0}_\infty \|_\infty \lesssim t^{-1}\ln^{5}(t),$$
and for every $\alpha = 1, ..., N$,
$$\int  F^{\alpha,0}_\infty(p) \ dp = \mcM^\alpha.$$

\item Due to the compact support of each $F_\infty^{\alpha,0}$, define the limiting charge density for $q \in \Gamma_1$ by
$$\rho_{0,\infty}(q) = \sum_{\alpha=1}^N e_\alpha \mcD_\alpha \left ( v_\alpha^{-1}(q) \right ) F^{\alpha,0}_\infty \left (v_\alpha^{-1}(q) \right )$$
and smoothly extend $\rho_{0,\infty}(q) = 0$ for $q \in \bfR^3 \setminus \Gamma_1$.
Similarly, define the limiting current density by
$$ j_{0,\infty}(q) = q \rho_{0,\infty}(q)$$
for $q \in \bfR^3$.
Then, the charge and current densities have the self-similar 
asymptotic profiles 
\begin{align*}
\sup_{x \in \bfR^3}   \left | t^3 \rho(t,x) - \rho_{0,\infty} \left (\frac{x}{t} \right) \right | & \lesssim t^{-1} \ln^{6}(t),\\
\sup_{x \in \bfR^3}  \left | t^3 j(t,x) - j_{0,\infty} \left (\frac{x}{t} \right) \right | & \lesssim t^{-1}\ln^{6}(t),
\end{align*}
and $\rho_{0,\infty}(q)$ satisfies
\begin{equation}
\label{Pinfmass}
\int \rho_{0,\infty}(q) \ dq = \mcM_{\mathrm{net}} = 0.
\end{equation}
Furthermore, for every $i = 1,2,3$ the derivatives satisfy
\begin{align*}
\sup_{x \in \bfR^3} \left | t^4 \partial_{x_i} \rho(t,x) - \partial_{q_i}\rho_{0,\infty} \left (\frac{x}{t} \right) \right | & \lesssim t^{-1} \ln^{8}(t),\\
\sup_{x \in \bfR^3}  \left | t^4 \partial_{x_i}j(t,x) -  \partial_{q_i}j_{0,\infty}\left (\frac{x}{t} \right) \right | & \lesssim t^{-1}\ln^{8}(t),\\
\sup_{x \in \bfR^3}  \left | t^4 \partial_t j^i(t,x) + \left [3j^i_{0,\infty}\left (\frac{x}{t} \right) +  \frac{x}{t} \cdot \nabla_qj^i_{0,\infty}\left ( \frac{x}{t} \right) \right ]  \right | & \lesssim t^{-1}\ln^{8}(t).
\end{align*}

\item Let $\gamma < 1$ and $\mcU < 1$ be as defined in \eqref{eq:gamma} and \eqref{eq:mcU}, respectively. Denote the unit vector in the direction of $u$ by $\hat{u} = \frac{u}{|u|}$.
Define the functions $E_{0,\infty}(q)$ and $B_{0, \infty}(q)$ by
\begin{equation}
\label{E0inf}
	E_{0,\infty}\rndP{q}=  \int_{|u|\leq \mcU} W^E_{0,\infty}\rndP{u, \frac{q + u}{1 - |u|}} \, \frac{du}{|u|^2\rndP{1-|u|}^3}
\end{equation}
where
$$ W^E_{0,\infty}(u, q)= \kappa_T^E(\hat{u},q) \rho_{0,\infty} (q)$$
and
\begin{equation}
\label{B0inf}
	B_{0,\infty}\rndP{q}=   \int_{|u|\leq \mcU} W^B_{0,\infty}\rndP{u, \frac{q + u}{1 - |u|}} \, \frac{du}{|u|\rndP{1-|u|}^3}
\end{equation}
where
$$ W^B_{0,\infty}(u, q)= \kappa_T^B(\hat{u},q) \rho_{0,\infty} (q)$$
respectively, for $|q| < \gamma$ and $|u| \leq \mcU$.
Then, $E$ and $B$ have the self-similar asymptotic profiles
\begin{align*}
\sup_{|x| \lesssim \gamma t} \left | t^2 E(t,x) - E_{0,\infty} \left (\frac{x}{t} \right ) \right | & \lesssim t^{-1}\ln^{8}(t),\\
\sup_{|x| \lesssim \gamma t} \left | t^2 B(t,x) - B_{0,\infty} \left (\frac{x}{t} \right ) \right | & \lesssim t^{-1}\ln^{8}(t).
\end{align*}
Additionally, the Lorentz force of the $\alpha$th species satisfies
$$\sup_{\substack{|x| \leq \ln(t)\\ |p| \leq \beta}}  \left | t^2 K^\alpha(t, x + v_\alpha(p)t , p) - K^\alpha_{0,\infty} \left ( p \right ) \right |  \lesssim t^{-1}\ln^{8}(t)$$
for all $\alpha = 1, ..., N$
where 
$$K^\alpha_{0,\infty}(p) = e_\alpha \bigg [E_{0,\infty}(v_\alpha(p)) + v_\alpha(p) \times B_{0,\infty}(v_\alpha(p) ) \bigg].$$


\item
For every $\alpha = 1, ..., N$
there is $f^\alpha_\infty \in C(\bfR^6)$ 
such that
$$f^\alpha \biggl (t,x +v_\alpha(p)t - \frac{e_\alpha}{m_\alpha} \ln(t) \mathbb{A}_\alpha(p) K^\alpha_{0,\infty}(p),p \biggr) \to f^\alpha_\infty(x,p)$$
uniformly
as $t \to \infty$, namely we have the convergence estimate
$$\sup_{(x,p) \in \bfR^6} \left | f^\alpha \biggl (t,x +v_\alpha(p)t - \frac{e_\alpha}{m_\alpha} \ln(t) \mathbb{A}_\alpha(p) K^\alpha_{0,\infty}(p), p \biggr) - f^\alpha_\infty(x,p) \right |  \lesssim t^{-1}\ln^{8}(t).$$
In particular, $f^\alpha_\infty$ is directly related to the limiting spatial average via
$$ F_\infty^{\alpha,0}(p) = \int f^\alpha_\infty(x,p) \ dx$$
for every $\alpha = 1, ..., N$.
\end{enumerate}
\end{theorem}

If the plasma is non-neutral, i.e. $\mcM \neq 0$, then $\rho_{0,\infty} \not\equiv 0$ due to \eqref{Pinfmass} and these estimates are sharp, up to a correction in the logarithmic powers of the error terms.
Hence, $\| \rho(t) \|_\infty, \| j(t) \|_\infty \sim \mathcal{O}\left(t^{-3} \right)$ and $\| (E,B)(t) \|_\infty \sim \mathcal{O}\left(t^{-2} \right)$.
However, when the plasma is neutral, i.e. $\mcM = 0$, it is possible that the limiting charge density $\rho_{0,\infty}$ (and hence, the current density $j_{0,\infty}$ and limiting fields $E_{0,\infty}$ and $B_{0,\infty}$) is identically zero, which implies stronger decay of $\| \rho(t) \|_\infty$, $\| E(t) \|_\infty$, $\|B(t) \|_\infty$, and related quantities.
Indeed, as we will show these quantities can decay at any polynomial rate greater than the above powers, depending upon the behavior of the limiting spatial averages. Note that a nontrivial neutral plasma is necessarily comprised of more than one species of charged particles, so that $N\geq2$ in what follows.

\begin{theorem}
\label{oldT2}
If $\rho_{0,\infty} \equiv 0$, then the asymptotic behavior described above is improved in the following manner:

\begin{enumerate}[(a)]

\item We have the faster decay estimates
\begin{align*}
& \| \rho(t) \|_\infty + \| j(t) \|_\infty \lesssim t^{-4}\ln^6(t),\\
& \| \nabla_x \rho(t) \|_\infty + \| \nabla_x j(t) \|_\infty + \| \partial_t j(t) \|_\infty \lesssim t^{-5}\ln^8(t), & 
\end{align*}
as well as
$$\sup_{|x| \lesssim \gamma t}  \left ( \left |E(t, x) \right | + \left |B(t,x) \right | \right ) \lesssim t^{-3}\ln^8(t).$$

\item The distribution functions scatter linearly, namely
for each $\alpha = 1, ..., N$ there is $f^\alpha_\infty \in C_c^1(\bfR^6)$  such that
$$f^\alpha(t,x +v_\alpha t,v_\alpha) \to f^\alpha_\infty(x,p)$$
uniformly
as $t \to \infty$ with the convergence estimate
$$\sup_{(x,p) \in \bfR^6} \left | f^\alpha(t,x +v_\alpha t, v_\alpha) - f^\alpha_\infty(x,p) \right |  \lesssim t^{-1}.$$

\end{enumerate}
\end{theorem}

Of course, the upper bounds listed above do not guarantee a sharp rate of decay or an identification of the correct asymptotic behavior of these quantities.
For instance, $\rho(t,x)  = j(t,x) \equiv 0$ satisfy these estimates and can be constructed from a rudimentary initial profile in a neutral plasma, for example, in which the distributions of all positive and negative charges overlap everywhere in phase space.
Therefore, we now turn our attention to stating new results that every order of decay is actually attained by some solution of \eqref{Vlasov}-\eqref{Maxwell}.

\subsection{Main Results}

We are now ready to present our main results, which extend the asymptotic behavior described by Theorems \ref{oldT1} and \ref{oldT2} to higher order expansions and close the estimates with bounds on higher order momentum derivatives of $g^\alpha$.
\begin{theorem}
\label{T1}
Let $m \in \N_0$ be given. There exist nontrivial $f^\alpha \in C^{m+9} \left((0,\infty) \times \mathbb{R}^6 \right)$  for every $\alpha = 1, ..., N$ satisfying \eqref{Vlasov}-\eqref{Maxwell} such that $\mcM_\mathrm{net} = 0$, and
$$
\begin{gathered}
\| \rho(t)\|_\infty, \| j(t)\|_\infty \sim t^{-m-3},\\
\| E(t)\|_\infty, \| B(t)\|_\infty \sim t^{-m-2}.
\end{gathered}
$$
If $m = 0$, then we have 
$$\| \nabla_x (E(t), B(t)) \|_\infty \sim t^{-3}$$
and for every $\alpha = 1, ..., N$ there is $f^\alpha_\infty \in C_c^{m+9}(\bfR^6)$ such that
$$\sup_{(x,p) \in \bfR^6} \left | f^\alpha \biggl (t,x +v_\alpha(p)t - \frac{e_\alpha}{m_\alpha}\ln(t) \mathbb{A}_\alpha(p) K^\alpha_{0,\infty}(p), p \biggr) - f^\alpha_\infty(x,p) \right |  \lesssim t^{-1}\ln^{8}(t).$$
In contrast, for $m \geq 1$  we have for every $k =1, ..., m$
$$\| \nabla_x^k (E(t), B(t))\|_\infty \sim t^{-m-3}, $$
and for every $\alpha = 1, ..., N$ there is $f^\alpha_\infty \in C_c^{m+9}(\bfR^6)$ such that
$$\sup_{(x,p) \in \bfR^6} \left | f^\alpha(t,x +v_\alpha(p)t, p) - f^\alpha_\infty(x,p) \right |  \lesssim t^{-m}. $$
\end{theorem}

\begin{remark}
To the best of our knowledge, Theorem \ref{T1} is the first result to demonstrate decay rates of the densities and fields that are nontrivial (i.e., $\rho, j, E, B \not\equiv 0$) and faster than the dispersive rates (of $t^{-3}$ and $t^{-2}$, respectively) for \eqref{Vlasov}-\eqref{Maxwell}.
\end{remark}

Theorem \ref{T1} will be implied by the following result (taking $m = n+1$ and $\mcA_{\delta,\infty} \not\equiv 0$ for some $|\delta| = n+1$) when $m \geq 1$, which will be established inductively.

\begin{theorem}
\label{T2}
Let $n\in\N_0$ be given. Consider the class of initial data $(f_0^\alpha, E_0, B_0) \in C_c^{n+10}(\R^6)$ launching solutions $f^\alpha\in C^{n+10}((0,\infty)\times\R^6)$ and $E,B \in C^{n+10}((0,\infty)\times\R^3)$ of \eqref{Vlasov}-\eqref{Maxwell} for $\alpha = 1, ..., N$. If
\begin{equation}
\label{Avanish}
\mcA_{\delta,\infty} \equiv 0 \qquad \mathrm{ for \ every \ } |\delta| \in \{0, ..., n\},
\end{equation}
then for every $\alpha = 1, ..., N$ there is $f^\alpha_\infty \in C_c^{n+10}(\bfR^6)$ with $\rho_{n+1,\infty}$, $j_{n+1,\infty}$, $E_{n+1,\infty}$, and $B_{n+1,\infty}$ defined by \eqref{eq:rho-l-inf}, \eqref{eq:j-l-inf}, \eqref{eq:E-l-inf}, and \eqref{eq:B-l-inf} respectively,
such that
$$
\begin{gathered}
\sup_{x \in \mathbb{R}^3} \left | t^{n+4}\rho(t,x) - \rho_{n+1, \infty} \left( \frac{x}{t}\right)  \right | \lesssim t^{-1}, \\
\sup_{x \in \mathbb{R}^3} \left | t^{n+4}j(t,x) - j_{n+1, \infty} \left( \frac{x}{t}\right)  \right | \lesssim t^{-1}, \\
\sup_{x \in \mathbb{R}^3} \left | t^{n+3} E(t,x) - E_{n+1,\infty} \left( \frac{x}{t}\right)  \right | \lesssim t^{-1}, \\ 
\sup_{x \in \mathbb{R}^3} \left | t^{n+3} B(t,x) - B_{n+1,\infty} \left( \frac{x}{t}\right)  \right | \lesssim t^{-1}, \\ 
\end{gathered}
$$
with
\begin{gather}
\sup_{x \in \mathbb{R}^3} \left | t^{n+3}\nabla^k_x(E,B)(t,x) - \nabla_q^k (E,B)_{n+1-k,\infty} \left( \frac{x}{t}\right)  \right | \lesssim t^{-1}, 
\label{eq:derivs-e} \\
\mcG^{n+2}_{x,p}(t) + \mcG^{n+2}_{p}(t) \lesssim 1\nonumber
\end{gather}
and
$$
\begin{gathered}
\sup_{x \in \mathbb{R}^3} \left | t^{n+4}\nabla^{n+2}_x(E,B)(t,x) - \nabla_q^{n+2}(E,B)_{0,\infty} \left( \frac{x}{t}\right)  \right | \lesssim t^{-1} \ln^2(t), \\
\end{gathered}
$$
In each of these cases, we further have
$$\sup_{(x,p) \in \bfR^6} \left | f^\alpha(t,x +v_\alpha(p)t, p) - f^\alpha_\infty(x,p) \right |  \lesssim t^{-n-1}$$
and
$$\sup_{(x,p) \in \bfR^6} \left | \nabla_x^i \nabla_p^j f^\alpha \left(t,x+v_\alpha(p) t, p \right) - \nabla_x^i \nabla_p^j  f^\alpha_\infty(x,p) \right | \lesssim \max \left \{ t^{-n-2}, t^{j -n -2} \right \}$$
for $i,j \in \mathbb{N}_0$ with $1 \leq i+j \leq n+1$.
\end{theorem}

\begin{remark}
\label{Finf}
Throughout the proof of Theorem \ref{T2}, we will establish and utilize the functions $F^{\alpha, \delta}_\infty(p)$, which will be initially understood only as the large time asymptotic limits of the functions $F^{\alpha, \delta}(t,p)$.
However, at the end of the proof we will justify the regularity needed to assert the ultimate relationship between the limiting distributions $f^\alpha_\infty(x,p)$ and the limiting spatial moments $F^{\alpha, \delta}_\infty(p)$ given by \eqref{eq:F-al-inf}.
Hence, this relationship is not needed to establish our results. 
Only the existence of the limits $F^{\alpha, \delta}_\infty(p) = \lim_{t\to\infty} F^{\alpha, \delta}(t,p)$ will be necessary, independent of an explicit formula for the limiting functions. 
\end{remark}

\begin{remark}
The additional regularity of solutions required by our main results is unlikely to be sharp, but in order to utilize the scattering map developed within \cite{BigorgneScattering}, it is necessary to assume considerably smooth scattering states (i.e., $f^\alpha_\infty \in C_c^{m+9}(\bfR^6)$).
\end{remark}


\begin{remark}
Our methods can further be extended, without the neutrality assumption or cancellation conditions on $\mcA_{\delta,\infty}$, in order to arrive at polyhomogeneous expansions of solutions, analogous to results obtained for the Vlasov-Poisson system in \cite{Bigorgne-Ruiz, Schlue-Taylor}. In this case, the spatial support of the translated distribution function $g^\alpha$ may grow in time like $\mathcal{O} \left ( \ln(t) \right )$ which introduces additional factors in the expansion of the form $\left ( \frac{\ln(t)}{t} \right )^{k}$ rather than $t^{-k}$ for $k \in \N_0$.
\end{remark}

\subsection{Organization of the paper}
Section \ref{sec:lemmas} is dedicated to some preliminary lemmas that provide asymptotic bounds (as $t \to \infty$) on the densities $\rho$ and $j$ and fields $E$ and $B$, as well as their derivatives. 
Then, these bounds are used to prove the two main theorems in Section \ref{sec:proof}. 
A rather lengthy argument establishing asymptotic bounds for derivatives of $g^\alpha$ is postponed until Section \ref{sec:g_deriv} in order to facilitate the exposition.

\section{Lemmas}
\label{sec:lemmas}

\subsection{Preliminaries}
To begin, the spatial characteristics in the moving reference frame are given by
\begin{equation}
\label{gcharalt}
\mcY(t,\tau, x,p) = \mcXa(t,\tau, x, p) - v_\alpha \bigg (\mcPa(t, \tau, x, p) \bigg ) t
\end{equation}
with $\mcY(\tau) = x - v_\alpha(p)\tau$. Furthemore, the spatial-momentum support in the original reference frame
and the translated reference frame are given by
		\[\mcS_f^\alpha(t) = \overline{\left \{ (x,p) \in \R^6 : f^\alpha(t,x,p) \neq 0 \right \}},\]
		and 
		\[\mcS_g^\alpha(t) = \overline{\left \{ (x,p) \in \R^6 : g^\alpha(t,x,p) \neq 0 \right \}},\]
respectively. We note that $\mcS_g^\alpha(0) = \mcS_f^\alpha(0)$ for all $\alpha = 1, ..., N$ and define the maximal spatial support amongst all particles by 
		\[\mcR(t) = \max_{\alpha=1,\dots,N}\sup \{|\mcY(t,0,x,p)| : (x,p) \in \mcS_g^\alpha(0) \}.\]
As we must track the asymptotic decay properties of the electromagnetic field, we define for every for $\ell \in \mathbb{N}_0$, the quantity
\begin{align*}
	\mcK_\ell(t) &:= \max_{\alpha=1,\dots,N}\sup_{\substack{|x| \lesssim \mcR(t)\\ |p| \leq \beta} } \biggl (\left |\nabla^\ell_xE(t, x +v_\alpha(p)t) \right | + \left |\nabla^\ell_xB(t,x +v_\alpha(p)t) \right | \biggr )\\
	& \quad +  \max_{\alpha=1,\dots,N}\sup_{\substack{\tau \in [0,t]\\ (x,p) \in \mcS_f^\alpha(\tau)}}  \biggl  ( \left |\nabla^\ell_x E(t, \mcXa(t, \tau, x, p) ) \right | + \left |\nabla^\ell_xB(t, \mcXa(t, \tau, x, p)  )\right | \biggr ).
\end{align*} 
We only consider fields within the interior of the light cone to take full advantage of the decay properties therein. Therefore, the suprema above are necessarily bounded for all time.
Additionally, we must track the time-asymptotic growth rates of spatial and momentum derivatives of $g^\alpha$, which leads us to define for every $\ell\in\N$
	\[
	\mcG^\ell_p(t) = 1+\max_{\alpha=1,\dots,N}\sum_{k=1}^{\ell} \left\| \nabla_p^k g^\alpha (t) \right\|_\infty
	\]
	and
	\[\mcG^\ell_{x,p}(t) =  1+ \max_{\alpha = 1, ..., N} \sum_{\substack{|{\beta_p}| + |{\beta_x}|\leq \ell \\ |\beta_x| > 0}} \| D_p^{\beta_p} D_x^{\beta_x}  g^\alpha(t) \|_\infty,
	\]
the former of which traces the behavior of only momentum derivatives, while the latter groups together all other combinations of mixed spatial and momentum derivatives.
Note that these quantities are nested so that, in particular, $\mcG^n_p(t) \lesssim 1$ implies $\mcG^\ell_p(t) \lesssim 1$ for every $1 \leq \ell \leq n$ with the same property for $\mcG^n_{x,p}(t)$.

With these definitions, we first establish some preliminary results and fundamental estimates that will be crucial to the proof of the main theorem within the next section.
\begin{lemma}
\label{lem:v}
The symmetric, matrix-valued function $\mathbb{A}_\alpha: \bfR^3 \to \bfR^{3 \times 3}$ defined by $\mathbb{A}_\alpha(p) = \nabla v_\alpha(p)$, so that its entrywise representation is
	\[
	\left (\mathbb{A}_\alpha \right)_{ij}(p)
	=
	\frac{ (m_\alpha^2 + |p|^2) \delta_{ij} - p_i p_j }{(m_\alpha^2 + |p|^2)^{3/2}}
	\]
satisfies
$$ |\mathbb{A}_\alpha(p)| \leq (m_\alpha^2 + |p|^2)^{-1/2} \leq m_\alpha^{-1}$$
for all $p \in \bfR^3$ and
$$ \left |\det  \left (\mathbb{A}_\alpha(p) \right ) \right | = m_\alpha^2\left (m_\alpha^2 + |p|^2\right)^{-5/2}.$$
Therefore, the inverse determinant of $\mathbb{A}(p)$ has the representation
\begin{equation}
\label{Ddef}
\mcD_\alpha(p) := m_\alpha^{-2}\left (m_\alpha^2 + |p|^2\right)^{5/2}.
\end{equation}
%
Finally, the functions $v_\alpha(p)$, $v_\alpha^{-1}(q)$, $\mcD_\alpha(p)$ and all of their derivatives are $C^\infty$ and bounded on any compact subset of their respective domains.

\end{lemma}
\begin{proof}
The proof consists of straightforward calculations and is therefore omitted.
\end{proof}

\subsection{Properties of the Translated Distribution Function}

With the basic properties of characteristics determined, we introduce some notation relating to the translated distribution functions.
As mentioned in the introduction, for every $\alpha = 1, ..., N$ we let
$$g^\alpha(t,x,p) = f^\alpha(t,x+v_\alpha(p)t, p)$$
so that 
$f^\alpha(t,x, p) = g^\alpha(t,x-v_\alpha(p)t,p) $.
The spatial characteristics $\mcY(t)$ of $g^\alpha$, introduced in \eqref{gcharalt}, satisfy
\begin{equation}
\label{gchar}
 \dot{\mcY}(t) = -t \mathbb{A}_\alpha(\mcPa(t)) K^\alpha \biggl (t, \mcY(t) +  v_\alpha\left (\mcPa(t) \right ) t, \mcPa(t) \biggr).
\end{equation}
In addition, note that $\Vert g^\alpha(t) \Vert_\infty \leq \Vert f_0 \Vert_\infty$ for all $t \geq 0$ and
$$g^\alpha(t,x,p) = 0 \qquad \mathrm{for} \qquad |p| \geq \beta.$$

\begin{lemma}
\label{Xsupp}
For every $\tau \geq 0$, and $(x,p) \in \mcS_g^\alpha(\tau)$ the characteristics satisfy
$$\left | \mcY(t, \tau, x, p)  - (x -v_\alpha(p)\tau) \right | \lesssim \int_\tau^t s \mcK_0(s) ds$$
and 
$$\mcR(t) \lesssim 1 + \int_0^t s \mcK_0(s) ds.$$
In particular, if $\rho_{0,\infty} \equiv 0$, then Theorem \ref{oldT2} further implies
$$\left | \mcY(t, \tau, x, p) \right | \lesssim 1 \qquad \mathrm{and} \qquad 
\mcR(t) \lesssim 1.$$
\end{lemma}

\begin{proof}
Using \eqref{gchar} we immediately find
$$\left | \dot{\mcY}(t) \right |  \leq t |\mathbb{A}_\alpha(\mcPa(t))| |K^\alpha(t,\mcXa(t), \mcPa(t))| \lesssim t \mcK_0(t),$$
and thus
$$\left | \mcY(t, \tau, x, p)  - (x - v_\alpha(p) \tau) \right | \leq \int_\tau^t \left |  \dot{\mcY}(s, \tau, x, p)  \right | \ ds \lesssim \int_\tau^t s \mcK_0(s) \ ds$$
for fixed $\tau \geq 0$ and $(x,p) \in \bS_g(\tau)$.
Furthermore, this implies
$$\left | \mcY(t, 0, x, p) \right | \lesssim  |x| + \int_0^t s |K^\alpha(s, \mcXa(s), \mcPa(s))| \ ds \lesssim 1 + \int_0^t s \mcK_0(s) \ ds$$
for $(x,p) \in \bS_g(0)$.
The estimate on the spatial radius then follows as
$$\mcR(t) = \max_{\alpha=1,\dots,N}\sup_{(x,p) \in \bS_g(0)} \left | \mcY(t, 0, x, p) \right |  \lesssim 1 + \int_0^t s \mcK_0(s) \ ds.$$
\end{proof}

\subsection{New tools for \eqref{RVMg}}

Throughout this section, we will assume $\mcA_{0,\infty} \equiv 0$, which implies $\rho_{0,\infty} \equiv 0$.
Hence, the spatial support of $g^\alpha$ is uniformly bounded in time by Lemma \ref{Xsupp}, namely
\begin{equation}
\label{Rbound}
\mcR(t) \lesssim 1.
\end{equation}
In order to emphasize the uniform boundedness of the spatial support, we define for use in subsequent lemmas the constant
\[\mcR = \sup_{t\geq 0} \mcR(t). \] 

\begin{lemma}
\label{DF0}
	For every $\alpha = 1, ..., N$ and $k \in \N_0$, we have 
$$	 \infnorm{\nabla_p^{k} \biggl (F^{\alpha,0}(t) - F^{\alpha,0}_\infty \biggr)}  \lesssim \int_t^\infty  \sum_{j=0}^k \sum_{m=0}^j  s^{j-m}\Bigl[ s\mcK_{j-m+1}(s)\mcG_p^{m}(s)+\mcK_{j-m}(s)\mcG_p^{k-j+1}(s)\Bigr] ds$$
	under the assumption that the integral on the right side is finite.
Moreover, if  $\mcG_p^{k-1}(t)\lesssim 1$, then this reduces to 
$$\infnorm{\nabla_p^k \biggl (F^{\alpha,0}(t)-F^{\alpha,0}_\infty\biggr)}\lesssim \int_t^\infty \left\{ \sum_{m=2}^{k+1} s^m \mcK_{m}(s) +\mcK_0(s)\mcG_p^{k+1}(s)+\big(s\mcK_1(s)+\mcK_0(s)\big)\mcG_p^k(s) \right\} ds.$$
\end{lemma}
\begin{proof}
We follow the ideas contained in the proof of \cite[Lemma 3.6]{Pan-BA2025}.
Recall that
$$F^{\alpha, 0}(t, p) 
	= \int_{\R^3}  g^\alpha \left (t, y, p \right ) \ dy.$$
Because the momentum support of $g^\alpha$ is uniformly bounded, we take $|p| \leq \beta$ throughout the proof.	
Using \eqref{RVMg} we find
\begin{align*}
	\partial_t F^{\alpha, 0} &(t, p)
	=	 \int_{\R^3}  \partial_tg^\alpha \left (t, x, p \right ) dx\\
	& =	\frac{e_\alpha}{m_\alpha} \int_{\R^3}  K^\alpha(t,x+v_\alpha(p)t, p)\cdot\big( t \mathbb{A}_\alpha(p) \nabla_{x}-\nabla_{p}\big) g^\alpha(t,x,p) \ dx\\
	& =	-\frac{e_\alpha}{m_\alpha} \int_{\R^3}  \Big( t\ \mathrm{tr}\big\{\mathbb{A}_\alpha(p)\nabla_{x} K^\alpha(t,x+v_\alpha(p)t, p)\big\} g^\alpha(t,x,p)+K^\alpha(t,x+v_\alpha(p)t, p) \cdot\nabla_{p}g^\alpha(t,x,p)\Big) \ dx
\end{align*}
upon integrating by parts in $x$ in the first term.
Taking any $k$th order momentum derivative with multi-index $\eta$ so that $|\eta|=k$, we have
\begin{align*}
	\partial_tD_p^\eta F^{\alpha, 0}(t, p)
	= -\frac{e_\alpha}{m_\alpha} \int_{\R^3}  D_p^\eta \Big ( &t \ \mathrm{tr}\big\{\mathbb{A}_\alpha(p)\nabla_{x} K^\alpha(t,x+v_\alpha(p)t, p)\big\} \, g^\alpha(t,x,p) \\
	& + K^\alpha(t,x+v_\alpha(p)t, p) \cdot\nabla_{p}g^\alpha(t,x,p) \Big ) \, dx.
\end{align*}
Applying this momentum derivative to the product of terms in the integrand yields derivatives ranging up to order $k$ on each, so that
\begin{align*}
    \left |\partial_tD_p^\eta F^{\alpha, 0}(t, p) \right | \lesssim
	\int_{\R^3}     \Bigg[&
		\:t\!\! \sum_{\substack{\nu_1+\nu_2+\nu_3 \preceq \eta\\ |\nu_1|+|\nu_2|+|\nu_3|=k}} \!\!\!\sum_{i,j=1}^3 D_p^{\nu_1}D_{x_j} K_i^\alpha(t,x+v_\alpha(p)t, p) D_p^{\nu_2}g^\alpha(t,x,p) \: D_p^{\nu_3} \mathbb{A}_{\alpha}^{i,j}(p)  \\
		& + \sum_{\substack{\nu \preceq \eta\\ 0 \leq |\nu| \leq k}} \!\!\!
 \sum_{i=1}^3 D_p^\nu \biggl (K_i^\alpha(t,x+v_\alpha(p)t, p) \biggr ) D_p^{\eta - \nu} D_{p_i} g^\alpha(t,x,p) 
    \Bigg] dx.
\end{align*}
Due to the bounded spatial support of $g^\alpha$ given by \eqref{Rbound}, the integral here is taken over $|x| \leq \mcR(t) \lesssim 1$. As the derivatives of $v_\alpha(p)$ and $\mathbb{A}_\alpha(p)$ of any order are uniformly bounded for $|p| \leq \beta$ and 
$$\left |D_p^\nu \biggl (K^\alpha(t,x+v_\alpha(p)t, p) \biggr ) \right | = \left |D_p^\nu E(t,x+v_\alpha(p)t) + D_p^\nu \biggl ( v_\alpha(p) \times B(t,x+v_\alpha(p)t) \biggr ) \right  | \lesssim \sum_{\ell=0}^{|\nu|}t^\ell \mathcal{K}_\ell(t),$$ 
taking the supremum over all momenta gives
\begin{align*}
	\infnorm{\partial_t  D_p^\eta F^{\alpha,0} (t)} &\lesssim  \sum_{\substack{\nu_1+\nu_2+\nu_3 \preceq \eta\\ |\nu_1|+|\nu_2|+|\nu_3|=k}}  \sum_{b=0}^{|\nu_1|}t^{b+1} \mathcal{K}_{b+1}(t) \mcG_p^{|\nu_2|}(t) +
 \sum_{\substack{\nu \preceq \eta\\ 0 \leq |\nu| \leq k}} \sum_{b=0}^{|\nu|} t^{b}\mcK_{b}(t) \mcG_p^{|\eta - \nu|+1}(t) : = I + II.
\end{align*}
To simplify the first term, we let $|\nu_3|=k-|\nu_1|-|\nu_2|$, $|\nu_1|=j-m$ for $j\in\{0,...,k\}$, and $|\nu_2|=m\in\{0,...,j\}$ to find 
\begin{align*}
I \lesssim \sum_{j=0}^k \sum_{m=0}^j  t^{j-m+1} \mcK_{j-m+1}(t)\mcG_p^{m}(t).
\end{align*}
To simplify the second sum, recall that $|\eta| = k$ and let $|\nu| = j$, so that it yields
$$II  \lesssim \sum_{j = 0}^k \sum_{b=0}^j t^{b}\mcK_{b}(t) \mcG_p^{k-j+1}(t) = \sum_{j = 0}^k \sum_{m=0}^j t^{j-m}\mcK_{j-m}(t) \mcG_p^{k-j+1}(t)$$
upon reindexing via $m=j-b$.
Putting these together gives
\begin{align*}
	\infnorm{\partial_t D_p^\eta  F^{\alpha,0} (t)} \lesssim \sum_{j=0}^k \sum_{m=0}^j  t^{j-m}\Big( t\, \mcK_{j-m+1}(t)\mcG_p^{m}(t)+\mcK_{j-m}(t)\mcG_p^{k-j+1}(t)\Big).
\end{align*}
Hence, for $\tau\geq t$, we can write
\begin{align*}
    \infnorm{ D_p^\eta  F^{\alpha,0} (t) - D_p^\eta  F^{\alpha,0} (\tau)}=\infnorm{\int_t^\tau \partial_t D_p^\eta  F^{\alpha,0} (s)\,ds} \leq \int_t^\tau \infnorm{\partial_t D_p^\eta  F^{\alpha,0} (s)}\, ds.
\end{align*}
Thus, for integrands that decay sufficiently fast, we let $\tau\rightarrow\infty$ and find
\begin{align*}
    \infnorm{ D_p^\eta  F^{\alpha,0} (t) - D_p^\eta  F^{\alpha,0}_\infty} \lesssim \int_t^\infty  \sum_{j=0}^k \sum_{m=0}^j  s^{j-m}\Big( s\mcK_{j-m+1}(s)\mcG_p^{m}(s)+\mcK_{j-m}(s)\mcG_p^{k-j+1}(s)\Big) ds.
\end{align*}
Summing over all such derivatives with $|\eta|=k$ gives the first result for any $\alpha=1,...,N$.

Assuming $\mcG_p^{k-1}(t)\lesssim 1$   and breaking up the sums to eliminate such terms, this estimate immediately simplifies to
\begin{align*}
\infnorm{\nabla_p^{k} \biggl (F^{\alpha,0}(t) - F^{\alpha,0}_\infty \biggr)}  \lesssim  \int_t^\infty  \Biggl [&\ \sum_{j=0}^k \sum_{m=0}^{j-1}  s^{j-m+1}\mcK_{j-m+1}(s) +  s\mcK_{1}(s) + s\mcK_1(s) \mcG_p^k(s)\\
& + 
\sum_{j=2}^k \sum_{m=0}^{j}  s^{j-m} \mcK_{j-m}(s) + \mcK_0(s) \mcG_p^{k+1}(s) +   \left(\mcK_0(s) + s\mcK_1(s) \right ) \mcG_p^{k}(s)\Biggr] ds.
\end{align*}
Reindexing the remaining sums and using $\mcG_p^k(t) \geq 1$, which implies
$$\mcK_0(s) + s\mcK_1(s) \leq \left(\mcK_0(s) + s\mcK_1(s) \right ) \mcG_p^{k}(s),$$
yields the stated conclusion.
\end{proof}

\begin{lemma}
\label{DFell}
For every $\alpha = 1, \hdots, N$, $k \in \N_0$, and $|\delta|\in\mathbb{N}$, the spatial density $\nabla_p^k F^{\alpha,\delta}(t,p)$ satisfies
$$\infnorm{\nabla_p^k \rndP{F^{\alpha,\delta}(t)- F^{\alpha,\delta}_\infty}} \lesssim \int_t^\infty \sum_{j=0}^{k} \sum_{m=0}^j s^{j-m} \Big( s \mcK_{j-m+1}(s)\mcG_p^m(s) +s \mcK_{j-m}(s)\mcG_p^m(s) +\mcK_{j-m}(s)\mcG_p^{k-j+1}(s) \Big) ds$$
under the assumption that the integral on the right side is finite.
Moreover, if  $\mcG_p^{k-1}(t)\lesssim 1$, this reduces to 
\[ \infnorm{\nabla_p^k \rndP{F^{\alpha,\delta}(t)-F^{\alpha,\delta}_\infty}} \lesssim \int_t^\infty \Bigg[ \sum_{m=0}^{k} s^{m+1} \mcK_{m}(s) + s^{k+1}\mcK_{k+1}(s)+\big( \mcK_0(s)+s\mcK_1(s)\big)\mcG_p^k(s) +\mcK_0(s)\mcG_p^{k+1}(s)\Bigg] ds \]
\end{lemma}

\begin{remark}
We note that the difference between Lemmas \ref{DF0} and \ref{DFell} is one additional term appearing with the integration by parts in $x$ within the latter result.
\end{remark}

\begin{proof}
	We follow the ideas contained in the proof of \cite[Lemma 2.6]{Pan-BA2025}. First, recall from \eqref{eq:F-al} that
$$	
F^{\alpha, \delta}(t, p) 
= \int_{\R^3} (-x)^\delta g^\alpha \left (t, x, p \right )  dy.
$$
Taking a time derivative and using \eqref{RVMg}, we find
\begin{align*}
	\partial_t F^{\alpha, \delta}(t, p) &= \int_{\R^3} (-x)^\delta \partial_t g^\alpha (t, x, p ) dx\\
	&=
	\frac{e_\alpha}{m_\alpha} \int_{\R^3} (-x)^\delta \Big[ K^\alpha(t,x+v_\alpha(p)t, p)\cdot\big( t \mathbb{A}_\alpha(p) \nabla_{x}-\nabla_{p}\big) g^\alpha(t,x,p) \Big] \ dx\\
	&=
	-\frac{e_\alpha}{m_\alpha} \int_{\R^3}  \Big( t\, \mathrm{tr}\Big\{(-x)^\delta \nabla_{x} K^\alpha(t,x+v_\alpha(p)t, p) \mathbb{A}_\alpha(p) \Big\} g^\alpha(t,x,p)\\
	&\quad +t\, \mathrm{tr}\Big\{ \left(\nabla_{x}(-x)^\delta \right) K^\alpha(t,x+v_\alpha(p)t, p) \mathbb{A}_\alpha(p) \Big\}g^\alpha +(-x)^\delta K^\alpha(t,x+v_\alpha(p)t, p) \cdot\nabla_{p}g^\alpha\Big) \, dx
\end{align*}
upon integrating by parts in $x$.
Applying any $k$th order momentum derivative with multi-index $\eta$ satisfying $|\eta|=k$ gives
\begin{align*}
	\partial_t D_p^\eta F^{\alpha, \delta}(t, p) &=-\frac{e_\alpha}{m_\alpha} \int_{\R^3}  D_p^\eta \Big( t\, \mathrm{tr}\Big\{(-x)^\eta \nabla_{x} K^\alpha(t,x+v_\alpha(p)t, p) \mathbb{A}_\alpha(p) \Big\} g^\alpha(t,x,p)\\
	&\quad +t\, \mathrm{tr}\Big\{ \left(\nabla_{x}(-x)^\eta\right) K^\alpha(t,x+v_\alpha(p)t, p) \mathbb{A}_\alpha(p) \Big\}g^\alpha +(-x)^\delta K^\alpha(t,x+v_\alpha(p)t, p) \cdot\nabla_{p}g^\alpha\Big) \, dx.
\end{align*}

Taking the supremum over all $|p|\leq \beta$, we note that derivatives of $v_\alpha(p)$ and $\mathbb{A}_\alpha(p)$ are uniformly bounded, and the spatial moments $(-x)^\eta$ are bounded on the support of $g^\alpha$ due to \eqref{Rbound}. Additionally, recalling the estimate $|D_p^{\nu} K^\alpha| \lesssim \sum_{\ell=0}^{|\nu|}t^\ell \mathcal{K}_\ell(t)$ from the previous lemma, we find
\begin{align*}
    \infnorm{D_p^\eta \partial_t F^{\alpha,\delta} (t)} & \lesssim 
    \sum_{\substack{\nu_1+\nu_2+\nu_3 \preceq \eta \\ |\nu_1|+|
    \nu_2|+|\nu_3|=k}} \sum_{b=0}^{|\nu_1|} \Big ( t^{b+1}\mathcal{K}_{b+1}(t) \mcG_p^{|\nu_2|}(t) + t^{b+1} \mathcal{K}_{b}(t) \mcG_p^{|\nu_2|}(t) \Big )\\
    & \qquad  +  \sum_{\substack{\nu \preceq \eta\\ 0 \leq |\nu| \leq k}} \sum_{b=0}^{|\nu|} t^{b}\mcK_{b}(t) \mcG_p^{|\eta - \nu|+1}(t) 
\end{align*}
where the last term has been estimated as in the previous lemma.
Letting $|\nu_3|=k-|\nu_1|-|\nu_2|$, $|\nu_2|=m\in\{0,...,j\}$, and $|\nu_1|=j-m$ for $j\in\{0,...,k\}$ (i.e. $0\leq b\leq m\leq j\leq k$) within the first two sums and reindexing the last sum as in the previous lemma, this becomes
\begin{align*}
	\infnorm{\partial_t D_p^\eta  F^{\alpha,\delta} (t)} \lesssim \sum_{j=0}^{k} \sum_{m=0}^j  t^{j-m}\Big( t\,\mcK_{j-m+1}(t)\mcG_p^{m}(t)+t\,\mcK_{j-m}(t)\mcG_p^{m}(t)+\mcK_{j-m}(t)\mcG_p^{k-j+1}(t)\Big).
\end{align*}
Hence, for $\tau\geq t$, we can write
\begin{align*}
    \infnorm{ D_p^\eta  F^{\alpha,\delta} (t) - D_p^\eta  F^{\alpha,\delta} (\tau)}=\infnorm{\int_t^\tau \partial_t D_p^\eta  F^{\alpha,\ell} (s)\,ds} \leq \int_t^\tau \infnorm{\partial_t D_p^\eta  F^{\alpha,\ell} (s)}\, ds.
\end{align*}
Thus, for integrands that decay sufficiently fast, we let $\tau\rightarrow\infty$ and find
\begin{align*}
    \infnorm{ D_p^\eta  F^{\alpha,\delta} (t) - D_p^\eta  F^{\alpha,\delta}_\infty} \lesssim \int_t^\infty  \sum_{j=0}^{k} \sum_{m=0}^j  s^{j-m}\Big( s\,\mcK_{j-m+1}(s)\mcG_p^{m}(s)+s\,\mcK_{j-m}(s)\mcG_p^{m}(s)+\mcK_{j-m}(s)\mcG_p^{k-j+1}(s)\Big) ds.
\end{align*}
Summing over all such derivatives with $|\eta|=k$ gives the first result for any $\alpha=1,...,N$.
Finally, assuming $\mcG_p^{k-1}(t)\lesssim 1$  and breaking up the sums to eliminate such terms, this estimate simplifies in the same manner as for Lemma \ref{DF0}.
\end{proof}

\begin{lemma}[Refined Field Representation]
\label{Krep}
Recall the definitions $\mcU := \frac{1}{2} + \frac{\gamma}{1+\gamma} < 1$ and $\hat{u} = \frac{u}{|u|}$.
Then, for $t$ sufficiently large and $x \in \bfR^3$ satisfying $|x| \lesssim \gamma t$, the fields admit the following representation
$$	E(t,x) = E_T(t,x) + E_S(t,x) \qquad \mathrm{and} \qquad
	B(t,x) = B_T(t,x) + B_S(t,x) $$
where the $T$ component satisfies
\begin{equation}
\label{ETrep}
E_T(t,x) = \frac{1}{t^2}\int_{|u|\leq \mcU} \int_{|z| \leq \mcR} \mcH_T^E \left (t,u,z, \frac{x+ut-z}{t(1-|u|)} \right ) \, dz \, \frac{du}{|u|^2(1-|u|)^3}
\end{equation}
with 
\begin{equation}
\label{HT}
\mcH_T^E(t,u,z,q) = \kappa^E_T{\left(\hat{u},q\right)} \mcH_\rho(t(1-|u|), z, q)
\end{equation}
and
\begin{equation}
\label{kappaT}
\kappa_T^E(\omega,q)=\frac{(\omega + q)(1-|q|^2)}{(1+q\cdot\omega)^2}.
\end{equation}
Additionally, the $S$ component satisfies
\begin{equation}
\label{ESrep}
E_S(t,x) = \frac{1}{t}\int_{|u|\leq \mcU}\int_{|z| \leq \mcR}\mcH^E_S\left(t,u,z,\frac{x+ut-z}{t(1-|u|)}\right)dz\frac{du}{{|u|(1-|u|)}^3}.
\end{equation}
with
\begin{equation}
\label{HS}
\mcH^E_S(t,u,z,q) = \sum_{\alpha = 1}^N e_\alpha^2 \kappa^E_S(\hat{u}, v_\alpha^{-1}(q)) \mcD_\alpha \rndP{v_\alpha^{-1}(q)} K^\alpha(t(1-|u|),z+qt(1-|u|))g^\alpha{\big(t(1-|u|),z,v_\alpha^{-1}(q)\big)}
\end{equation}
and
\begin{equation}
\label{kappaS}
\kappa^E_S(\omega,p)=\nabla_p\cdot \left ( \frac{\omega+v_\alpha(p)}{1+v_\alpha(p)\cdot\omega} \right ).
\end{equation}
For the magnetic field $B(t,x)$, there is an analogous representation that merely replaces $\hat{u} + q$ within $\kappa^E_T$ by $\hat{u} \times q$ and $\hat{u}+v_\alpha(p)$ within $\kappa^E_S$ by $\hat{u} \times v_\alpha(p)$ in order to form the kernels $\kappa^B_T$ (whence $\mcH^B_T$) and $\kappa^B_S$ (whence $\mcH^B_S$), respectively.
\end{lemma}

\begin{proof}
Throughout, we will focus on the electric field decomposition and note that an analogous representation follows for the magnetic field via the same argument.
The main idea generally involves changing variables to the reference frame along free transport.
We begin by recalling the Glassey-Strauss decomposition of the electric field \cite{GS2}.
For any $t \geq 0$ and $x \in \bfR^3$, we may write
	\begin{equation}
	\label{eq:glassey-strauss-decomp}
	E(t,x) = E_T(t,x) + E_S(t,x) + E_\mathrm{data}(t,x)
	\end{equation}
where
\begin{equation}
\label{ET_GS}
E_T(t,x) = - \sum_{\alpha = 1}^N e_\alpha \int_{|y-x|\leq t} \int \frac{(\omega + v_\alpha(p))(1-|v_\alpha(p)|^2)}{(1+v_\alpha(p)\cdot \omega)^2}f^\alpha(t-|y-x|,y,p) \, dp\,\frac{dy}{|y-x|^2},
\end{equation}
\begin{equation}
\label{ES_GS}
	E_S(t,x)
	= \sum_{\alpha = 1}^N e_\alpha^2
	\int_{|y-x|\leq t}\int\frac{\omega+v_\alpha(p) }{1+v_\alpha(p)\cdot\omega}(E+v_\alpha\times B)(t-|y-x|,y)\cdot\nabla_pf^\alpha(t-|y-x|,y,p)dp\frac{dy}{|y-x|}
\end{equation}
and
\begin{align*}
 E_\mathrm{data}(t,x) & = \frac{1}{4\pi t^2}\int_{|y-x|=t} \biggl (E_0(y)+(y-x)\cdot\nabla E_0(y)+t\nabla\times B_0(y)\biggr )\,dS_y\\
 & \quad -\sum_{\alpha = 1}^N \frac{e_\alpha}{t} \int_{|y-x|=t}\int \left (v_\alpha(p)  f^\alpha_0(y,p)
+ \frac{\omega-(v_\alpha(p)\cdot\omega)v_\alpha(p) }{1+v_\alpha(p)\cdot\omega}f^\alpha_0(y,p) \right )\,dp\,dS_y
 \end{align*}
 where $\omega=\frac{y-x}{|y-x|}$.
Beginning from this representation, we proceed as follows.\\ 

\textbf{Step 1:}  Vanishing of $E_\mathrm{data}$\\

We first claim that the data term vanishes for sufficiently large time because $|x| \lesssim \gamma t$.
Indeed, each of the terms within the above formula involves an integral over the boundary of the light cone $\{ y: |x - y| = t\}$, and we will show that the corresponding integrands vanish on this set for $t$ large.
As we consider only $|x| \lesssim \gamma t$,  we have
$$ |y| \geq |y-x| - |x| \gtrsim  t - \gamma t = (1- \gamma) t > L$$
on the spatial set of integration by taking $t$ sufficiently large, and in particular, $t > \frac{L}{1-\gamma}$.
However, due to the assumptions on the initial data, $E_0(y)$, $B_0(y)$, and $f^\alpha_0(y,p)$ all vanish for $|y| > L$.
Therefore, the region of integration lies strictly outside of the support of any of the integrands in the above expression, and it follows that 
$E_\mathrm{data}(t,x) = 0$
for $|x| \lesssim \gamma t$.\\

\textbf{Step 2:} Representation of $E_T(t,x)$\\

Next, we recall the representation \eqref{ET_GS}.
Transforming the spatial variable $y$ to a velocity variable $u$ via $u=\frac{y-x}{t}$ and defining the prefactor
$$\kappa^E_T(\omega,q)=\frac{(\omega + q)(1-|q|^2)}{(1+q\cdot\omega)^2},$$
for any $|q| < 1$,
the integral becomes
$$    E_T(t,x)=-t \sum_{\alpha=1}^N e_\alpha  \int_{|u|\leq 1} \int \kappa^E_T{\left(\hat{u},v_\alpha(p)\right)}f^\alpha (t(1-|u|),x+ut,p) \, dp \, \frac{du}{|u|^2}.$$
%
With this, we change the reference frame via
$$f^\alpha(t,x,p)=g^\alpha(t,x-v_\alpha(p)t,p)$$
to find
\begin{align*}
    E_T(t,x)=-t \sum_{\alpha=1}^N e_\alpha\int_{|u|\leq 1} \int \kappa^E_T{\left(\hat{u},v_\alpha(p)\right)} g^\alpha{\Big(t(1-|u|),x+ut-v_\alpha(p)t(1-|u|),p\Big)} \, dp \, \frac{du}{|u|^2}.
\end{align*}
Transforming the momentum variable $p$ to a spatial variable $z$ via 
$$z=x+ut-v_\alpha(p)t(1-|u|)$$
yields
\[v_\alpha(p)=\frac{x+ut-z}{t(1-|u|)} \quad \Rightarrow \quad p=v_\alpha^{-1}\rndP{\frac{x+ut-z}{t(1-|u|)}}. \]
Therefore, upon noting the compact spatial support of $g^\alpha$, we arrive at the representation
\begin{equation}
\label{ETa_rep}
    E_T(t,x)=\frac{1}{t^2}\int_{|u|\leq 1} \int_{|z| \leq \mcR} \mcH^E_T \left (t,u,z, \frac{x+ut-z}{t(1-|u|)} \right ) \, dz \, \frac{du}{|u|^2(1-|u|)^3}
\end{equation}
where
\begin{align}
\label{HET}
\mcH^E_T(t,u,z,q) & =  \sum_{\alpha=1}^N e_\alpha \kappa^E_T{\left(\hat{u},q\right)}\mcD_\alpha \rndP{v_\alpha^{-1}(q)}  g^\alpha{\big(t(1-|u|),z,v_\alpha^{-1}(q)\big)}\\
& =  \kappa^E_T{\left(\hat{u},q\right)} \mcH_\rho(t(1-|u|), z, q) \nonumber
\end{align}
which nearly provides the stated representation of $E_T(t,x)$.\\

\textbf{Step 3:} Representation of $E_S(t,x)$\\

We recall the representation \eqref{ES_GS}.
Using the fact that the Lorentz force is divergence-free in momentum, we integrate by parts in $p$, obtaining
$$E_S(t,x) = - \sum_{\alpha=1}^N e_\alpha^2 \int_{|y-x|\leq t}\int \nabla_p\cdot\left( \frac{\omega+v_\alpha(p)}{1+v_\alpha(p) \cdot\omega}\right)K^\alpha(t-|y-x|,y,p) f^\alpha(t-|y-x|,y,p)dp\frac{dy}{|y-x|}.$$
To simplify notation, we define
$$\kappa^E_S(\omega,p)=\nabla_p\cdot \left ( \frac{\omega+v_\alpha(p)}{1+v_\alpha(p)\cdot\omega} \right ).$$
Then, we make the change of variables $u=\frac{y-x}{t}$
so that $\kappa^E_S(\omega,p)=\kappa^E_S\left(\hat{u},p\right)$ and
$$	E_S(t,x)
	=
	- t^2 \sum_{\alpha=1}^N e_\alpha^2 \int_{|u|\leq 1}\int \kappa^E_S\left(\hat{u},p\right)K^\alpha(t-t|u|,x+ut,p)f^\alpha(t(1-|u|),x+ut,p)dp \frac{du}{|u|},
$$
where $ \hat{u}=\frac{u}{|u|}$.
Proceeding as for $E_T$ by changing the reference frame using
$$f^\alpha(t,x,p)=g^\alpha(t,x-v_\alpha(p)t,p)$$
and changing variables via
$$z=x+ut-v_\alpha(p)t(1-|u|)$$
yields
\begin{equation}
\label{ESa_rep}
E_S(t,x)
	=
	\frac{1}{t}\int_{|u|\leq 1}\int_{|z| \leq \mcR}\mcH^E_S\left(t,u,z,\frac{x+ut-z}{t(1-|u|)}\right)dz\frac{du}{{|u|(1-|u|)}^3}
\end{equation}
where
$$ \mcH^E_S(t,u,z,q)= \sum_{\alpha=1}^N e_\alpha^2 \kappa^E_S{\left(\hat{u},v_\alpha^{-1}(q)\right)} \mcD_\alpha \rndP{v_\alpha^{-1}(q)} K^\alpha(t(1-|u|),z+qt(1-|u|),v_\alpha^{-1}(q))g^\alpha{\big(t(1-|u|),z,v_\alpha^{-1}(q)\big)}, $$
which nearly provides the stated representation of $E_S(t,x)$.\\

\textbf{Step 4:} Removal of Singularity\\

With the field representation close to complete, it only remains to remove the singularity within the $u$ integration appearing in both $E_T$ and $E_S$. 
To this end, we note that this region of integration can be dramatically simplified for $t$ sufficiently large due to the support of $g^\alpha$.
Because $g^\alpha$ has compact spatial and momentum support, we note that
$g^\alpha(t,z,p) = 0$ for $|z| \geq \mcR$ or $|p| \geq \beta.$
Therefore, the integrals within \eqref{ETa_rep} and \eqref{ESa_rep} are constrained by $|z| \leq \mcR$ and
$$ \left |v_\alpha^{-1}\rndP{\frac{x+ut-z}{t(1-|u|)}} \right | \leq \beta,$$
respectively.
Notice that $|p| \leq \beta$ implies $|v_\alpha(p)| \leq \gamma$. 
Thus, letting $p = v_\alpha^{-1}\rndP{\frac{x+ut-z}{t(1-|u|)}}$, we find
$$ \left |\frac{x+ut-z}{t(1-|u|)} \right | \leq \gamma.$$
or
$$|x+ut-z| \leq \gamma t(1-|u|)$$
as $|u| \leq 1$.
Because
$$|u|t - |x| - |z| \leq |x+ut-z|,$$
rearranging the above inequality implies
$$(1+\gamma)|u|t \leq |x| + |z| + \gamma t.$$
Finally, we take $t$ sufficiently large, consider only $|x| \leq \gamma t$, and use the constraint $|z| \leq \mcR$ to find
$$|u| \leq \frac{2\gamma}{1+\gamma} + \frac{\mcR}{t(1+ \gamma)}.$$
Note that $\gamma < 1$ guarantees $\frac{\gamma}{1+\gamma} < \frac{1}{2}$.
Finally, taking $t$ sufficiently large, in particular
\begin{equation}
\label{tcond}
t \geq \frac{2\mcR}{(1+\gamma) \left ( 1 -  \frac{2\gamma}{1+\gamma} \right )},
\end{equation}
gives
$$|u| \leq \frac{1}{2} + \frac{\gamma}{1+\gamma} = \mcU < 1.$$
Hence, the support of $g^\alpha(t)$ is a subset of $\{(z,u): |z| \leq \mcR, |u| \leq \mcU\}$.
Thus, for $t$ sufficiently large and $|x| \leq \gamma t$, the integrals above are equal to 
$$ E_T(t,x)=\frac{1}{t^2}\int_{|u|\leq \mcU} \int_{|z| \leq \mcR} \mcH^E_T \left (t,u,z, \frac{x+ut-z}{t(1-|u|)} \right ) \, dz \, \frac{du}{|u|^2(1-|u|)^3}$$
and
$$ E_S(t,x)=\frac{1}{t^2}\int_{|u|\leq \mcU} \int_{|z| \leq \mcR} \mcH^E_S \left (t,u,z, \frac{x+ut-z}{t(1-|u|)} \right ) \, dz \, \frac{du}{|u|(1-|u|)^3}.$$
This removes any singularity at $|u| = 1$ within these expressions.
As $\mcH^E_T(t,u,z,q)$ and $\mcH^E_S(t,u,z,q)$ vanish for $|v_\alpha^{-1}(q)| \geq \beta$, due to the compact momentum support of $g^\alpha$, the functions $\kappa^E_T$, $\kappa^E_S$, and $D_\alpha$ remain $C^\infty$ and uniformly bounded on the support of $g^\alpha(t)$.
\end{proof}

\begin{remark}
We note that the representation herein is significantly more useful, in terms of computing and estimating field derivatives, than the original Glassey-Strauss field representation \cite{GS2}. In particular, the spatial derivatives of $E$ and $B$ now clearly depend upon momentum derivatives of $g^\alpha$ and gain an extra $t^{-1}$ decay factor with each derivative.
Additionally, unlike the decomposition in \cite{GS2}, there is no need to estimate additional projections of the $T$ and $S$ operators within representations of derivatives, e.g., no $E_{TT}$, $E_{TS}$, $E_{ST}$, or $E_{SS}$ terms are necessary to estimate $\Vert \nabla^k_x E(t) \Vert_\infty$ or $\Vert \nabla^k_x B(t) \Vert_\infty$ for $k \in \mathbb{N}$.
\end{remark}


\begin{lemma}
\label{Kconv}
For any $k \in \N_0$ and $\ell \in \mathbb{N}$, we have 
\begin{align*}
	\sup_{x \in \bfR^3} \left | t^{k+\ell+2} \nabla^k_xE(t,x) - \nabla^k_qE_{\ell,\infty}\left(\frac{x}{t} \right) \right |
	\lesssim \sum_{m=0}^{\ell-1} t^{\ell-m} \Vert \nabla^k_qW^E_m (t) \Vert_\infty 
+  \frac{1}{t^{1+k}} \left (\sum_{j=0}^k t^j \sup_{|u| \leq \mcU} \mcK_j (t (1-|u|)) \right ) \mcG_p^k (t)\\
	 + \sum_{j=0}^{k+\ell} \max_{\substack{\alpha = 1,..., N \\ |\delta| = \ell}} \sup_{|u| \leq \mcU} \left \Vert \nabla^j_q \biggl (F^{\alpha,\delta} (t(1-|u|)) - F_\infty^{\alpha,\delta} \biggr )\right \Vert_\infty
	+ t^{-1} \mcG_p^{k+\ell+1}(t).
\end{align*}
In the case $\ell = 0$, this reduces to
\begin{align*}
\sup_{x \in \bfR^3} \left | t^{k+2} \nabla^k_xE(t,x) - \nabla^k_qE_{0,\infty}\left(\frac{x}{t} \right) \right | 
& \lesssim \frac{1}{t^{1+k}} \left (\sum_{j=0}^k t^j \sup_{|u| \leq \mcU} \mcK_j (t (1-|u|)) \right ) \mcG_p^k (t)\\
& \quad  +  \sum_{j=0}^k \max_{\alpha = 1,..., N} \sup_{|u| \leq \mcU} \left \Vert \nabla^j_q \biggl (F^{\alpha, 0} (t(1-|u|)) - F_\infty^{\alpha, 0} \biggr )\right \Vert_\infty + t^{-1} \mcG_p^{k+1}(t)
\end{align*}
for any $k \in \N_0$.
The convergence estimates for $B(t,x)$ are identical.
\end{lemma}

\begin{proof}
We begin by invoking the field representation of $E(t,x)$, along with an analogous decomposition for $B(t,x)$, provided by Lemma \ref{Krep}, namely
$$E(t,x) = E_T(t,x) + E_S(t,x)$$
where 
$$E_T(t,x) = \frac{1}{t^2}\int_{|u|\leq \mcU} \int_{|z|\leq \mcR} \mcH^E_T \left (t,u,z, \frac{x+ut-z}{t(1-|u|)} \right ) \, dz \, \frac{du}{|u|^2(1-|u|)^3}$$
with
$$ \mcH^E_T(t,u,z,q)= \kappa^E_T{\left(\hat{u},q\right)} \ \mcH_\rho\rndP{t(1-|u|), z, q},$$
where $\kappa^E_T$ and $\mcH_\rho$ are given by \eqref{kappaT} and \eqref{Hrho},
and
$$E_S(t,x) = \frac{1}{t}\int_{|u|\leq \mcU}\int_{|z|\leq \mcR} \mcH^E_S\left(t,u,z,\frac{x+ut-z}{t(1-|u|)}\right)dz\frac{du}{{|u|(1-|u|)}^3}$$
with
$$\mcH^E_S(t,u,z,q) = \sum_{\alpha = 1}^N e_\alpha^2 \kappa^E_S(\hat{u}, v_\alpha^{-1}(q)) \mcD_\alpha \rndP{v_\alpha^{-1}(q)} K^\alpha(t(1-|u|),z+qt(1-|u|), v_\alpha^{-1}(q))g^\alpha{\big(t(1-|u|),z,v_\alpha^{-1}(q)\big)}$$
where $\kappa^E_S$ is given by \eqref{kappaS}.
Here,  $K^\alpha(t,x,p) = E(t,x)+ v_\alpha(p) \times B(t,x)$.
Throughout the proof we will use the fact that $\kappa^E_S$, $\kappa^E_T$ and $\mcD_\alpha$ are infinitely smooth with derivatives that are uniformly bounded on the support of $g^\alpha$.

Taking spatial derivatives of order $k \in \mathbb{N}_0$ within this representation, we find
$$D_x^kE(t,x) = D_x^kE_T(t,x) + D_x^kE_S(t,x)$$
where
\begin{equation}
\label{DxKET}
D_x^kE_T(t,x)  = \frac{1}{t^{2+k}}\int_{|u|\leq \mcU} \int_{|z| \leq \mcR} D^k_q\mcH^E_T \left (t,u,z, \frac{x+ut-z}{t(1-|u|)} \right ) \, dz \, \frac{du}{|u|^2(1-|u|)^{3+k}},
\end{equation}
and due to the compact support of $g^\alpha$,
$$\left |D_x^kE_S(t,x) \right | \lesssim \frac{1}{t^{1+k}}\int_{|u|\leq \mcU} \sup_{\substack{|z| \leq \mcR \\|q| \leq \gamma}} \left |D^k_q\mcH^E_S(t,u,z, q) \right | \, \frac{du}{|u|(1-|u|)^{3+k}}.$$

\textbf{Step 1:} Estimate of $D_x^kE_S(t,x)$\\

First, we estimate the $D_x^kE_S$ term, which should ultimately decay faster than the convergence rate of the corresponding $D_x^kE_T$ term.
In particular, within $D_x^kE_S$, we wish to bound derivatives of $\mcH_S^E(t,u,z,q)$.
We first note that the arguments of the field terms remain inside the set $|x| \lesssim \gamma t$.
Considering the field term $D_x^jK^\alpha(t(1-|u|),z+qt(1-|u|),v_\alpha^{-1}(q))$ for $j \in \mathbb{N}_0$ we claim that the magnitude of the spatial argument is bounded by $\gamma$ multiplied by its time argument, meaning
$$\left | z+qt(1-|u|) \right | \leq \gamma t(1-|u|),$$
for $t$ sufficiently large. As we restrict $|u|$ away from 1 for sufficiently large time via \eqref{tcond}, this is equivalently written
$$\left | \frac{z}{t(1-|u|)} + q \right | \leq \gamma .$$
Beginning with the uniform bound on the spatial support of $g^\alpha$ and $|u| \leq \mcU < 1$, we find
$$\left | \frac{z}{t(1-|u|)}\right | \leq t^{-1} \frac{\mcR}{1-\mcU}.$$
The bound on the momentum support of $g^\alpha$, namely $|v_\alpha^{-1}(q)| \leq \beta$, further yields
$|q| \leq \frac{\beta}{\sqrt{1+ \beta^2}}$.
Hence, assuming \eqref{tcond} and further taking times satisfying 
$$t \geq \frac{\mcR}{1-\mcU} \left ( \gamma - \frac{\beta}{\sqrt{1+ \beta^2}}\right )^{-1}$$
gives
$$\begin{aligned}
\left | \frac{z}{t(1-|u|)} + q \right | & \leq \gamma.
\end{aligned}$$
Therefore, upon rearranging we have
$$\left | z+qt(1-|u|) \right | \lesssim \gamma t(1-|u|).$$
Applying this to the fields allows us to estimate for each $j\in\N_0$ and $\alpha=1,\dots,N$,
\begin{align*}
\biggl |D_x^jK^\alpha(t(1-|u|),z+qt(1-|u|),v_\alpha^{-1}(q)) \biggr | &\lesssim \sup_{|y| \lesssim \gamma t(1-|u|)} \biggl ( |D_x^jE(t(1-|u|),y)| + |D_x^jB(t(1-|u|),y)| \biggr )\\
& \lesssim  \mcK_j \left ( t (1 - |u|) \right ). 
\end{align*}
The analogous argument holds for $p$ derivatives of the force, as well, so that for $\alpha=1,\dots,N$ and $j\in\N$,
$$\biggl |D_p^jK^\alpha(t(1-|u|),z+qt(1-|u|),v_\alpha^{-1}(q)) \biggr | \lesssim \sup_{|y| \lesssim \gamma t(1-|u|)}  |B(t(1-|u|),y)| \lesssim  \mcK_0 \left ( t (1 - |u|) \right )$$
and for $j=0$
$$\biggl |K^\alpha(t(1-|u|),z+qt(1-|u|),v_\alpha^{-1}(q)) \biggr |  \lesssim  \mcK_0 \left ( t (1 - |u|) \right ).$$
Thus, taking derivatives of $\mcH^E_S$ and using the increasing nature of $\mcG_p^j(t)$ in $j$, we find
$$\left|D_x^kE_S(t,x)\right| \lesssim
	\frac{1}{t^{1+k}}\int_{|u|\leq \mcU} \left (\sum_{j=0}^k [t(1-|u|)]^j \mcK_j (t(1-|u|)) \right ) \mcG_p^k (t(1-|u|)) \frac{du}{{|u|(1-|u|)}^{3+k}}
$$
and, as $\mcG_p^k(t)$ is increasing in $t$,  this yields
	\begin{equation}
	\label{ES_estimate}
\left|D_x^kE_S(t,x)\right|
\lesssim \frac{1}{t^{1+k}} \left (\sum_{j=0}^k t^j \sup_{|u| \leq \mcU} \mcK_j (t (1-|u|)) \right ) \mcG_p^k (t) .
	\end{equation}

\textbf{Step 2:} Field Limit and Convergence of $E_T(t,x)$\\

Next, we return to \eqref{DxKET} and consider the difference between the tangential component and the limiting electric field.
First, the kernel of the electric field can be expanded within its last argument so that
\begin{align*}
\mcH^E_T \left (t,u,z, \frac{x+ut-z}{t(1-|u|)} \right ) &= 
	\sum_{m=0}^\ell\sum_{|\delta|=m}\frac{(-z)^\delta}{(t(1-|u|))^m\delta!}D^\delta_q\mcH^E_T\left(t,u,z, \frac{x+ut}{t(1-|u|)} \right)\\
	& \qquad +
	\sum_{|\delta|=\ell+1}\frac{(-z)^\delta}{(t(1-|u|))^{\ell+1}\delta!}D^\delta_q\mcH^E_T\left(t,u,z,\frac{x+ut - \theta z}{t(1-|u|)}  \right )
\end{align*}
for some $\theta \in[0,1]$.
Moreover, by \eqref{eq:F-al}, \eqref{eq:A-l}, \eqref{eq:W-l}, and \eqref{HT} the $z$ integral of the first term on the right side can be rewritten in terms of the function $W_m^E(t,u,q)$ 
so that
\begin{align*}
\sum_{|\delta|=m}\frac{1}{(1-|u|)^m\delta!} & \int_{|z| \leq \mcR} (-z)^\delta D^\delta_q\mcH_T^E \left (t,u,z, q \right)  \ dz\\
&  =  \sum_{|\delta|=m}\frac{1}{(1-|u|)^m\delta!} D^\delta_q \left (\sum_{\alpha = 1}^N e_\alpha \kappa^E_T(\hat{u},q) \mcD_\alpha\rndP{v_\alpha^{-1}(q)} \int (-z)^\delta g^\alpha{\big(t(1-|u|),z,v_\alpha^{-1}(q)\big)} \ dz \right )\\
& = \sum_{|\delta|=m} \frac{1}{(1-|u|)^m\delta!}  D^\delta_q \biggl (\kappa^E_T(\hat{u},q) \mcA_\delta\big(t(1-|u|), q\big) \biggr  )\\
& = W_m^E(t,u,q).
\end{align*}
Upon taking derivatives to match the expression within \eqref{DxKET}, we insert this in the above expansion to find
\begin{align*}
\int_{|z| \leq \mcR} D^k_q\mcH^E_T \left (t,u,z, \frac{x+ut-z}{t(1-|u|)} \right ) \ dz &= 
	\sum_{m=0}^\ell t^{-m} D^k_qW_m^E\left(t,u, \frac{x+ut}{t(1-|u|)} \right)\\
	& \quad +
	\int_{|z| \leq \mcR} \sum_{|\delta|=\ell+1}\frac{(-z)^\delta}{(t(1-|u|))^{\ell+1}\delta!}D^k_qD^\delta_q\mcH^E_T\left(t,u,z,\frac{x+ut - \theta z}{t(1-|u|)}  \right ) \ dz.
\end{align*}
Further, we note that from \eqref{eq:E-l}, we have
$$\int_{|u|\leq \mcU} D_q^kW^E_m\rndP{t, u, \frac{x + ut}{t(1 - |u|)}} \, \frac{du}{|u|^2\rndP{1-|u|}^{3+k}} = t^k D_x^kE_m(t,x).$$
Hence, inserting the above expansion within \eqref{DxKET}, subtracting the limiting field, and estimating the remaining error term yields
\begin{align*}
     \left|t^{2+k+\ell}D_x^kE_T(t,x)- D_q^kE_{\ell,\infty}\rndP{\frac{x}{t}}\right|
     &\lesssim \sum_{m=0}^{\ell-1} t^{k+\ell-m} \left | D^k_xE_m(t,x) \right |\\
     &\quad + \left | t^kD^k_xE_\ell(t,x) - D_q^kE_{\ell,\infty}\rndP{\frac{x}{t}}\right|\\
& \quad +
	t^{-1}  \int_{|u|\leq \mcU} \sup_{\substack{|z| \leq \mcR, \\|q| \leq \gamma}} \left |\nabla^{k+\ell+1}_q\mcH^E_T(t,u,z, q) \right | \, \frac{du}{|u|^2(1-|u|)^{4+k+\ell}}\\
    &=I + II + III.
\end{align*}

To estimate $I$, we immediately find from \eqref{eq:E-l}
$$ I \leq \sum_{m=0}^{\ell-1} t^{\ell-m} \int_{|u| \leq \mcU} \left | D_q^k W_m^E \left (t, u, \frac{x + ut}{t(1-|u|)} \right ) \right | \frac{du}{|u|^2(1-|u|)^{3+k}}
\lesssim \sum_{m=0}^{\ell-1} t^{\ell-m} \Vert \nabla^k_qW^E_m (t) \Vert_\infty.$$
Similarly, using \eqref{HET} and the fact that $\mcG_p^k(t)$ is increasing in $t$, estimating $III$ yields
$$III \lesssim t^{-1} \mcG_p^{k+\ell+1}(t).$$

Finally, to estimate $II$, we use the structure of the approximating and limiting fields in \eqref{eq:E-l} and \eqref{eq:E-l-inf}, respectively, as well as \eqref{eq:A-l}, \eqref{eq:W-l}, \eqref{eq:A-l-inf}, and \eqref{eq:W-l-inf} to estimate the difference of the corresponding kernels, which yields 
\begin{align*}
II & \leq \int_{|u|\leq \mcU}  \left | D^k_qW^E_\ell\rndP{t, u, \frac{x + ut}{t(1 - |u|)}} - D^k_q W^E_{\ell,\infty}\rndP{u, \frac{x + ut}{t(1 - |u|)}}  \right | \, \frac{du}{|u|^2\rndP{1-|u|}^{3+k}}\\
& \leq \int_{|u|\leq \mcU}  \left | D^k_q  \sum_{|\delta| = \ell} \frac{1}{\delta!} D_q^\delta \left [\kappa_T^E(\hat{u},q) \biggl ( \mcA_{\delta} (t(1-|u|),q) -  \mcA_{\delta,\infty} (q) \biggr ) \right ]  \biggr |_{q = \frac{\frac{x}{t} + u}{1 - |u|}}  \right | \, \frac{du}{|u|^2\rndP{1-|u|}^{3+k+\ell}}\\
& \leq \sum_{j=0}^{k+\ell} \max_{\substack{\alpha = 1,..., N \\ |\delta| = \ell}} \sup_{|u| \leq \mcU} \left \Vert \nabla^j_q \biggl (F^{\alpha,\delta} (t(1-|u|)) - F_\infty^{\alpha,\delta} \biggr )\right \Vert_\infty.
\end{align*}
%

Combining the estimates for $I$, $II$, and $III$, and including the estimate \eqref{ES_estimate} for $E_S(t,x)$ completes the proof.
The analogous result holds for $B(t,x)$, rather than $E(t,x)$, using identical methods.
\end{proof}


\begin{lemma}
\label{lem:rhojderiv-l}
	For any $k,\ell \in \N_0$, we have
	\[\sup_{x \in \bfR^3} \left | t^k \nabla^k_x \rho_\ell(t,x) - \nabla^k_p \rho_{\ell,\infty}\left(\frac{x}{t} \right) \right | \lesssim \sum_{j=0}^{k+\ell}\max_{\substack{\alpha=1,\hdots,N \\ |\delta| = \ell}}\left\| \nabla_q^j \biggl (F^{\alpha,\delta}(t)- F_\infty^{\alpha,\delta} \biggr )\right\|_\infty,\]
	\[\sup_{x \in \bfR^3} \left | t^k \nabla^k_x j_\ell(t,x) - \nabla^k_p j_{\ell,\infty}\left(\frac{x}{t} \right) \right | \lesssim \sum_{j=0}^{k+\ell}\max_{\substack{\alpha=1,\hdots,N \\ |\delta| = \ell}}\left\| \nabla_q^j \biggl (F^{\alpha,\delta}(t)- F_\infty^{\alpha,\delta} \biggr )\right\|_\infty,\]
		\[\sup_{\substack{|q| \leq \gamma \\ |u| \leq \mcU}} \left | \nabla^k_q \biggl (W^E_\ell(t,u,q) - W^E_{\ell,\infty}(u,q) \biggr ) \right | \lesssim \sum_{j=0}^{k+\ell}\max_{\substack{\alpha=1,\hdots,N \\ |\delta| = \ell}} \sup_{|u| \leq \mcU} \left\| \nabla_q^j \biggl (F^{\alpha,\delta}(t(1-|u|))- F_\infty^{\alpha,\delta} \biggr )\right\|_\infty,\]
	and
	\[\sup_{\substack{|q| \leq \gamma \\ |u| \leq \mcU}} \left | \nabla^k_q \biggl (W^B_\ell(t,u,q) - W^B_{\ell,\infty}(u,q) \biggr ) \right | \lesssim \sum_{j=0}^{k+\ell}\max_{\substack{\alpha=1,\hdots,N \\ |\delta| = \ell}} \sup_{|u| \leq \mcU} \left\| \nabla_q^j \biggl (F^{\alpha,\delta}(t(1-|u|))- F_\infty^{\alpha,\delta} \biggr )\right\|_\infty.\]
\end{lemma}

\begin{proof}
	Taking any $k$th order spatial derivative of $\rho_\ell(t,x)$ using the multi-index $|\eta|=k$ in \eqref{eq:rho-l} and substituting \eqref{eq:A-l} gives
	\[D_x^\eta\rho_\ell(t,x) = 4\pi t^{-k} \sum_{|\delta| = \ell}  \frac{1}{\delta!} D^\delta_qD_q^\eta \biggl [ \sum_{\alpha=1}^N e_\alpha \mcD_\alpha (v_\alpha^{-1}(q)  )F^{\alpha,\delta} \left (t, v_\alpha^{-1}(q) \right) \biggr ] \biggr |_{ q=\frac{x}{t} } \]
	Thus, subtracting the proposed limit given by derivatives of \eqref{eq:rho-l-inf}, using the boundedness of $\mcD_\alpha(p)$, $v_\alpha^{-1}(q)$, and their derivatives, and estimating yields
$$
\left| t^k D_x^\eta \rho_\ell(t,x)-D_p^\eta\rho_{\ell,\infty}\left(\frac{x}{t}\right)\right| 
\lesssim \sum_{j=0}^{k+\ell}\max_{\substack{\alpha=1,\hdots,N \\ |\delta| = \ell}}\left\| \nabla_q^j \biggl (F^{\alpha,\delta}(t)- F_\infty^{\alpha,\delta} \biggr )\right\|_\infty.
$$
upon summing over all such derivatives and taking the supremum over all $x$. This provides the stated result for $\rho_\ell(t,x)$, and the proofs for $j_\ell(t,x)$, $W_\ell^E(t,x)$, $W_\ell^B(t,x)$ are analogous upon using the boundedness of $\kappa_T^E(\hat{u}, q)$, $\kappa_T^B(\hat{u}, q)$, and $\sup_{|u| \leq \mcU} (1-|u|)^{-k}$ for any $k \in \N_0$ in the latter, field cases.
\end{proof}


\begin{lemma}
	\label{rhoconv}
	For any $\ell \in \mathbb{N}$, we have
	$$\sup_{x \in \bfR^3} \left | t^{\ell+3} \rho(t,x) - \rho_{\ell,\infty} \left(\frac{x}{t} \right) \right | \lesssim \sum_{m=0}^{\ell-1}  t^{\ell - m}\left \|  \rho_m(t) \right \|_\infty +  \sum_{j=0}^\ell\max_{\substack{\alpha = 1, ..., N \\ |\delta| = \ell}} \left\| \nabla_q^j \biggl (F^{\alpha, \delta}(t)  - F^{\alpha, \delta}_\infty \biggr ) \right\|_\infty + t^{-1} \mcG_p^{\ell +1}(t) $$
	and
	$$\sup_{x \in \bfR^3} \left | t^{\ell+3} j(t,x) - j_{\ell,\infty} \left(\frac{x}{t} \right) \right | \lesssim \sum_{m=0}^{\ell-1}  t^{\ell - m}\left \|  j_m(t) \right \|_\infty +  \sum_{j=0}^\ell\max_{\substack{\alpha = 1, ..., N \\ |\delta| = \ell}} \left\| \nabla_q^j \biggl (F^{\alpha, \delta}(t)  - F^{\alpha, \delta}_\infty \biggr ) \right\|_\infty + t^{-1} \mcG_p^{\ell +1}(t) .$$
\end{lemma}
\begin{proof}
		We begin by fixing $\ell \in \mathbb{N}$. Then, utilizing the representation of the charge density within \eqref{rhop}, we have
		\[\rho(t,x)=4\pi t^{-3}  \int \mcH_\rho\left( t,y,\frac{x-y}{t}\right) dy.\]
		Using the Taylor expansion in \eqref{rho_exp}, separating the order $\ell$ term, and estimating derivatives of the remainder using \eqref{Hrho}, we find
		\[t^3\rho(t,x)=\sum_{m=0}^{\ell-1} t^{-m}\rho_m(t,x) + t^{-\ell} \rho_\ell(t,x) + t^{-(\ell+1)}\mcE (t) \]
		where 
		$$|\mcE(t)| = \sup_{x \in \mathbb{R}} \sup_{\theta \in [0,1]} \left | 4\pi  \sum_{|\delta| = \ell+1}  \frac{1}{\delta!} \int_{|y| \leq \mcR} (-y)^\delta D^\delta_q\,\mcH_\rho\left(t,y,\frac{x - \theta y}{t} \right) dy \right | \lesssim \mcG_p^{\ell+1}(t).$$ 
		Using \eqref{eq:A-l}, \eqref{eq:rho-l}, \eqref{eq:A-l-inf}, and \eqref{eq:rho-l-inf} the $\ell$th order approximation of $\rho(t,x)$ satisfies 
		\begin{equation}
		\label{rho-ell}
			\left| \rho_\ell(t,x)-\rho_{\ell,\infty}\left(\frac{x}{t}\right) \right| \lesssim \sum_{j=0}^\ell\max_{\substack{\alpha = 1, ..., N \\ |\delta| = \ell}} \left\| \nabla_q^j \biggl (F^{\alpha, \delta}(t)  - F^{\alpha, \delta}_\infty \biggr ) \right\|_\infty.
		\end{equation}
		Thus, expanding $\rho(t,x)$ as above, we find
		\begin{align*}
			\left| t^{\ell+3}\rho(t,x) - \rho_{\ell,\infty}\left(\frac{x}{t}\right)\right| &\leq \left| \sum_{m=0}^{\ell-1} t^{\ell-m}\rho_m(t,x) + \rho_{\ell}(t,x)-\rho_{\ell,\infty}\left(\frac{x}{t}\right) \right| + t^{-1} \left |\mcE(t) \right | \\
			&\lesssim  \sum_{m=0}^{\ell-1} t^{\ell-m} \left\|\rho_m(t)\right\|_\infty + \left| \rho_\ell(t,x)-\rho_{\ell,\infty}\left(\frac{x}{t}\right) \right| + t^{-1} \mcG_p^{\ell+1}(t)\\
		\end{align*}
which, upon using \eqref{rho-ell}, yields the stated result. 
The proof for $j(t,x)$ is identical upon expanding $\mcH_j$.
\end{proof}

\section{Proofs of Theorems \ref{T1} and \ref{T2}}
\label{sec:proof}
We first focus on proving Theorem \ref{T2}, which will essentially imply Theorem \ref{T1} subject to the construction of suitable scattering states.
\subsection{Base Case}
We begin by assuming the condition
\begin{equation} 
\tag{$\mcA_0$}
\label{A0}
    \mcA_{0,\infty}\equiv 0
\end{equation} so that $\rho_{0,\infty} = j_{0,\infty} = W_{0,\infty}^E = W_{0,\infty}^B = E_{0,\infty} = B_{0,\infty}\equiv 0$, as well.
As a reminder, this condition implies $\mcR(t) \lesssim 1$ due to Lemma \ref{Xsupp}.
We wish to prove that \eqref{A0} implies 
\begin{equation}\tag{$\mathcal{Q}_0$} \label{eq:Q0}
    \left\{
\begin{aligned}
     \displaystyle
        \sup_{x\in\R^3} \left| t^{4}\rho(t,x)-\rho_{1,\infty}\rndP{\frac{x}{t}} \right| &\lesssim t^{-1},\\ \displaystyle
        \sup_{x\in\R^3} \left| t^{4} j(t,x)-j_{1,\infty}\rndP{\frac{x}{t}} \right| &\lesssim t^{-1},\\ \displaystyle
        \sup_{x\in\R^3} \left| t^{3} (E,B)(t,x)- (E,B)_{1,\infty}\rndP{\frac{x}{t}} \right| &\lesssim t^{-1} \\ \displaystyle
        \sup_{x\in\R^3} \left| t^{3} \nabla_x(E,B)(t,x)- \nabla_q(E,B)_{0,\infty}\rndP{\frac{x}{t}} \right| &\lesssim t^{-1}\\ \displaystyle
        \sup_{x\in\R^3} \left| t^{4}\nabla_x^2 (E,B)(t,x)- \nabla_q^2(E,B)_{0,\infty}\rndP{\frac{x}{t}} \right| &\lesssim t^{-1}\ln^2(t),\\ \displaystyle
   	\mcG_{x,p}^2 (t)+\mcG_p^2(t) &\lesssim 1.
\end{aligned}\right .
\end{equation}
To achieve this goal, we outline the following procedure. 
First, in order to obtain \eqref{eq:Q0} we need to establish the optimal decay rates of the fields and their derivatives, namely
$\mcK_0(t) \lesssim t^{-3}$ and $\mcK_1(t) \lesssim t^{-4}$.
These rates depend explicitly on uniform-in-time bounds for $\mcG_p^1(t)$ and $\mcG_p^2(t)$, respectively. 
Next, the convergence of the spatial average and its derivatives must be refined using the updated field decay so that
\begin{align*}
    \infnorm{F^{\alpha,0}(t)-F^{\alpha,0}_\infty}&\lesssim t^{-2},\\
    \infnorm{F^{\alpha,\delta}(t)-F^{\alpha,\delta}_\infty}&\lesssim t^{-1},\\
    \infnorm{\nabla_p\rndP{F^{\alpha,\delta}(t)-F^{\alpha,\delta}_\infty}}&\lesssim t^{-1}
\end{align*}
for $|\delta| = 1$. 
With this, improved decay rates of the densities and field kernels can be achieved
\[ \infnorm{\rho_0(t)}, \infnorm{j_0(t)},\infnorm{W^E_0(t)},\infnorm{W^B_0(t)}\lesssim t^{-2}.\]
Finally, these improvements will give rise to the stated convergence of the densities and fields to their limits.\\

\textbf{Step 1:}  Optimal decay rate of fields\\

We begin this process by obtaining the best possible decay rate for $\mcK_0(t)$. From Theorem \ref{oldT2}, we deduce the improved decay rate
\begin{equation}
\label{K0_Base}
\mcK_0(t)\lesssim t^{-3}\ln^8(t),
\end{equation}
and Lemma \ref{lem:glassey-strauss} implies initial rates of decay on derivatives of the fields, namely
\begin{equation}
\label{K1K2_Base}
\mcK_1(t) \lesssim t^{-3}\ln(t) \qquad \mathrm{and} \qquad \mcK_2(t) \lesssim t^{-4}\ln(t).
\end{equation}
With this, Lemma \ref{Xsupp} guarantees that the particle distribution $g^\alpha$ is compactly supported in the spatial domain as $\mcR(t) \lesssim 1$.
Furthermore, $\mcG^0(t)\lesssim 1$ due to the boundedness of the initial data. 
Invoking Lemmas \ref{Dpg} and \ref{D2g} with \eqref{K0_Base} and \eqref{K1K2_Base} then provides initial decay rates on momentum derivatives of $g^\alpha$, namely
\begin{equation}
\label{G1G2_Base}
\mcG_p^1(t) \lesssim \ln^2(t) \qquad \mathrm{and} \qquad  \mcG_p^2(t) \lesssim \ln^4(t).
\end{equation}
Applying Lemma \ref{DF0} with $k=0$ and  $k=1$ and inserting \eqref{K0_Base} and \eqref{K1K2_Base}, we further establish initial rates for the convergence of the spatial average so that
for every $\alpha = 1, ..., N$
$$ \infnorm{F^{\alpha,0}(t)-F^{\alpha,0}_\infty} \lesssim \int_t^\infty \rndP{s\mcK_1(s)\mcG^0(s)+\mcK_0(s)\mcG_p^1(s)}\,ds \lesssim \int_t^\infty s^{-2}\ln(s) \,ds
     \lesssim t^{-1}\ln(t)$$
and
\begin{align*}
    \infnorm{\nabla_p \rndP{F^{\alpha,0}(t)-F^{\alpha,0}_\infty}}&\lesssim\int_t^\infty \Bigl[ s \mcK_1(s)\mcG^0(s) + \mcK_0(s) \mcG_p^2(s) + s^2 \mcK_2(s) \mcG^0(s) + s\mcK_1(s) \mcG_p^1(s) + \mcK_0(s) \mcG_p^1(s) \Bigr] ds\\
     &\lesssim \int_t^\infty s^{-2}\ln^3(s) \,ds \lesssim t^{-1}\ln^3(t).
\end{align*}
With this, we improve the decay rates of field derivatives by invoking Lemma \ref{Kconv} with $k=1$, $\ell = 0$, noting $E_{0,\infty} \equiv 0$, and using \eqref{K0_Base}-\eqref{G1G2_Base} to find 
\begin{align*}
    t^3\sup_{x\in\R^3}\tildes{\nabla_xE(t,x)} &\lesssim t^{-2}\Big(\sup_{|u| \leq \mcU} \mcK_0\rndP{t(1-|u|)}+t \sup_{|u| \leq \mcU}\mcK_1\rndP{t(1-|u|)}\Big)\mcG_p^1(t)+\infnorm{F^{\alpha,0}(t)-F^{\alpha,0}_\infty}\\&\qquad\qquad + \sup_{|u| \leq \mcU}\infnorm{\nabla_p \rndP{F^{\alpha,0}(t(1-|u|))-F^{\alpha,0}_\infty}} +t^{-1}\mcG_p^2(t)\\
    &\lesssim t^{-1}\ln^4(t).
\end{align*}
As the decay of the magnetic field is identical, we have 
\begin{equation}
\label{K1_Base_prelim}
\mcK_1(t)\lesssim t^{-4}\ln^4(t).
\end{equation}
This provides a refined estimate of $\mcG_p^1(t)$ by invoking Lemma \ref{Dpg} and using \eqref{K0_Base} and \eqref{K1_Base_prelim} to find
$$ \mcG_p^1(t)\lesssim 1+\int_1^t\rndP{s^2\mcK_1(s)+s\mcK_0(s)}\,ds \lesssim 1+\int_1^t s^{-2} \ln^{8}(s)\,ds \lesssim 1.$$
Similarly, the estimates of $\mcG_{x,p}^2(t)$ and $\mcG_p^2(t)$ can be refined by invoking Lemma \ref{D2g} to find
$$\mcG_{x,p}^2(t) \lesssim 1 + \int_1^t s^2 \mcK_1(s)\,ds \lesssim 1$$
and
$$\mcG_p^2(t) \lesssim 1 + \int_1^t \left (s^3 \mcK_2(s) + \ln^2(s) \left [s^2 \mcK_1(s) + s \mcK_0(s) \right ] \right )\,ds  \lesssim 1+\int_1^t s^{-1} \ln(s)\,ds \lesssim \ln^2(t).$$
Of course, this improvement further provides a minor update on the estimate of $\mcK_1(t)$ from the above computations so that
\begin{equation}
\label{K1_Base}
\mcK_1(t)\lesssim t^{-4}\ln^2(t).
\end{equation}
The uniform-in-time bound on first derivatives of $g^\alpha$ now further refines the convergence estimate of the spatial average by applying Lemma \ref{DF0} with $k=0$ and using \eqref{K1_Base} so that
$$ \infnorm{F^{\alpha,0}(t)-F^{\alpha,0}_\infty} \lesssim \int_t^\infty \rndP{s\mcK_1(s)\mcG^0(s)+\mcK_0(s)\mcG_p^1(s)}\,ds \lesssim \int_t^\infty s^{-3}\ln^8(s) \,ds \\
     \lesssim t^{-2}\ln^8(t).$$
The uniform bound on $\mcG_p^1(t)$ also allows us to deduce the optimal field decay rate using Lemma \ref{Kconv} with $k=\ell=0$, which yields
$$ t^2\sup_{x\in\R^3}\tildes{E(t,x)}\lesssim t^{-1}\sup_{|u| \leq \mcU} \mcK_0\rndP{t(1-|u|)}\mcG^0(t)+\sup_{|u| \leq \mcU}\infnorm{F^{\alpha,0}(t(1-|u|))-F^{\alpha,0}_\infty} +t^{-1}\mcG_p^1(t) \lesssim t^{-1}.$$
As the decay rate of the magnetic field is identical, we have 
\begin{equation}
\label{K0_Base1}
\mcK_0(t)\lesssim t^{-3}.\\
\end{equation}

\textbf{Step 2:}   Optimal decay rate and convergence of field derivatives\\

Now, we can apply the same steps from above to refine the decay rate of the next order field derivatives, ultimately resulting in $\mcK_{1}(t)\lesssim t^{-4}$.
First, we find an initial decay estimate on third-order field derivatives by invoking Lemma \ref{lem:glassey-strauss} with $k=3$ to find
\begin{equation}
\label{K3_Base}
\mcK_{3}(t)\lesssim t^{-5}\ln(t).
\end{equation}

With this, we construct initial estimates on 2nd- and 3rd-order derivatives of $g^\alpha$ with a lemma whose proof is postponed until a later section.
\begin{lemma}
\label{Dng0}
Assume \eqref{K1K2_Base} and \eqref{K1_Base}-\eqref{K3_Base}, as well as the known estimates 
$$\mcG_{x,p}^1(t) + \mcG_p^1(t) \lesssim 1, \qquad \mcG_{x,p}^2(t) \lesssim 1,  \qquad \mathrm{and}  \qquad  \mcG_p^2(t) \lesssim \ln^2(t).$$
Then, we have
$\mcG_{x,p}^{3}(t) \lesssim \ln^2(t)$
and
\begin{equation}
\label{G3_Base}
\mcG_p^{3}(t)\lesssim \ln^2(t).
\end{equation}
\end{lemma}
\begin{proof}
The proof is contained in Section \ref{sec:g_deriv} below.
\end{proof}
With these initial estimates, we find a preliminary convergence estimate for second-order derivatives of $F^{\alpha,0}(t,p)$ by invoking Lemma \ref{DF0} with $k=2$ to find
$$\infnorm{\nabla_p^{2}\rndP{F^{\alpha,0}(t)-F_\infty^{\alpha,0}}}\lesssim \int_t^\infty s^{-2} \ln(s) ds \lesssim t^{-1}\ln(t)$$
for every $\alpha = 1, ..., N$.
This result and \eqref{G3_Base} allow us to update the decay estimate on $\mcK_2(t)$ using Lemma \ref{Kconv} with $k=2$ and $\ell=0$ so that
\begin{align*}
 t^4\sup_{x\in\R^3}\tildes{\nabla^2_xE(t,x)} & \lesssim \frac{1}{t^3} \left (\sum_{j=0}^2 t^j \sup_{|u| \leq \mcU} \mcK_j (t (1-|u|)) \right ) \mcG_p^2 (t)\\
& \quad  +  \sum_{j=0}^2 \max_{\alpha = 1,..., N} \sup_{|u| \leq \mcU} \left \Vert \nabla^j_q \biggl (F^{\alpha, 0} (t(1-|u|)) - F_\infty^{\alpha, 0} \biggr )\right \Vert_\infty + t^{-1} \mcG_p^{3}(t)\\
& \lesssim t^{-1}\ln^2(t).
\end{align*}
As the same result holds for derivatives of the magnetic field, we have the improved decay bound
\begin{equation}
\label{K2_Base1}
\mcK_{2}(t)\lesssim t^{-5}\ln^2(t)
\end{equation}
as well as the convergence estimate
\begin{equation}
\label{QEB2_final}
\sup_{x\in\R^3}\tildes{t^4\nabla^2_x(E,B)(t,x) - \nabla^2_q (E,B)_{0,\infty}\left (\frac{x}{t} \right)}  \lesssim t^{-1} \ln^2(t).
\end{equation}

Now that $\mcK_2(t)$ is known to decay faster, we apply Lemma \ref{D2g} to deduce the uniform-in-time bound 
$$\mcG_p^2(t) \lesssim 1 + \int_1^t \left (s^3 \mcK_2(s) + \ln^2(s) \left [s^2 \mcK_1(s) + s \mcK_0(s) \right ] \right )\,ds  \lesssim 1+\int_1^t s^{-2} \ln^4(s)\,ds \lesssim 1.$$
Finally, \eqref{K2_Base1} and the improvement to $\mcG_p^{2}(t)$ allow us to update the convergence estimate of $\nabla_pF^{\alpha,0}$ using Lemma \ref{DF0} with $k=1$, which yields
$$\infnorm{\nabla_p\rndP{F^{\alpha,0}(t)-F_\infty^{\alpha,0}}}\lesssim \int_t^\infty s^{-3} \ln^2(s) \ ds \lesssim t^{-2}\ln^2(t).$$
 Next, applying these new decay rates to Lemma \ref{Kconv} with $k=1$ and $\ell=0$, we find
\begin{align*}
    \sup_{x\in\R^3}\tildes{t^3\nabla_xE(t,x) - \nabla_q E_{0,\infty} \left ( \frac{x}{t} \right )} &\lesssim t^{-2}\Big(\mcK_0\rndP{t(1-|u|)}+t\mcK_1\rndP{t(1-|u|)}\Big)\mcG_p^1(t)+\infnorm{F^{\alpha,0}(t(1-|u|))-F^{\alpha,0}_\infty}\\&\qquad\qquad +\infnorm{\nabla_p \rndP{F^{\alpha,0}(t(1-|u|))-F^{\alpha,0}_\infty}} +t^{-1}\mcG_p^2(t)\\
    &\lesssim t^{-1},
\end{align*}
and as the decay rate for the magnetic field is identical, we achieve the optimal decay result $\mcK_1(t)\lesssim t^{-4}$ along with the convergence estimate
\begin{equation}
\label{QEB1}
\sup_{x\in\R^3} \left| t^{3} \nabla_x(E,B)(t,x)- \nabla_q(E,B)_{0,\infty}\rndP{\frac{x}{t}} \right| \lesssim t^{-1}.\\
\end{equation}

%

\textbf{Step 3:}  Convergence of densities and fields\\

Now that the estimates for the fields and their first-order derivatives have been determined, we can update the rate of convergence for $F^{\alpha,0}(t)$ using Lemma \ref{DF0} with $k=0$ so that
$$ \infnorm{F^{\alpha,0}(t)-F^{\alpha,0}_\infty} \lesssim \int_t^\infty \Big(s\mcK_1(s)\mcG^0(s)+\mcK_0(s)\mcG_p^1(s)\Big)\,ds \lesssim t^{-2},$$
which gives 
$$\infnorm{\rho_0(t)},\infnorm{j_0(t)},\infnorm{W_0^B(t)},\infnorm{W_0^E(t)}\lesssim t^{-2}$$ 
upon invoking Lemma \ref{lem:rhojderiv-l} with $k=\ell=0$ and noting that each of the respective limits vanishes.
We then use Lemma \ref{DFell} with $k=0$ and $k=1$, as well as, $|\delta|=1$ to find
$$\infnorm{F^{\alpha,\delta}(t)-F^{\alpha,\delta}_\infty}\lesssim \int_t^\infty \Big(s\mcK_1(s)\mcG^0(s)+s\mcK_0(s)\mcG^0(s)+\mcK_0(s)\mcG_p^1(s)\Big)\,ds \lesssim t^{-1}$$
and
\begin{align*}
    \infnorm{\nabla_p\rndP{F^{\alpha,\delta}(t)-F^{\alpha,\delta}_\infty}}&\lesssim \int_t^\infty \Big\{\Big(s\mcK_1(s)\mcG^0(s)+s\mcK_0(s)\mcG^0(s)+\mcK_0(s)\mcG_p^2(s)\Big)\\&\qquad\qquad+s\Big(s\mcK_2(s)\mcG^0(s)+s\mcK_1(s)\mcG^0(s)+\mcK_1(s)\mcG_p^1(s)\Big)\\&\qquad\qquad+\Big(s\mcK_1(s)\mcG_p^1(s)+s\mcK_0(s)\mcG_p^1(s)+\mcK_0(s)\mcG_p^1(s)\Big)\Big\}\,ds\\
    &\lesssim t^{-1}.
\end{align*}

Finally, we apply Lemma \ref{rhoconv} with $\ell=1$, which gives
\begin{align*}
\sup_{x\in\R^3} \left| t^4\rho(t,x)-\rho_{1,\infty}\rndP{\frac{x}{t}} \right| &\lesssim t\infnorm{\rho_0(t)}\!+\infnorm{F^{\alpha,\delta}(t)-F^{\alpha,\delta}_\infty}\!+\infnorm{\nabla_p\rndP{F^{\alpha,\delta}(t)-F^{\alpha,\delta}_\infty}}\!+t^{-1}\mcG_p^2(t)\\
&\lesssim t^{-1}.
\end{align*}
and analogously, \[\sup_{x\in\R^3} \left| t^4 j(t,x)-j_{1,\infty}\rndP{\frac{x}{t}} \right| \lesssim t^{-1},\]
and Lemma \ref{Kconv} with $k=0$, $\ell=1$ so that
\begin{align*}
    \sup_{x\in\R^3} \left| t^3E(t,x)-E_{1,\infty}\rndP{\frac{x}{t}} \right| &\lesssim t\infnorm{W^E_0(t)}+t^{-1}\mcK_0(t(1-|u|))\mcG^0(t)+\infnorm{F^{\alpha,\delta}(t(1-|u|))-F^{\alpha,\delta}_\infty}\\&\qquad\qquad +\infnorm{\nabla_p\rndP{F^{\alpha,\delta}(t(1-|u|))-F^{\alpha,\delta}_\infty}}+t^{-1}\mcG_p^2(t)\\
    &\lesssim t^{-1}.
\end{align*}
As the convergence of $B$ is the same, we have 
\[\sup_{x\in\R^3} \left| t^3 (E,B)(t,x)-(E,B)_{1,\infty}\rndP{\frac{x}{t}} \right| \lesssim t^{-1}\]
and thus, combining these results with \eqref{QEB2_final}, \eqref{QEB1}, and the uniform bounds on $\mcG_{x,p}^2(t)$ and $\mcG_p^2(t)$, the conditions within \eqref{eq:Q0} have been established, completing the base case.

\subsection{Induction Step}
To address the inductive step we fix $n \geq 1$ and assume $\mcP_{n-1} \Rightarrow \mcQ_{n-1}$ and $\mcP_n$, so that we must then establish $\mcQ_n$.
Note that, as $\mcP_n$ implies $\mcP_{n-1}$, we can immediately deduce $\mcQ_{n-1}$, as well.
Thus, by the induction hypothesis
\begin{equation} 
\label{eq:Pn}
\tag{$\mcP_n$} 
\forall \ell=0,...,n \quad \mathrm{and} \quad |\delta|=\ell, \qquad \mcA_{\delta,\infty}\equiv 0
\end{equation}
we further deduce
\begin{equation} 
\label{eq:Qnm1}\tag{$\mcQ_{n-1}$}
    \left\{
\begin{aligned}
    \sup_{x\in\R^3} \left| t^{n+3}\rho(t,x)-\rho_{n,\infty}\rndP{\frac{x}{t}} \right| &\lesssim t^{-1},\\
    \sup_{x\in\R^3} \left| t^{n+3} j(t,x)-j_{n,\infty}\rndP{\frac{x}{t}} \right| &\lesssim t^{-1},\\
    \sup_{x\in\R^3} \left| t^{n+2} (E,B)(t,x)- (E,B)_{n,\infty}\rndP{\frac{x}{t}} \right| &\lesssim t^{-1},\\
    \sup_{x\in\R^3} \left| t^{n+2} \nabla_x^j(E,B)(t,x)- \nabla_q^j(E,B)_{n-j,\infty}\rndP{\frac{x}{t}} \right| &\lesssim t^{-1} \qquad \forall j=1,...,n,\\
    \sup_{x\in\R^3} \left| t^{n+3}\nabla_x^{n+1} (E,B)(t,x)- \nabla_q^{n+1}(E,B)_{0,\infty}\rndP{\frac{x}{t}} \right| &\lesssim t^{-1}\ln^2(t),\\
    \mcG_{x,p}^{n+1}(t)+\mcG_p^{n+1}(t) &\lesssim 1.
    \end{aligned}\right .
\end{equation}
Because \eqref{eq:Pn} implies $\rho_{n,\infty}, j_{n,\infty} \equiv 0$ and $(E,B)_{\ell,\infty} \equiv 0$ for all $\ell = 0, ..., n$, the above time-dependent quantities enjoy faster rates of decay, and thus we have the following set of estimates stemming directly from the induction hypothesis
\begin{equation}
	\label{IH}
	\left \{ \,
	\begin{aligned}
		\| \rho(t) \|_\infty & \lesssim t^{-n-4}, & \\
		\| j(t) \|_\infty & \lesssim t^{-n-4}, & \\
		\mcK_j(t) & \lesssim t^{-n-3} \qquad & \forall j \in \{0, ..., n\}, & \\
		\mcK_{n+1}(t) & \lesssim t^{-n-4}\ln^2(t),\\
		\mcG^{n+1}_{x,p}(t) + \mcG^{n+1}_{p}(t) & \lesssim 1.
	\end{aligned}
	\right .
\end{equation}

With this, we must establish
\begin{equation}\tag{$\mathcal{Q}_n$} \label{eq:Qn}
    \left\{
\begin{aligned}
     \displaystyle
        \sup_{x\in\R^3} \left| t^{n+4}\rho(t,x)-\rho_{n+1,\infty}\rndP{\frac{x}{t}} \right| &\lesssim t^{-1},\\ \displaystyle
        \sup_{x\in\R^3} \left| t^{n+4} j(t,x)-j_{n+1,\infty}\rndP{\frac{x}{t}} \right| &\lesssim t^{-1},\\ \displaystyle
        \sup_{x\in\R^3} \left| t^{n+3} (E,B)(t,x)- (E,B)_{n+1,\infty}\rndP{\frac{x}{t}} \right| &\lesssim t^{-1},\\ \displaystyle
        \sup_{x\in\R^3} \left| t^{n+3} \nabla_x^j(E,B)(t,x)- \nabla_q^j(E,B)_{n-j+1,\infty}\rndP{\frac{x}{t}} \right| &\lesssim t^{-1} \qquad \forall j=1,...,n+1,\\ \displaystyle
    	\sup_{x\in\R^3} \left| t^{n+4}\nabla_x^{n+2} (E,B)(t,x)- \nabla_q^{n+2}(E,B)_{0,\infty}\rndP{\frac{x}{t}} \right| &\lesssim t^{-1}\ln^2(t),\\ \displaystyle
    	\mcG_{x,p}^{n+2} (t)+\mcG_p^{n+2}(t) &\lesssim 1.
\end{aligned}\right .
\end{equation}

For clarity, we separate the inductive step into several smaller steps.\\

\textbf{Step 1:}  Preliminary estimates of highest order derivatives\\

We begin, as in the base case, by gathering preliminary estimates of $n+2$ and $n+3$ order derivatives.
In particular, Lemma \ref{lem:glassey-strauss} implies initial rates of decay on these derivatives of the fields, namely
\begin{equation}
\label{Knp2Knp3_initial}
\mcK_{n+2}(t)\lesssim t^{-n-4}\ln(t) \qquad \mathrm{and} \qquad \mcK_{n+3}(t)\lesssim t^{-n-5}\ln(t).
\end{equation}

Next, we state the following lemma, which generalizes Lemma \ref{Dng0} for the induction step.
\begin{lemma}
\label{Dng}
Let $n \in \N$ be given and assume \eqref{Knp2Knp3_initial} and \eqref{eq:Pn}, which implies \eqref{IH}.
Then, we have
\begin{equation}
\label{DngDnE}
\mcG_p^{n+2}(t) \lesssim 1 + \int_1^t s^{n+3} \mcK_{n+2}(s) \ ds,
\end{equation}
$$\mcG_{x,p}^{n+2}(t) \lesssim 1, \qquad  \mcG_{x,p}^{n+3}(t) \lesssim \ln^2(t), \qquad \mathrm{and} \qquad  \mcG_p^{n+3}(t)\lesssim \ln^2(t)$$ 
and, as a preliminary estimate,
$$\mcG_p^{n+2}(t)\lesssim \ln^2(t).$$
\end{lemma}
\begin{proof}
The proof is contained in Section \ref{sec:g_deriv} below.
\end{proof} 
With this, we have
\begin{equation}
\label{Gnp23}
\mcG_{x,p}^{n+2}(t)\lesssim 1, \qquad \mcG_p^{n+2}(t)\lesssim \ln^2(t), \qquad \mathrm{and} \qquad \mcG_p^{n+3}(t)\lesssim \ln^2(t)
\end{equation}
and use them to estimate spatial averages and their derivatives.
In this direction, we use Lemma \ref{DF0} 
with $k=0,...,n-1$
along with the uniform-in-time bounds on $\mcG_p^{n+1}(t)$ to find
\begin{equation}
\label{DkF0}
 \infnorm{\nabla^k_p\rndP{F^{\alpha,0}(t)-F^{\alpha,0}_\infty}} \lesssim \int_t^\infty \left ( \sum_{m=1}^{k+1} s^m\mcK_m(s) +\mcK_0(s)\right )\,ds \lesssim \int_t^\infty \Big\{ s^{k-n-2} + s^{-n-3}\Big\}\,ds
\lesssim t^{k-n-1}
\end{equation}
for $k=0,...,n-1$ and $\alpha=1,...,N$.
Additionally, invoking Lemma \ref{DF0} with $k=n$ yields
\begin{align*}
    \infnorm{\nabla^n_p\rndP{F^{\alpha,0}(t)-F^{\alpha,0}_\infty}} &\lesssim \int_t^\infty \bigg\{s^{n+1}\mcK_{n+1}(s)+ \sum_{m=0}^{n} s^m\mcK_m(s) \bigg\}\,ds
    \lesssim \int_t^\infty \Big\{ s^{n+1}s^{-n-4}\ln^{2}(s)+ s^ns^{-n-3} \Big\}\,ds\\
    &\lesssim t^{-2}\ln^{2}(t).
\end{align*}
In order to estimate the $(n+1)$st derivative, we separate the dependence on $\mcG_p^{n+2}(t)$ and again use Lemma \ref{DF0} with $k=n+1$ along with \eqref{Knp2Knp3_initial} so that 
\begin{align*}
    \infnorm{\nabla^{n+1}_p\rndP{F^{\alpha,0}(t)\!-\!F^{\alpha,0}_\infty}}
    &\!\lesssim\! \int_t^\infty\! \bigg\{ s^{n+2}\mcK_{n+2}(s)\!+\! s^{n+1}\mcK_{n+1}(s)\!+\!\sum_{m=1}^{n} s^m\mcK_m(s) \!+\!\mcK_0(s)\mcG_p^{n+2}(s) \bigg\}ds\\
    &\lesssim \int_t^\infty \Big\{{s^{-2}\ln(s)}+  {s^{-3}\ln^{2}(s)} +{s^{-3}} +{s^{-n-3}\ln^{2}(s)}\Big\}\,ds\\
    &\lesssim t^{-1}\ln(t).
\end{align*}
Finally, we apply Lemma \ref{DF0} with $k=n+2$ and separate the terms with $\mcG_p^{n+2}(t)$ and $\mcG_p^{n+3}(t)$ as they grow in time.
Similarly, we separate the $\mcK_{n+1}(t)$, $\mcK_{n+2}(t)$, and $\mcK_{n+3}(t)$ terms so their distinct decay rates may be applied individually. 
As the terms with the largest derivatives currently possess the most primitive estimates, they dominate within the integral. In particular, we find
\begin{align*}
    \infnorm{\nabla_p^{n+2}\rndP{F^{\alpha,0}(t)-F_\infty^{\alpha,0}}} &\!=\! \int_t^\infty \!\bigg\{s^{n+3}\mcK_{n+3}(s)+s^{n+2}\mcK_{n+2}(s)+s^{n+1}\mcK_{n+1}(s) + \!\sum_{m=1}^{n} s^{m}\mcK_{m}(s) \\
    &\qquad + \mcK_0(s)\mcG_p^{n+3}(s)+\big(s\mcK_1(s)+\mcK_0(s)\big)\mcG_p^{n+2}(s)\bigg\}\,ds\\
    &\lesssim \int_t^\infty \Big\{s^{n+3}s^{-n-5}\ln(s)+s^{n+2}s^{-n-4}\ln(s) + s^{n+1}s^{-n-4}\ln^{2}(s)\\
    &\qquad+s^{n}s^{-n-3} + s^{-n-3}\ln^{2}(s)+\big(s^{-n-2}+s^{-n-3}\big)\ln^{2}(s)\Big\}ds\\
    &\lesssim t^{-1}\ln(t).
\end{align*}

Equipped with these rates, we may apply Lemma \eqref{Kconv} with $k=n+2$ and $\ell=0$ to obtain a faster decay estimate of $\mcK_{n+2}(t)$.
In particular, we split each of the sums into four terms so that the $n+2,\,n+1$, $n$, and remaining $j=1,...,n-1$ terms can be estimated individually, which provides
\begin{align*}
    \sup_{x\in\R^3}\tildes{t^{n+4}\nabla_x^{n+2}E(t,x)-\nabla_q^{n+2}E_{0,\infty}\left (\frac{x}{t} \right) } & \lesssim t^{-n-3}\bigg(\sum_{j=0}^{n+2}t^j\sup_{|u|\leq \mathcal{U}}\mcK_j\rndP{t(1-|u|)}\bigg)\mcG_p^{n+2}(t)\\
     & +\sum_{j=0}^{n+2} \sup_{|u|\leq \mathcal{U}} \infnorm{\nabla_p^j \rndP{F^{\alpha,0}(t(1-|u|))-F^{\alpha,0}_\infty}} +t^{-1}\mcG_p^{n+3}(t)\\
         &\lesssim t^{-n-3}\Big(t^{n+2}t^{-n-4}\ln(t)+t^{n+1}t^{-n-4}\ln^{2}(t)+\sum_{j=0}^{n}t^jt^{-n-3}\Big)\ln^{2}(t)\\ 
    &\qquad \qquad +\Big(t^{-1}\ln(t)+t^{-1}\ln(t) +t^{-2}\ln^{2}(t)+\sum_{j=0}^{n-1}t^{j-n-1}\Big)+t^{-1}\ln^{2}(t)\\
    &\lesssim t^{-1}\ln^{2}(t).
\end{align*}
As the decay rate for $\nabla_x^{n+2}B(t,x)$ is the same, we have 
\begin{equation}
\label{Knp2FINAL}
\mcK_{n+2}(t)\lesssim t^{-n-5}\ln^{2}(t).
\end{equation} 

\textbf{Step 2:}  Refine decay rates of $(n+1)$st order field derivatives\\

The improved decay rate for $\mcK_{n+2}(t)$ immediately implies a uniform bound on the corresponding derivative of $g^\alpha$ via \eqref{DngDnE}, as 
\begin{equation}
\label{Gnp2}
\mcG_p^{n+2}(t) \lesssim 1+\int_1^t s^{n+3}\mcK_{n+2}(s) ds  \lesssim 1+\int_1^t s^{-2}\ln^2(s) \ ds \lesssim 1.
\end{equation}
This allows us to update the rate of convergence of the $(n+1)$st derivative of the spatial average and ensures that all derivatives of the spatial average up to order $n+1$ converge faster than $t^{-1}$. Using \eqref{IH}, \eqref{Knp2FINAL}, and \eqref{Gnp2} within Lemma \ref{DF0} with $k=n+1$ gives
\begin{align*}
    \infnorm{\nabla_p^{n+1}\rndP{F^{\alpha,0}(t)-F_\infty^{\alpha,0}}} &\lesssim \int_t^\infty \bigg\{ s^{n+2}\mcK_{n+2}(s) +s^{n+1}\mcK_{n+1}(s) +\sum_{m=0}^{n} s^{m}\mcK_{m}(s) \bigg\} \,ds\\
    &\lesssim \int_t^\infty \Big\{ s^{-3}\ln^{2}(s)+s^{-3}\ln^{2}(s) + s^{-3} \Big\} \,ds\\
    &\lesssim t^{-2}\ln^{2}(t).
\end{align*}
We now obtain the convergence estimate for $(n+1)$st order field derivatives and update the estimate of $\mcK_{n+1}(t)$ by invoking Lemma \ref{Kconv} with $k=n+1$ and $\ell=0$ and using \eqref{IH}, \eqref{DkF0}, and \eqref{Gnp2} so that
\begin{align*}
    \sup_{x\in\R^3}\tildes{t^{n+3}\nabla_x^{n+1}E(t,x)-\nabla_q^{n+1}E_{0,\infty} \left (\frac{x}{t} \right)} &\lesssim t^{-n-2}\bigg(\sum_{j=0}^{n+1}t^j\mcK_j\rndP{t(1-|u|)}\bigg)\mcG_p^{n+1}(t)\\
    &\qquad +\sum_{j=0}^{n+1}\sup_{|u| \leq \mcU}\infnorm{\nabla_p^j \rndP{F^{\alpha,0}(t(1-|u|))-F^{\alpha,0}_\infty}} +t^{-1}\mcG_p^{n+2}(t)\\
    &\lesssim t^{-n-2}\Big(t^{n+1}t^{-n-4}\ln^{2}(t)+\sum_{j=0}^{n}t^jt^{-n-3}\Big)\\
    &\qquad + \left ( t^{-2}\ln^{2}(t) +t^{-2}\ln^{2}(t) + \sum_{j=0}^{n-1}t^{j-n-1} \right )+t^{-1}\\
    &\lesssim t^{-1}.
\end{align*}
As the convergence estimate for $B$ is the same, we have 
\begin{equation}
\label{Knp1FINAL}
\sup_{x\in\R^3} \left| t^{n+3} \nabla_x^{n+1}(E,B)(t,x)- \nabla_q^{n+1}(E,B)_{0,\infty}\rndP{\frac{x}{t}} \right| \lesssim t^{-1}
\end{equation}
as in \eqref{eq:Qn}
and 
\begin{equation}
\label{Knp1}
\mcK_{n+1}(t)\lesssim t^{-n-4}.
\end{equation} 

\vspace{0.1in}

\textbf{Step 3:}  Convergence estimates for $j$th field derivatives for all $j=0,...,n$\\

To establish the convergence of these field derivatives, we take an iterative approach and do so in descending order. 
In particular, defining \eqref{eq:Rk} for a fixed $j=0,..., n$ to consist of the estimates
    \begin{equation} 
    \tag{$\mathcal{R}_j$} 
    \label{eq:Rk}
    \left \{
            \begin{aligned}
                \infnorm{\nabla^j_q\rndP{F^{\alpha,0}(t)-F^{\alpha,0}_\infty}} &\lesssim t^{j-n-2}\\
                \max_{0< |\delta|\leq n+1-j}\infnorm{\nabla_p^{j+1}\rndP{F^{\alpha,\delta}(t)-F_\infty^{\alpha,\delta}}} &\lesssim t^{j-n-1}\\
                \infnorm{\nabla_q^j W_{\ell}^E(t)},\infnorm{\nabla_q^j W_{\ell}^B(t)} &\lesssim t^{j+\ell-n-2} \quad \forall \ell\in\{0,...,n-j\}\\
       \sup_{x\in\bfR^3} \left | t^{n+3} \nabla_x^j(E,B)(t,x) - \nabla^j_q(E,B)_{n+1-j, \infty} \left (\frac{x}{t}\right ) \right | & \lesssim t^{-1},\\
            \end{aligned}
            \right.
        \end{equation}
we will first prove that \eqref{eq:Rk} holds with $j=n$ and then show that for every $j \in\{0,\dots,n-1\}$, if $(\mathcal{R}_i)$ is true for every $i\in\{j+1, \dots, n\}$, then \eqref{eq:Rk} holds.
With this, we iteratively establish \eqref{eq:Rk} for every $j \in\{0,\dots, n\}$, which provides the respective field and field derivative convergence estimates stated within \eqref{eq:Qn}.
Throughout, we make use of the uniform bounds $\mcG_p^{n+2}(t) \lesssim 1$ to simplify our applications of Lemmas \ref{DF0} and \ref{DFell}.

We first show that the estimates within \eqref{eq:Rk} hold for $j=n$. 
Applying Lemma \ref{DF0} with $k=n$ and using \eqref{Knp1} yields
\begin{equation}
\label{DpFn0}
\begin{gathered}
    \infnorm{\nabla_p^{n}\rndP{F^{\alpha,0}(t)-F_\infty^{\alpha,0}}}  \lesssim \int_t^\infty \left\{\sum_{m=0}^{n} s^m\mcK_{m}(s) + s^{n+1}\mcK_{n+1}(s) \right\}\,ds\\
    \lesssim \int_t^\infty \left\{ s^ns^{-n-3} + s^{n+1}s^{-n-4} \right\}\,ds
    \lesssim t^{-2}.
\end{gathered}
\end{equation}
Then using Lemma \ref{DFell} with $k=n+1$ and $| \delta | =1$ along with \eqref{IH}, \eqref{Knp2FINAL}, and \eqref{Knp1}, we find
\begin{align*}
    \infnorm{\nabla_p^{n+1}\rndP{F^{\alpha,\delta}(t)-F_\infty^{\alpha,\delta}}} 
    &\lesssim \int_t^\infty\!\! \bigg\{s^{n+2}\mcK_{n+2}(s)+\sum_{m=0}^{n} s^{m}(s\!+\!1)\mcK_{m}(s)+s^{n+1}(s\!+\!2)\mcK_{n+1}(s)\bigg\}ds\\
    &\lesssim \int_t^\infty \!\! \Big\{s^{n+2}s^{-n-5}\ln^{2}(s)+ s^{n}(s\!+\!1) s^{-n-3}  +s^{n+1}(s\!+\!2) s^{-n-4} \Big\} ds\\
    &\lesssim t^{-1}.
\end{align*}
Similarly, we have for any $i=0, ..., n$ with $|\delta| \in \mathbb{N}$
\begin{equation}
\label{DkFd}
\infnorm{\nabla_p^{i}\rndP{F^{\alpha,\delta}(t)-F_\infty^{\alpha,\delta}}} \lesssim \int_t^\infty \left ( \sum_{m=0}^{i} s^{m+1}\mcK_{m}(s) + s^{i+1}\mcK_{i+1}(s) + s^{-n-2} \right ) ds \lesssim t^{i-n-1}
\end{equation}
by using Lemma \ref{DFell} with $k=i$.
Moreover, using \eqref{DkF0} and \eqref{DpFn0} within Lemma \ref{lem:rhojderiv-l} with $k=n$ and $\ell=0$ gives 
\begin{align*}
    \infnorm{\nabla_q^n \rndP{W_0^E(t)-W_{0,\infty}^E}} &\lesssim \sum_{j=0}^n \sup_{|u|\leq \mathcal{U}} \infnorm{\nabla_p^j \rndP{F^{\alpha,0}\rndP{t(1-|u|)}-F^{\alpha,0}_\infty}} \\
    &\lesssim \sum_{j=0}^{n-1}  t^{j-n-1}+t^{-2} \lesssim t^{-2} .
\end{align*}
As $W_{0,\infty}^E\equiv W_{0,\infty}^B\equiv 0$ and the estimate of $\nabla_q^n W^B_0(t)$ is identical, we find 
$$\infnorm{\nabla_q^n W_0^E(t)},\infnorm{\nabla_q^n W_0^B(t)}\lesssim t^{-2}.$$
Finally, using each of these within Lemma \ref{Kconv} with $k=n$ and $\ell=1$ we find
\begin{align*}
       \sup_{x\in\bfR^3} \left | t^{n+3} \nabla_x^nE(t,x) - \nabla^n_qE_{1, \infty} \left (\frac{x}{t}\right ) \right |
      & \lesssim t \infnorm{\nabla_q^n W_0^E(t)}+t^{-n-1}\bigg(\sum_{j=0}^{n}t^j\mcK_j\rndP{t(1-|u|)}\bigg)\mcG_p^{n}(t)\\ &\qquad+\sum_{j=0}^{n+1}\max_{|\delta|=1}\infnorm{\nabla_p^j \rndP{F^{\alpha,\delta}(t(1-|u|))-F^{\alpha,\delta}_\infty}} +t^{-1}\mcG_p^{n+2}(t)\\
    &\lesssim t^{-1}+t^{-n-1}\bigg(\sum_{j=0}^{n}t^j t^{-n-3}\bigg)+t^{-1} +\sum_{j=0}^{n}t^{j-n-1}+t^{-1}\\
    &\lesssim t^{-1}.
\end{align*}
As the estimate for $\nabla_x^nB(t,x)$ is the same, we have established $(\mathcal{R}_n)$.
Moreover, as $(E,B)_{n+1-k, \infty} \equiv 0$, this last estimate implies $\mcK_n(t)\lesssim t^{-n-4}$.\\

Next, we let $0 \leq j \leq n-1$ be given and assume $(\mathcal{R}_i)$ is true for all $i=j+1,...,n$ in order to show \eqref{eq:Rk}.
Because the field limits vanish, the estimates within  $(\mathcal{R}_i)$ collapse to provide
$$\mcK_i(t) \lesssim t^{-n-4}$$
for every $i = j+1,...,n$.
We begin by using this and \eqref{IH} and applying Lemma \ref{DF0} with $k=j$ to find
\begin{align*}
    \infnorm{\nabla_p^{j}\rndP{F^{\alpha,0}(t)-F_\infty^{\alpha,0}}}
    &\lesssim \int_t^\infty \left (\sum_{m=0}^{j} s^m\mcK_{m}(s) +s^{j+1}\mcK_{j+1}(s) \right )\,ds\\
    &\lesssim \int_t^\infty \bigg\{\sum_{m=0}^{j}s^ms^{-n-3} +s^{j+1}s^{-n-4}\bigg\}\,ds\\
    &\lesssim t^{j-n-2}.
\end{align*}
Next, we apply Lemma \ref{DFell} with $k=j+1$ taking $0 < |\delta| \leq n+1-j$ so that
\begin{align*}
    \infnorm{\nabla_p^{j+1}\rndP{F^{\alpha,\delta}(t)-F_\infty^{\alpha,\delta}}}  &\lesssim \int_t^\infty \left\{ \sum_{m=0}^{j} s^{m}(s+1)\mcK_{m}(s)+ s^{j+1}(s+2)\mcK_{j+1}(s)+s^{j+2}\mcK_{j+2}(s)\right\} ds\\
    &\lesssim \int_t^\infty \left\{ s^{j}(s+1)s^{-n-3} + s^{j+1}(s+2)s^{-n-4}+s^{j+2}s^{-n-4}\right\}\, ds\\
    &\lesssim t^{j-n-1}.
\end{align*}
To estimate the field kernels, we use Lemma \ref{lem:rhojderiv-l} with $k=j$ and $\ell=0,...,n-j$ while splitting into two cases, namely $\ell=0$ and the remaining $\ell = 1, ..., n-j$, as only the latter depends upon the $|\delta| \neq 0 $ moments of the spatial average.
In the $\ell=0$ case we use \eqref{DkF0} and the above estimate on $j$th order derivatives so that
\begin{align*}
    \infnorm{\nabla_q^{j} \rndP{W_0^E(t)-W_{0,\infty}^E}} &\lesssim \sum_{m=0}^{j} \max_{\alpha=1,...,N}\sup_{|u|\leq \mathcal{U}} \infnorm{\nabla_p^m \rndP{F^{\alpha,0}\rndP{t(1-|u|)}-F^{\alpha,0}_\infty}} \\
    &\lesssim \sum_{m=0}^{j-1} t^{m-n-1} + t^{j-n-2}
     \lesssim t^{j-n-2}.
\end{align*}     
Taking $1\leq \ell\leq n-j$ and splitting the sum at $j+1$, we use \eqref{DkFd} for the $m =0,...,j$ terms and $(\mathcal{R}_i)$ for $i=j+1,...,n$ applied to the $m=j+1$, ..., $j+\ell$ terms in order to find
\begin{align*}
    \infnorm{\nabla_q^{j} \rndP{W_\ell^E(t)-W_{\ell,\infty}^E}}
    &\lesssim \sum_{m=0}^j \max_{\substack{\alpha=1,...,N\\ 0 <|\delta|\leq n+1-j}} \sup_{|u|\leq \mathcal{U}} \infnorm{\nabla_p^m \rndP{F^{\alpha,\delta}\rndP{t(1-|u|)}-F^{\alpha,\delta}_\infty}} \\
    &\qquad + \sum_{m=j+1}^{j+\ell} \max_{\substack{\alpha=1,...,N\\  0 <|\delta| \leq n+1-j}} \sup_{|u|\leq \mathcal{U}} \infnorm{\nabla_p^m \rndP{F^{\alpha,\delta}\rndP{t(1-|u|)}-F^{\alpha,\delta}_\infty}}\\
    &\lesssim \sum_{m=0}^{j} t^{m-n-1} + \sum_{m=j+1}^{j+\ell} t^{m-n-2}
    \lesssim t^{j+\ell-n-2}
\end{align*}
because $\ell \geq 1$.
As $W_{\ell,\infty}^E\equiv W_{\ell,\infty}^B\equiv 0$ and the decay rates for $\nabla_q^{j} W_{\ell}^B(t)$ are identical, this and the above estimate on $W_0^E$ implies
   $$\infnorm{\nabla_q^{j} W_{\ell}^E(t)}, \infnorm{\nabla_q^{j} W_{\ell}^B(t)}  \lesssim t^{j+\ell-n-2}$$
   for every $\ell = 0,...,n-j$.

Finally, we invoke Lemma \ref{Kconv} with $k=j$ and $\ell=n-j+1$ (so that $k+\ell = n+1$) and use \eqref{DkFd} and  $(\mathcal{R}_i)$ for $i=j+1,...,n$ to obtain convergence estimates for field derivatives, namely
\begin{align*}
    & \sup_{x\in\R^3}\tildes{t^{n+3}\nabla_x^{j}E(t,x) - \nabla_q^{j} E_{n-j+1,\infty} \left ( \frac{x}{t} \right )}\\
     &\lesssim \!\sum_{m=0}^{n-j}t^{n-j+1-m} \infnorm{\nabla_q^{j} W_{m}^E(t)}+t^{-1-j}\bigg(\sum_{m=0}^{j}t^m\sup_{|u|\leq \mathcal{U}}\mcK_j\rndP{t(1\!-\!|u|)}\bigg)\mcG_p^{j}(t)\\ &\qquad+\sum_{m=0}^{n+1}\max_{|\delta|=n+1-j}\sup_{|u| \leq \mcU} \infnorm{\nabla_p^m \rndP{F^{\alpha,\delta}(t(1-|u|))-F^{\alpha,\delta}_\infty}} +t^{-1}\mcG_p^{n+2}(t)\\
    &\lesssim \sum_{m=0}^{n-j}t^{n-j+1-m} t^{j+m-n-2}+t^{-1-j}\bigg(\sum_{m=0}^{j}t^mt^{-n-3}\bigg)+\sum_{m=0}^{j}t^{m-n-1}
    +\sum_{m=j+1}^{n+1}t^{m-n-2}+t^{-1}\\
                %
                %
    &\lesssim t^{-1}
\end{align*}
as $0 \leq j \leq n-1$.
Because the estimate for $\nabla_x^jB(t,x)$ is identical, we find 
\begin{equation}
\label{EBconv0n}
\sup_{x\in\bfR^3} \left | t^{n+3} \nabla_x^j(E,B)(t,x) - \nabla^j_q(E,B)_{n+1-j, \infty} \left (\frac{x}{t}\right ) \right | \lesssim t^{-1}.
\end{equation}
With this, we have established \eqref{eq:Rk}.
Moreover, as $(E,B)_{n+1-j, \infty} \equiv 0$, this last estimate implies $\mcK_j(t)\lesssim t^{-n-4}$, which continues the iterative process.
Hence, as \eqref{eq:Rk} holds for every $j = 0, ...., n$, we have established \eqref{EBconv0n} for every $j = 0, ...., n$.\\


\textbf{Step 4:}  Refine other estimates using the updated field decay rates\\

To complete the induction step, we will make use of Lemma \ref{lem:rhojderiv-l} to update the estimates of the charge and current densities and their limiting behavior.
First, recall that the previous step guarantees
\begin{equation}
\label{Kdecay}
\mcK_j(t)\lesssim t^{-n-4}
\end{equation}
for all $j = 0, ..., n+1$.
We further update the convergence rates of the $n$th spatial average using Lemma \ref{DFell} with $k=0$ and $|\delta|\in\N$, which yields
$$ \infnorm{F^{\alpha,\delta}(t)-F_\infty^{\alpha,\delta}} \lesssim \int_t^\infty \Big\{ s\mcK_1(s)+(s+1)\mcK_0(s)  \Big\}\,ds \lesssim t^{-n-2}.$$
Similarly, we update the estimates of derivatives using Lemma \ref{DFell} with $k=1,...,n$ to find
$$ \infnorm{\nabla_p^k\rndP{F^{\alpha,\delta}(t)-F_\infty^{\alpha,\delta}}}
    \lesssim \int_t^\infty \!\bigg\{\sum_{m=0}^{k} \Big(s^{m+1}\mcK_{m+1}(s)+s^m(s+1)\mcK_m(s)\Big)+(s+1)\mcK_0(s)\bigg\}ds
    \lesssim t^{k-n-2}$$
for any $\alpha=1,...,N$ and $|\delta|\in\N$ and
\begin{align*}
\infnorm{\nabla_p^{n+1}\rndP{F^{\alpha,\delta}(t)-F_\infty^{\alpha,\delta}}}
    & \lesssim \int_t^\infty \!\bigg\{s^{n+2} \mcK_{n+2}(s) + s^{n+2}\mcK_{n+1}(s) + \sum_{m=0}^n \Big(s^{m+1}\mcK_{m+1}(s)+s^m(s+1)\mcK_m(s)\Big)\bigg\}ds\\
    & \lesssim \int_t^\infty \!\bigg\{s^{-3}\ln^{2}(s) + s^{-2}\bigg\}ds \lesssim t^{-1}
\end{align*}
using Lemma \ref{DFell} with $k=n+1$.
With these estimates complete, we apply Lemma \ref{lem:rhojderiv-l} again with $\ell = 0,...,n$ and $k=0$, which gives
$$    \sup_{x\in\R^3}\tildes{\rho_\ell(t,x)-\rho_{\ell,\infty}\rndP{\frac{x}{t}}} \lesssim \sum_{j=0}^\ell \max_{\substack{\alpha=1,...,N\\|\delta|=\ell}} \infnorm{\nabla_p^j\rndP{F^{\alpha,\delta}(t)-F_\infty^{\alpha,\delta}}}\lesssim \sum_{j=0}^\ell t^{j-n-2}
    \lesssim t^{\ell-n-2}$$ 
 and
$$ \sup_{\substack{|q|\leq \gamma\\|u|\leq \mathcal{U}}}\tildes{W_\ell^E(t,u,q)-W_{\ell,\infty}^E(u,q)} \lesssim \sum_{j=0}^\ell \max_{\substack{\alpha=1,...,N\\|\delta|=\ell}} \sup_{|u|\leq \mathcal{U}}\infnorm{\nabla_p^j\left(F^{\alpha,\delta}\big(t(1-|u|)\big)-F_\infty^{\alpha,\delta}\right)} \lesssim t^{\ell-n-2}$$
with the same estimates for $j_\ell(t,x)$ and $W_\ell^B(t,u,q)$, respectively.
As the limits vanish, namely $\rho_{\ell,\infty}\equiv j_{\ell,\infty}\equiv W^E_{\ell,\infty}\equiv W^B_{\ell,\infty}\equiv0$ for $\ell = 0, ..., n$, we find
$$ \infnorm{\rho_\ell(t)},\infnorm{j_\ell(t)}, \infnorm{W^E_\ell(t)},\infnorm{W^B_\ell(t)}\lesssim t^{\ell-n-2}$$
for all $\ell =0,...,n$.
Finally, we apply Lemma \ref{rhoconv} with $k=0$ and $\ell=n+1$ and use the previous estimates to deduce
\begin{align*}
\sup_{x\in\R^3} \left| t^{n+4}\rho(t,x)-\rho_{n+1,\infty}\rndP{\frac{x}{t}} \right| &\lesssim 
\sum_{m=0}^{n}  t^{n- m+1}\left \|  \rho_m(t) \right \|_\infty +  \sum_{j=0}^{n+1}\max_{\substack{\alpha = 1, ..., N \\ |\delta| = n+1}} \left\| \nabla_p^j \biggl (F^{\alpha, \delta}(t)  - F^{\alpha, \delta}_\infty \biggr ) \right\|_\infty + t^{-1} \mcG_p^{n+2}(t)\\
& \lesssim t^{-1}
\end{align*}
and 
$$\sup_x \left| t^{n+4} j(t,x)-j_{n+1,\infty}\rndP{\frac{x}{t}} \right| \lesssim t^{-1}.$$


Collecting the results of each step, including \eqref{Gnp23}, \eqref{Knp2FINAL}, \eqref{Gnp2}, \eqref{Knp1FINAL}, \eqref{EBconv0n} and the above convergence estimates for the charge and current densities, we have obtained \eqref{eq:Qn}, which completes the induction step and the proof.

Lastly, to justify the rate of convergence of $g^\alpha(t,x,v)$ and its derivatives, we note that using the stated decay rate of the field and the uniform boundedness of derivatives of $g^\alpha$, the Vlasov equation \eqref{RVMg} yields
$$\Vert \partial_t g^\alpha (t) \Vert_\infty \lesssim t\mcK_0(t) \mcG_{x,p}^1(t) + \mcK_0(t)  \mcG_p^1(t) \lesssim t^{-n-2}.$$
Upon integrating in $t$, we find for any $\alpha = 1, ..., N$
$$\sup_{(x,p) \in \bfR^6} \left | g^\alpha \left(t,x, p \right) - f^\alpha_\infty(x,p) \right | \lesssim \int_t^\infty  \Vert \partial_t g^\alpha (s) \Vert_\infty \ ds  \lesssim t^{-n-1}.$$
Similarly, taking any $k$th-order $x$-derivative and $\ell$th-order $p$-derivative in the Vlasov equation  \eqref{RVMg} with $k,\ell \in \mathbb{N}$ and $k + \ell \leq n+1$ and taking the supremum over $(x,p)\in \mathbb{R}^6$, we find
	\begin{align*}
	\Vert \partial_t D^k_x D^\ell_p g^\alpha(t) \Vert_\infty 
	 \lesssim& \sum_{i=0}^{k-1} \sum_{j=0}^\ell t^j \mcK_{i+j}(t) \biggl ( t \mcG_{x,p}^{k+\ell-i-j+1}(t) + \mcG_{p}^{k+\ell-i-j+1}(t) \biggr )\\
	 &+ \sum_{j=0}^\ell t^j \mcK_{j+k}(t) \biggl ( t \mcG_{x,p}^{\ell-j+1}(t) + \mcG_{p}^{\ell-j+1}(t) \biggr )
	\end{align*}
where the indices $i$ and $j$ represent the number of $x$ and $p$ derivatives that are applied to the Lorentz force, respectively. 
As this holds for arbitrary derivatives, we decompose the double sum and use the previous estimates on the decay of the fields and their derivatives \eqref{Kdecay} and derivatives of the distribution function via \eqref{IH}, Lemma \ref{Dng},  and \eqref{Gnp2} to find 
	\begin{align*}
	\Vert \partial_t \nabla^k_x \nabla^\ell_p g^\alpha(t) \Vert_\infty 
	 \lesssim&
	  \mcK_0(t) \left (t \mcG_{x,p}^{k+\ell+1}(t) + \mcG_p^{k+\ell+1}(t) \right ) + \sum_{j=1}^\ell t^j (t+1) \mcK_j(t)\\
	  & \ + \sum_{i=1}^{k-1} \sum_{j=0}^\ell t^{j}(t+1)  \mcK_{i+j}(t) + \sum_{j=0}^\ell t^j(t+1)  \mcK_{j+k}(t)\\
	   \lesssim& t^{-n-3} + \sum_{j=1}^\ell t^{j-n-3} + \sum_{i=1}^{k-1} \sum_{j=0}^\ell t^{j-n-3} + \sum_{j=0}^\ell t^{j-n-3}\\
	    \lesssim&  \max\{t^{-n-3}, t^{\ell-n-3}\}.
\end{align*}
Upon integrating in $t$, we find for any $\alpha = 1, ..., N$
$$\sup_{(x,p) \in \bfR^6} \left | \nabla_x^k \nabla_p^\ell g^\alpha \left(t,x, v \right) - \nabla_x^k \nabla_p^\ell  f^\alpha_\infty(x,p) \right | \lesssim \int_t^\infty  \Vert \partial_t \nabla_x^k \nabla_p^\ell g^\alpha (s) \Vert_\infty \ ds  \lesssim \max\{t^{-n-2}, t^{\ell-n-2}\}.$$
for $k + \ell \leq n+1$.
The estimates are analogous for the boundary cases $k=0$, $1 \leq \ell \leq n+1$ and $\ell=0$, $ 1 \leq k \leq n+1$.
In particular, the uniform convergence for all derivatives of order $n+1$ and the previously known compact support of the limiting function then implies $f_\infty^\alpha \in C_c^{n+1}(\bfR^6)$, which provides the regularity needed to justify \eqref{eq:F-al-inf}, as stated in Remark \ref{Finf}.
This completes the proof of  Theorem \ref{T2}.

Next, we invoke Theorem \ref{T2} to prove Theorem \ref{T1}.
\begin{proof}[Proof of Theorem \ref{T1}]
Assume $\mcM = 0$. 
Then, for $m = 0$, we merely take $\mcA_{0,\infty} \not\equiv 0$, which implies $\rho_{0,\infty} \not \equiv 0$, and apply Theorem \ref{oldT1}, which yields the sharp bounds
$$
\begin{gathered}
\| \rho(t)\|_\infty \sim t^{-3},\\
\| (E,B)(t)\|_\infty \sim t^{-2},\\
\| \nabla_x (E,B)(t)\|_\infty \sim t^{-3}
\end{gathered}
$$
and the modified scattering result
\begin{equation}
\label{modscattering}
\sup_{(x,p) \in \bfR^6} \left | f^\alpha \left(t,x +v_\alpha(p)t - \frac{e_\alpha}{m_\alpha} \ln(t) \mathbb{A}_\alpha(p) K^\alpha_{0,\infty}(p), p \right) - f^\alpha_\infty(x,p) \right |  \lesssim t^{-1}\ln^{4}(t)
\end{equation}
for every $\alpha = 1, ..., N$ and any associated solution of \eqref{Vlasov}-\eqref{Maxwell}.

Otherwise, we let $m \geq 1$ be given, define $n = m - 1$, and take $\mcA_{\delta,\infty} \equiv 0$ for all $\delta \in \mathbb{N}_0^3$ satisfying $|\delta| = 0, ..., n$ with $\mcA_{\delta,\infty} \not\equiv 0$
for some $|\delta| = n+1$.
Then, applying Theorem \ref{T2} with $n = m-1$ gives
$$
\begin{gathered}
\| \rho(t)\|_\infty \sim t^{-m-3},\\
\| (E,B)(t)\|_\infty \sim t^{-m-2},\\
\| \nabla_x^k (E,B)(t)\|_\infty \sim t^{-m-3} \\
\end{gathered}
$$
for every $k =1, ..., m$.
Additionally, the distribution function scatters linearly so that
\begin{equation}
\label{scattering}
\sup_{(x,p) \in \bfR^6} \left | f^\alpha(t,x +v_\alpha(p)t, p) - f^\alpha_\infty(x,p) \right |  \lesssim t^{-m}
\end{equation}
for every $\alpha = 1, ..., N$ and any associated solution of \eqref{Vlasov}-\eqref{Maxwell}.

It remains to justify the existence of solutions $f^\alpha \in C^{m+9} \left((0,\infty) \times \mathbb{R}^6 \right)$ for $\alpha = 1, ..., N$ 
that satisfy $\mcA_{0,\infty} \not\equiv 0$ or, alternatively, the conditions
$\mcA_{\delta,\infty} \equiv 0$ for all $|\delta| = 0, ..., m-1$
and
$\mcA_{\delta,\infty} \not\equiv 0$
for some $|\delta| = m$.
To do this, we merely construct smooth functions with compact support whose moments, up to a desired order, must vanish.
Thus, we build well-behaved limits and utilize the scattering map constructed within \cite{BigorgneScattering} (see also \cite{Flynn} for the analogous result pertaining to the Vlasov-Poisson system) to guarantee the existence of solutions that tend to these limits as $t \to \infty$.
In particular, the limiting scattering state within \cite{BigorgneScattering} requires $f_\infty^\alpha \in C^{m+9} \left(\mathbb{R}^6 \right)$, and so we must construct a suitably smooth limiting function.
We perform this separately for $m = 0$ and $m \in \mathbb{N}$.
Throughout, we will utilize the useful identity
\begin{equation}
\label{Didentity}
\mcD_\alpha(v_\alpha^{-1}(q)) = m_\alpha^3 \left ( 1 - |q|^2\right )^{-5/2},
\end{equation}
which follows from a brief computation,
so that this quantity only depends upon $\alpha$ via the particle mass $m_\alpha$.

First, for $m=0$, we let a nonzero function $\eta \in C_c^{9}(\mathbb{R}^3)$ be given with
$$\int \eta(q) \ dq  = 0.$$
Then, we choose
$$f^\alpha_{0,\infty}(x,p) = \phi_0(x) \psi_0^\alpha(v_\alpha(p))$$
where $\phi_0 \in C_c^{9}(\mathbb{R}^3)$ is nonnegative (but nontrivial), and for every $\alpha = 1, ..., N$
the functions $\psi_0^\alpha \in C_c^{9}(\mathbb{R}^3)$ are nonnegative and
satisfy the constraint
$$ \sum_{\alpha = 1}^N e_\alpha m_\alpha^3 \psi_0^\alpha (q) =  \left ( 1 - |q|^2 \right )^{5/2}\eta(q)$$
for $|q| \leq \gamma < 1$.
With this, we have by \eqref{eq:F-al-inf}, \eqref{eq:A-l-inf}, and \eqref{Didentity}
	\begin{align*}
	\mcA_{0,\infty}(q)
	& =
	\sum_{\alpha=1}^N e_\alpha \mcD_\alpha \left ( v_\alpha^{-1}(q) \right ) F^{\alpha,0}_\infty \left (v_\alpha^{-1}(q) \right )
	= \int \sum_{\alpha = 1}^N e_\alpha \mcD_\alpha \left ( v_\alpha^{-1}(q) \right ) f^\alpha_{0,\infty}  \left (x, v_\alpha^{-1}(q) \right )  \ dx\\
	& = 	\left( \int \phi_0(x) \ dx \right) \left ( 1 - |q|^2 \right )^{-5/2}\sum_{\alpha = 1}^N e_\alpha m_\alpha^3 \psi^\alpha_{0}(q)\\
	& = \left( \int  \phi_0(x) \ dx \right) \eta(q) 
	 \not\equiv 0
	\end{align*}
and
\begin{align*}
\int \rho_{0,\infty}(q) \ dq=
	4\pi \int \mcA_{0,\infty}(q) \ dq
	= 4\pi \left( \int  \phi_0(x) \ dx \right) \left( \int \eta(q) \ dq \right)
	 = 0,
\end{align*}
thereby satisfying the global neutrality condition $\mcM_{\mathrm{net}} = 0$ imposed by \eqref{Pinfmass}. 

Next, for a given $m \in \N$, we take $p > m+10$ and define $\Phi \in C_c^{m+9}(\mathbb{R})$ to be the corresponding weighted Gegenbauer polynomial of degree $m$ (see \cite{AbramSteg}), namely
$$\Phi_m(x) = \left( 1 - x^2 \right)^{p - \frac{1}{2}} {C}^p_{m}(x) \mathbbm{1}_{[-1,1]}$$
where $C^p_k(x)$ is a $k$th order Gegenbauer polynomial satisfying the orthogonality relationship
$$\int_{-1}^1 \left( 1 - x^2 \right)^{p - \frac{1}{2}} C^p_{k}(x)  C^p_{\ell}(x) \ dx = 0$$
for all $k,\ell \in \N_0$ with $k \neq \ell$
and
$$\int_{-1}^1 \left( 1 - x^2 \right)^{p - \frac{1}{2}} \left [ C^p_{k}(x)  \right]^2 \ dx > 0$$
for any $k \in \N_0$.
Because each $C^p_{k}(x)$ is a polynomial of degree $k$, we may normalize    $\Phi_m(x)$ so that
$$ \int x^{m} \Phi_m(x) \ dx = 1,$$
and orthogonality implies
$$ \int x^k \Phi_m(x) \ dx = 0, \qquad \mathrm{for \ all \ } k = 0, ...,m-1.$$
We note that the support of $\Phi_m$ can be rescaled to be a compact set of arbitrary size, if necessary, while maintaining the moment and regularity properties.
Then, letting $\Psi \in C_c^{m+9}(\bfR)$ be any function satisfying
$$\int \Psi(x) \ dx = 1,$$
we define $\mu_m \in C_c^{m+9}(\mathbb{R}^3)$ by
$$\mu_m (x) =  \Phi_m(x_1) \Psi(x_2) \Psi(x_3). $$
%
Then, for all $\delta \in \mathbb{N}_0^3$ with $|\delta| \leq m$, this yields 
\begin{align*}	
\int x^{\delta} \mu_m(x) \ dx & =  \left( \int x_1^{\delta_1} \Phi_m(x_1) \ dx_1 \right) \left( \int x_2^{\delta_2} \Psi(x_2) \ dx_2 \right) \left( \int x_3^{\delta_3} \Psi(x_3) \ dx_3 \right)\\
& = 
\begin{cases}
1 & \mathrm{if} \ \delta = (m,0,0) \\
0 & \mathrm{else,}
\end{cases}
\end{align*}
as $\delta_1 \in \{0, ..., m-1\}$ implies that the first integral vanishes due to the orthogonality condition stated above.

With the $\mu_m(x)$ functions in place for every $m \in \mathbb{N}$, we can now define the remaining scattering limits.
For a given $m \in \N$ we choose
$$f^\alpha_{m,\infty}(x,p) = \phi_m^\alpha(x) \psi_m(v_\alpha(p))$$
for every $\alpha = 1, ..., N$
where
$\psi_m \in C_c^{m+9}(\mathbb{R}^3)$ 
is nonnegative and nontrivial,
and the collection of nonnegative spatial distributions 
$\phi_m^\alpha \in C_c^{m+9}(\mathbb{R}^3)$
satisfies the constraint
$$\sum_{\alpha = 1}^N e_\alpha m_\alpha^3 \phi_m^\alpha(x) = \mu_m(x)$$
where $m_\alpha$ is the particle mass.
Using \eqref{eq:F-al-inf}, \eqref{eq:A-l-inf}, and \eqref{Didentity}, this construction implies for all $|\delta| = 0, ..., m-1$,
\begin{align*}
\mcA_{\delta, \infty}(q) & = \sum_{\alpha=1}^N e_\alpha \mcD_\alpha \left ( v_\alpha^{-1}(q) \right ) \int (-y)^\delta f^\alpha_{m,\infty}\left (y, v_\alpha^{-1}(q) \right ) \ dy\\
& = \psi_m(q) \left ( 1 - |q|^2 \right )^{-5/2}  \left(\int (-y)^\delta \sum_{\alpha = 1}^N  e_\alpha  m_\alpha^3  \phi_m^\alpha(y) \ dy \right)\\
& = (-1)^\delta \psi_m(q) \left ( 1 - |q|^2 \right )^{-5/2}  \left(\int y^\delta \mu_m(y) dy \right) \equiv 0
\end{align*}
as the lower-order moments of $\mu_m(x)$ must vanish.
Additionally, for the final limiting density with $\delta= (m,0,0)$ we find
\begin{align*}
\mcA_{\delta, \infty}(q) 
& =  (-1)^\delta \psi_m(q) \left ( 1 - |q|^2 \right )^{-5/2}  \left(\int y_1^m \mu_m(y) dy \right)\\
& =   (-1)^\delta \psi_m(q) \left ( 1 - |q|^2\right )^{-5/2}   \not\equiv 0.
\end{align*}

We further note that the construction of each limiting state $f^\alpha_{m,\infty}$ for $\alpha = 1, ..., N$ and $m \in \N_0$ is invariant under scaling of the distribution function. Therefore, we may merely multiply each of these functions by a chosen $\epsilon > 0$, and the vanishing and non-vanishing properties of $\mcA_{\delta, \infty}$ remain. 
Because of this, we let $\epsilon_0$ be the minimum of the smallness constants within \cite{Bigorgne}, \cite{BigorgneScattering}, and  \cite{GS}, and then
and take $\epsilon < \epsilon_0$ sufficiently small so that the scattering limits and initial distributions (due to conservation of the $L^\infty$ norm of $f^\alpha$) are as small as desired.
Furthermore, the limiting radiation fields obtained within \cite{BigorgneScattering} can be taken to be arbitrarily small herein in order to guarantee smallness of the initial electromagnetic fields as well. All of this implies that the initial data launching the solutions which converge to these constructed limiting states must satisfy the required smallness conditions.
This guarantees that the electromagnetic fields $(E,B)(t,x)$ satisfy the decay conditions of Lemma \ref{lem:glassey-strauss}, as for suitably smooth solutions in \cite{GS}.
Finally, as we have constructed a family of limits $f^\alpha_{m,\infty} \in C_c^{m+9}(\mathbb{R}^6)$ for every $\alpha = 1, ..., N$ and $m \in \N_0$, an application of  \cite[Theorem 1.8]{BigorgneScattering} implies that for each $m \in \N_0$ there exists a unique smooth solution of \eqref{Vlasov}-\eqref{Maxwell}, in this case a small data solution $ f_m^\alpha \in C^{m+9} \left((0,\infty) \times \mathbb{R}^6 \right)$ for every $\alpha = 1, ..., N$, which is associated to this limit $f^\alpha_{m,\infty}(x,p)$ via \eqref{modscattering} for $m=0$ and \eqref{scattering} for $m \geq 1$, respectively. 


\end{proof}

\section{Derivatives of the translated distributions}
\label{sec:g_deriv}
In this section we prove Lemmas \ref{Dng0} and \ref{Dng}, which are combined within the following result. 

\begin{lemma}
\label{Dkg}
Let $n \in \N_0$ be given. 
If $n=0$, assume \eqref{K1K2_Base} and \eqref{K1_Base}-\eqref{K3_Base}, as well as
$$\mcG_{x,p}^1(t) + \mcG_p^1(t) \lesssim 1, \qquad \mcG_{x,p}^2(t) \lesssim 1,  \qquad \mathrm{and}  \qquad  \mcG_p^2(t) \lesssim \ln^2(t).$$
Then, we have
$\mcG_{x,p}^{3}(t) \lesssim \ln^2(t)$
and
$$\mcG_p^{3}(t)\lesssim \ln^2(t).$$
Alternatively, if $n \geq 1$, assume \eqref{IH} and \eqref{Knp2Knp3_initial}.
Then, we have
$$\mcG_p^{n+2}(t) \lesssim 1 + \int_1^t s^{n+3} \mcK_{n+2}(s) \ ds,$$
$$\mcG_{x,p}^{n+2}(t) \lesssim 1, \qquad  \mcG_{x,p}^{n+3}(t) \lesssim \ln^2(t), \qquad \mathrm{and} \qquad  \mcG_p^{n+3}(t)\lesssim \ln^2(t)$$ 
and, as a preliminary estimate from \eqref{Knp2Knp3_initial}, we have
$$\mcG_p^{n+2}(t)\lesssim \ln^2(t).$$
\end{lemma}

We introduce some notation to simplify the presentation. For each $k\in\N$ and  $j = 0, ..., k$, we denote
	\[
	\mfG^{k,j}(t)
	:=
	\max_{\alpha = 1, ..., N} \sup_{\substack{|\beta_x|+|\beta_p|=k\\|\beta_p|=j}} \|D_p^{\beta_p} D_x^{\beta_x}g^\alpha(t)\|_\infty
	\]
and
	\[
	\mfG^k(t) =1+ \sum_{j=0}^kt^{-j}\mfG^{k,j}(t).
	\]
We also recall the following notation introduced earlier in the paper:
	\[
	\mcG^k_p(t) = 1 + \max_{\alpha = 1, ..., N} \sum_{|\beta_p|\leq k} \| D_p^{\beta_p} g^\alpha(t) \|_\infty
	\]
and
	\[\mcG^k_{x,p}(t) =  1+ \max_{\alpha = 1, ..., N} \sum_{\substack{|{\beta_p}| + |{\beta_x}|\leq k \\ |\beta_x| > 0}} \| D_p^{\beta_p} D_x^{\beta_x}  g^\alpha(t) \|_\infty.
	\]

\subsection*{Sketch of the proof.}
Before proceeding with the proof, we first provide a brief sketch, as the proof itself is quite technical in some places. The strategy of the proof proceeds as follows.

We wish to understand the asymptotic properties of $D_p^{\beta_p} D_x^{\beta_x} g^\alpha$ with $|\beta_x+\beta_p| = n+m$, where $n=0$ in the base case and $n\geq1$ in the induction step, and $m=2$ or $3$. To do this, we apply these derivatives to the Vlasov equation $\mcV_g g^\alpha = 0$, where the translated Vlasov operator was introduced in \eqref{eq:vlasov-op-g}. The idea is to get an expression that one can bound, and then integrate along characteristics in order to eliminate the operator $\mcV_g$.  This will allow us to construct a bound on the sum of all $n+m$ order derivatives in $x$ and $p$ of $g^\alpha$ to show \[\mfG^{n+m}(t)\lesssim 1\] following an application of Gr\"onwall's inequality.

The uniform bound on $\mfG^{n+m}(t)$ will then imply \[\mfG^{n+m,\ell}(t) \lesssim t^\ell,\qquad\forall\ell=0,\hdots,n+m.\] This polynomial growth will be iteratively improved by applying a Gr\"onwall argument until each of these quantities is uniformly bounded for all $\ell=1,\hdots,n$. The next terms, $\mfG^{n+m,\ell}(t)$,  $\ell=n+1,\dots,n+m$, require further iteration in order to improve on the initial polynomial bounds.

We only consider $|\beta_p+\beta_x|=n+m$ for $m=2$ or $3$  because we already have
 	\begin{equation}
	\label{eq:g-derivs-1}
	\mfG^{k,j}(t) \lesssim \mcG^k_{x,p}(t) \lesssim 1
	\end{equation}
for $j<k \leq n+1$, and
  	\begin{equation}
	\label{eq:g-derivs-2}
	\mfG^{k,k}(t) \lesssim \mcG^k_{p}(t) \lesssim 1
	\end{equation}
 for $j = k\leq n+1$. These are implied by \eqref{IH} for $n \geq 1$, and by assumption for $n=0$.
Using \eqref{K1K2_Base}, \eqref{K1_Base}-\eqref{K3_Base}, \eqref{IH}, and \eqref{Knp2Knp3_initial}, the following field decay rates are valid for all $n\in\N_0$:
\begin{equation}\label{eq:fieldRates}
    \left\{ 
    \begin{aligned}
        \mcK_j(t) & \lesssim t^{-n-3} \qquad & \forall j \in \{0, ..., n\}, & \\
		\mcK_{n+1}(t) & \lesssim t^{-n-4}\ln^2(t),\\
        \mcK_{n+2}(t)&\lesssim t^{-n-4}\ln(t)\\
        \mcK_{n+3}(t)&\lesssim t^{-n-5}\ln(t).
    \end{aligned}
    \right.
\end{equation}
 
\begin{proof}
Because this result is a consequence of the Vlasov equation only, and not the form of the force field, the proof merely follows that of Lemma 3.2 in \cite{Mattingly2025}, which establishes the analogous estimate for the Vlasov-Poisson system, though we provide a full proof for completeness. 
First, we recall the expression \eqref{eq:vlasov-op-g} of the translated Vlasov operator
	\[
	\mcV_g 
	=
	\partial_{t} + \frac{e_\alpha}{m_\alpha}  K^\alpha(t,x+v_\alpha(p)t, p) \cdot \left( - t \A_\alpha(p) \nabla_{x} + \nabla_{p} \right).
	\]
to $n+m$ mixed derivatives (in space and momentum) of $g^\alpha$.

On the one hand, one can integrate along characteristics to eliminate $\mcV_g$, while on the other hand, one can compute the commutator between $\mcV_g$ and the derivatives.  More precisely, we take $\beta_x$ derivatives in $x$ and $\beta_p$ derivatives in $p$, respectively, within the Vlasov equation where $\beta_x,\beta_p$ are multi-indices satisfying $|\beta_x+\beta_p|=n+m$ throughout, then
$$	
\mcV_g\big(D_p^{\beta_p} D_x^{\beta_x} g^\alpha\big)
	=
	\sum_{j=0}^{n+m-1}\left(\!\begin{array}{c}n+m\\j\end{array}\!\right) \sum_{\substack{|\gamma_x+\gamma_p|=j \\ \gamma_x \preceq\beta_x, \gamma_p \preceq\beta_p}}(D_p^{\beta_p-\gamma_p} D_x^{\beta_x-\gamma_x}\mcV_g) \left( D_p^{\gamma_p} D_x^{\gamma_x}g^\alpha \right)
$$
and note that the term with all $|\gamma_p+\gamma_x|=n+m$ derivatives applied to $g^\alpha$ vanishes 
{because the pure Vlasov operator applied to $g^\alpha$ is zero.} 
We invert the operator $\mcV_g$ (which later amounts to integration along the characteristics), writing
	\begin{equation}
	\label{eq:k-derivs-g-2}
	D_p^{\beta_p} D_x^{\beta_x} g^\alpha
	=
	\sum_{j=0}^{n+m-1}\left(\!\begin{array}{c}n+m\\j\end{array}\!\right) \sum_{\substack{|\gamma_x+\gamma_p|=j \\ \gamma_x \preceq\beta_x, \gamma_p \preceq\beta_p}}(\mcV_g)^{-1}\left[(D_p^{\beta_p-\gamma_p} D_x^{\beta_x-\gamma_x}\mcV_g) \left( D_p^{\gamma_p} D_x^{\gamma_x}g^\alpha \right)\right].
	\end{equation}

The next step is to estimate the terms on the right hand side of the form
	\[(D_p^{\beta_p-\gamma_p} D_x^{\beta_x-\gamma_x}\mcV_g) \left( D_p^{\gamma_p} D_x^{\gamma_x}g^\alpha \right),\]
and, in particular, the coefficients appearing within the operator $D_v^{\beta_p-\gamma_p} D_x^{\beta_x-\gamma_x}\mcV_g$.

When considering the differences from the Vlasov-Poisson case, in this Vlasov-Maxwell system, the Vlasov operator contains the Lorentz force, 
	\[
	K^\alpha(t,x,p) = E(t,x) + v_\alpha(p) \times B(t,x),
	\]
which depends upon momentum, rather than simply the electrostatic force $E(t,x)$, which does not. With this in mind, it is a straightforward observation that spatial derivatives satisfy
	\[
	\left | \nabla^j_x K^\alpha(t,x+v_\alpha(p)t, p) \right | \lesssim \mcK_j(t),
	\]
while momentum derivatives satisfy
	\[
	\left | \nabla^j_p  \biggl (  K^\alpha(t,x+v_\alpha(p)t, p)  \biggr ) \right | \lesssim \sum_{i=0}^j t^i \mcK_i(t),
	\]
because of lower-order terms arising from derivatives being applied to $v_\alpha(p)$ and all of its derivatives (recall that $v_\alpha(p)$ and all of its derivatives are uniformly bounded for $|p| \leq \beta$). These estimates imply that
	\begin{align*}
	\left|D_p^{\beta_p-\gamma_p} D_x^{\beta_x-\gamma_x}\mcV_g\right|
	&=
	\left|D_p^{\beta_p-\gamma_p} D_x^{\beta_x-\gamma_x}\left( \frac{e_\alpha}{m_\alpha}  K^\alpha(t,x+v_\alpha(p)t, p) \cdot \left( - t \A_\alpha(p) \nabla_{x} + \nabla_{p} \right)\right)\right|\\
	&\lesssim
	\sum_{i=0}^{|\beta_p-\gamma_p|}t^i\mcK_{i+|\beta_x-\gamma_x|}(t)\left(t\nabla_x+\nabla_p\right).
	\end{align*}
Consequently, we have upon denoting $j=|\gamma_x+\gamma_p|$
\begin{align*}
	\big|(D_p^{\beta_p-\gamma_p} D_x^{\beta_x-\gamma_x}\mcV_g)& \left( D_p^{\gamma_p} D_x^{\gamma_x}g^\alpha \right)\big|
	\lesssim
	\sum_{i=0}^{|\beta_p-\gamma_p|}\big|t^i\mcK_{i+|\beta_x-\gamma_x|}(t)\left(t\nabla_x+\nabla_p\right)\left( D_p^{\gamma_p} D_x^{\gamma_x}g^\alpha \right)\big|\\
	&\lesssim
	\sum_{i=0}^{|\beta_p-\gamma_p|}t^i\mcK_{i+|\beta_x-\gamma_x|}(t)\left(t\mfG^{j+1,|\gamma_p|}(t)+\mfG^{j+1,|\gamma_p|+1}(t)\right).
\end{align*}
Returning to \eqref{eq:k-derivs-g-2} and using this, we have the following bound for derivatives of $g^\alpha$:
\begin{equation}
	\label{DngOrig}
	\|D_p^{\beta_p} D_x^{\beta_x} g^\alpha(t)\|_\infty
	\lesssim
	1+\int_1^t\sum_{j=0}^{n+m-1} \!\!\! \sum_{\substack{|\gamma_x+\gamma_p|=j \\ \gamma_x \preceq\beta_x, \gamma_p \preceq\beta_p}} \!\!\! \sum_{i=0}^{|\beta_p-\gamma_p|}s^i\mcK_{i+|\beta_x-\gamma_x|}(s)\left(s\mfG^{j+1,|\gamma_p|}(s)+\mfG^{j+1,|\gamma_p|+1}(s)\right)\, ds.
\end{equation}

\subsection*{Estimates of $(n+2)$nd order derivatives of $g^\alpha$.}

Letting $m=2$ in Equation \eqref{DngOrig} above results in the initial estimate for $|\beta_p+\beta_x|=n+2$:
\begin{align}
\label{Dng_np2}
    \|D_p^{\beta_p} D_x^{\beta_x} g^\alpha(t)\|_\infty
	& \lesssim 1 + \!\int_1^t \sum_{j=0}^{n+1}\sum_{\substack{|\gamma_x+\gamma_p|=j \\ \gamma_x \preceq\beta_x, \gamma_p \preceq\beta_p}} \sum_{i=0}^{|\beta_p-\gamma_p|}\! s^i\mcK_{i+|\beta_x-\gamma_x|}(s)\!\left(s\mfG^{j+1, |\gamma_p|}(s) +\mfG^{j+1,|\gamma_p|+1}(s)\right)ds.
\end{align}
We will take the supremum over \eqref{Dng_np2} for each $|\beta_p|=\ell\in\{0,...,n+2\}$ to construct estimates of every $\mfG^{n+2,\ell}(t)$. 
\\

\textbf{Step 1:} Initial estimates of $\mfG^{n+2,\ell}(t)$ for each $\ell=0,...,n+2$

First, we set aside the integral and separate the terms within the sum over $j$ on the right side of this inequality as follows $\sum_{j=0}^{n+1}=\sum_{j=0}^{0}+\sum_{j=1}^{1}+\sum_{j=2}^{n}+\sum_{j=n+1}^{n+1}$ denoting these four terms by $I + II + III + IV$. Observe that for $n=0$ we only have the terms $I$ and $IV$.

\textbf{Term $I$ ($j=0$).}
Using \eqref{eq:g-derivs-1} and \eqref{eq:g-derivs-2} to bound derivatives of $g^\alpha$ and \eqref{eq:fieldRates} to bound the field derivatives of order $k=0,...,n+1$, the term $I$ has the straightforward estimate
\begin{equation}
	\label{DngI}
	I \lesssim \sum_{i=0}^{|\beta_p|}s^i\mcK_{i+|\beta_x|}(s)\left(s\mcG_{x,p}^1(s)+\mcG_{p}^{1}(s)\right) \lesssim \sum_{i=0}^{|\beta_p|}s^{i+1}\mcK_{i+|\beta_x|}(s)
	\lesssim s^{-2}\ln^2(s) + s^{1+|\beta_p|}\mcK_{n+2}(s)
\end{equation}
as $j=0$ implies $|\gamma_x| = |\gamma_p| = 0$.

\textbf{Term $II$ ($j=1$ with $n\geq 1$).}  As $|\beta_p- \gamma_p|+|\beta_x-\gamma_x|=|\beta_p+\beta_x|-|\gamma_p+\gamma_x|=n+1$, this term includes derivatives of the fields of all orders up to order $n+1$. 
Also, with $|\gamma_p|\leq 1$, applying \eqref{eq:g-derivs-1} and \eqref{eq:g-derivs-2} and $n\geq 1$, we find
\[ 
    s\mfG^{2,|\gamma_p|}(s)	+\mfG^{2,|\gamma_p|+1}(s) \lesssim s  .
\]
Thus we have
\begin{align*}
	II & =
	\sum_{\substack{|\gamma_x+\gamma_p| = 1 \\ \gamma_x \preceq\beta_x, \gamma_p \preceq\beta_p}} \sum_{i=0}^{|\beta_p-\gamma_p|}s^i \mcK_{i+|\beta_x-\gamma_x|}(s) \left( s\mfG^{2,|\gamma_p|}(s)	+\mfG^{2,|\gamma_p|+1}(s) \right)\lesssim
	\sum_{i=0}^{n} s^{i+1}\mcK_i(s) +s^{n+2}\mcK_{n+1}(s).
\end{align*}
And using the field estimates in \eqref{eq:fieldRates}, we find 
\[ 
    II\lesssim \sum_{i=0}^{n} s^{i+1} s^{-n-3} +s^{n+2}s^{-n-4}\ln^{2}(s)	\lesssim s^{-2}\ln^{2}(s) 
\] 
and II is integrable.

\textbf{Term $III$ ($j=2,...,n$ with $n\geq 2$).} We first notice that $|\beta_p-\gamma_p|+|\beta_x-\gamma_x|=|\beta_p+\beta_x|-|\gamma_p+\gamma_x|=n+2-j$ and from both \eqref{eq:g-derivs-1} and \eqref{eq:g-derivs-2}, we know that the $g^\alpha$ derivatives up to order $n+1$ are uniformly bounded. We also find field derivatives up to order $n$, therefore we have the bound 
\begin{align*}
    III & =
	\sum_{j=2}^{n} \sum_{\substack{|\gamma_x+\gamma_p|=j \\ \gamma_x \preceq\beta_x, \gamma_p \preceq\beta_p}}\sum_{i=0}^{|\beta_p-\gamma_p|}s^i\mcK_{i+|\beta_x-\gamma_x|}(s)\left(s\mfG^{j+1,|\gamma_p|}(s)+\mfG^{j+1,|\gamma_p|+1}(s)\right)
    \lesssim \sum_{j=2}^n s^{n+3-j}\mcK_{n+2-j}(s).
    %
\end{align*}
Applying the field decay estimates in \eqref{eq:fieldRates}, we find
\begin{equation}
\label{DngIII}
III  \lesssim \sum_{j=2}^{n} s^{-j}\lesssim s^{-2}
\end{equation}
and term $III$ is integrable.

\textbf{Term $IV$ ($j=n+1$).} 
Because  $|\beta_x+\beta_p|=n+2$, it follows that either $\gamma_x=\beta_x$ or $\gamma_p =\beta_p$, so that the expression can be simplified to
	\begin{equation}
	\begin{split}
	IV 
	=&
	 \sum_{\substack{|\gamma_x+\gamma_p|=n+1 \\ \gamma_x \preceq\beta_x, \gamma_p \preceq\beta_p}}\sum_{i=0}^{|\beta_p-\gamma_p|}s^i\mcK_{i+|\beta_x-\gamma_x|}(s)\left(s\mfG^{n+2,|\gamma_p|}(s)+\mfG^{n+2,|\gamma_p|+1}(s)\right)\\
	=&
	 \sum_{\substack{|\gamma_x+\gamma_p|=n+1 \\ \gamma_x =\beta_x, \gamma_p \prec\beta_p}}\left(\mcK_{0}(s)+s\mcK_1(s)\right)\left(s\mfG^{n+2,|\gamma_p|}(s)+\mfG^{n+2,|\gamma_p|+1}(s)\right)\\
	 &+
	 \sum_{\substack{|\gamma_x+\gamma_p|=n+1 \\ \gamma_x \prec\beta_x, \gamma_p =\beta_p}}\mcK_{1}(s)\left(s\mfG^{n+2,|\gamma_p|}(s)+\mfG^{n+2,|\gamma_p|+1}(s)\right)\\
	\end{split}
	\label{DngIV}
	\end{equation}
as $s \geq 1$.

Because there are $n+2$ derivatives of $g^\alpha$ appearing, the estimates \eqref{eq:g-derivs-1} and \eqref{eq:g-derivs-2} do not apply.\\

\textbf{Putting all four terms together.} 
Combining the estimates \eqref{DngI}-\eqref{DngIV} and inserting them into the integral within \eqref{DngOrig}, we observe that the terms  $II$ and $III$ are benign, as both are integrable over $s \in(1,\infty)$, and we are therefore left with only terms from $I$ and $IV$. This produces the following estimate, valid for $n\in\N_0$.
	\begin{equation}
	\label{Dng-estimate}
	\begin{aligned}
	\|D_p^{\beta_p} D_x^{\beta_x} g^\alpha(t)\|_\infty
	\lesssim&\
	1+\int_1^ts^{1+|\beta_p|}\mcK_{n+2}(s)\,ds\\
	 &+\int_1^t\sum_{\substack{|\gamma_x+\gamma_p|=n+1 \\ \gamma_x =\beta_x, \gamma_p \prec\beta_p}}\left(\mcK_{0}(s)+s\mcK_1(s)\right)\left(s\mfG^{n+2,|\gamma_p|}(s)+\mfG^{n+2,|\gamma_p|+1}(s)\right)ds
	 \\
	 &+
	 \int_1^t\sum_{\substack{|\gamma_x+\gamma_p|=n+1 \\ \gamma_x \prec\beta_x, \gamma_p =\beta_p}}\mcK_{1}(s)\left(s\mfG^{n+2,|\gamma_p|}(s)+\mfG^{n+2,|\gamma_p|+1}(s)\right)ds
	\end{aligned}
	\end{equation}

Within \eqref{Dng-estimate} we take the supremum over $\alpha = 1, ..., N$ and  $|\beta_x + \beta_p| = n+2$ with $|\beta_p| = j$, $j \in \{0, ..., n+2\}$. This gives us the estimates
	\begin{subequations}
	\begin{align}
	\mfG^{n+2,0}(t)
	 \lesssim&\
	 1 + \int_1^t \Big( s\mcK_{n+2}(s) 
	+ \mcK_1(s) \left( s\mfG^{n+2,0}(s)
	+ \mfG^{n+2,1}(s) \right) \Big)\, ds\label{eq:g_(n+2)_0}\\
	\mfG^{n+2,\ell}(t) \lesssim & \
	 1 + \int_1^t \Big( s^{1+\ell}\mcK_{n+2}(s)  +  s
	\left(\mcK_{0}(s)+s\mcK_1(s)\right)\mfG^{n+2,\ell-1}(s)  \label{eq:g_(n+2)_l}\\
    & \qquad\quad + \left(\mcK_{0}(s)+2s\mcK_1(s)\right)\mfG^{n+2,\ell}(s)
	+ \mcK_1(s) \mfG^{n+2,\ell+1}(s) \Big) \, ds \qquad\forall \ell\in\{1,\dots,n+1\} \nonumber\\
	\mfG^{n+2,n+2}(t) \lesssim & \ 1 + \int_1^t  \Big( s^{n+3}\mcK_{n+2}(s) + \left(\mcK_{0}(s)+s\mcK_1(s)\right) \left(s\mfG^{n+2,n+1}(s)+\mfG^{n+2,n+2}(s)\right)\Big) \, ds .\label{eq:g_(n+2)_(n+2)}
	\end{align}
	\end{subequations}

As the estimate \eqref{Dng-estimate} and thus \eqref{eq:g_(n+2)_0}-\eqref{eq:g_(n+2)_(n+2)} are valid for $n\in\N_0$, we combine the cases $n=0$ and $n\geq 1$ for the remaining steps within this subsection, as the only difference between these cases is that $\mcK_{1}(t)\lesssim t^{-4}\ln^2(t)$ when $n=0$ and $\mcK_1(t)\lesssim t^{-n-3}$ for $n\geq 1$. Because $\mcK_1(t)\lesssim t^{-4}\ln^2(t)\lesssim t^{-3}$ and the weaker bound is sufficient for our remaining arguments, we merely use
\begin{equation}\label{eq:K1estimate}
    \mcK_1(t)\lesssim t^{-n-3} 
\end{equation}
for $n\in\N_0$ in order to simplify the computations.\\

\textbf{Step 2:} Uniform bound on $\mfG^{n+2}(t)$ and a polynomial bound on $\mfG^{n+2,\ell}(t)$

Within each of the resulting inequalities we use the field estimates \eqref{eq:fieldRates} and \eqref{eq:K1estimate}.
From \eqref{eq:g_(n+2)_0}, we find
\begin{align*}
	\mfG^{n+2,0}(t)
	 &\lesssim
	 1 + \int_1^t  s\mcK_{n+2}(s)  \, ds
	+ \int_1^t 	\mcK_1(s) \left( s\mfG^{n+2,0}(s)
	+ \mfG^{n+2,1}(s) \right) \, ds\\
	 &\lesssim 	 1 	+ \int_1^t 
	s^{-n-2} \left(\mfG^{n+2,0}(s)
	+ s^{-1}\mfG^{n+2,1}(s) \right) \, ds\\
	 &\lesssim 1 + \int_1^t 
	s^{-n-2}\mfG^{n+2}(s) \, ds.
\end{align*}

Upon multiplying \eqref{eq:g_(n+2)_l} by $t^{-j}$ and using $s\leq t$, as well as \eqref{eq:fieldRates} and \eqref{eq:K1estimate}, this becomes
	\begin{align*}
	t^{-j}\mfG^{n+2,j}(t)
	 \lesssim&\
	 1 + \int_1^t  s\mcK_{n+2}(s)  \, ds+ \int_1^t s^{-j+1}
	\left(\mcK_{0}(s)+s\mcK_1(s)\right)\mfG^{n+2,j-1}(s)\, ds\\
	&+ \int_1^t s^{-j}
	\left(\mcK_{0}(s)+2s\mcK_1(s)\right)\mfG^{n+2,j}(s)\, ds
	+ \int_1^t s^{-j}
	\mcK_1(s) \mfG^{n+2,j+1}(s)  \, ds\\
	 \lesssim&\
	 1 +\int_1^t s^{-n-2}\left(s^{-(j-1)}
	\mfG^{n+2,j-1}(s)+ s^{-j}
	\mfG^{n+2,j}(s)+ s^{-(j+1)}
	 \mfG^{n+2,j+1}(s)\right)  ds\\
	 \lesssim&\
	 1 
	+ \int_1^t 
	s^{-n-2}\mfG^{n+2}(s) \, ds.
	\end{align*}
Finally, multiplying \eqref{eq:g_(n+2)_(n+2)} by $t^{-n-2}$ and performing the same operations, we have
\begin{align*}
	t^{-n-2}\mfG^{n+2,n+2}(t)
	& \lesssim
	 1 + \int_1^t  s\mcK_{n+2}(s)  \, ds
	+ \int_1^t s^{-n-2}
	\left(\mcK_{0}(s)+s\mcK_1(s)\right)\left(s\mfG^{n+2,n+1}(s)+\mfG^{n+2,n+2}(s)\right)ds\\
	& \lesssim
	 1 + \int_1^t s^{-n-2}\left(
	s^{-(n+1)}\mfG^{n+2,n+1}(s)+s^{-(n+2)}\mfG^{n+2,n+2}(s)\right)ds\\
	& \lesssim
	 1 + \int_1^t 
	s^{-n-2}\mfG^{n+2}(s) \, ds.
\end{align*}
Summing these estimates yields
\[
	\mfG^{n+2}(t) =1+	\sum_{j=0}^{n+2}t^{-j}\mfG^{n+2,j}(t)
	\lesssim 1 + \int_1^t s^{-n-2}\mfG^{n+2}(s)\, ds
\]
and a straightforward Gr\"onwall argument leads to
	\[\mfG^{n+2}(t) \lesssim 1.\]

From the definition of $\mfG^{n+2}(t)$, for $n\geq0$, the estimates for $\mfG^{n+2}(t)$ imply
\begin{equation}
\label{Dk}
\mfG^{n+2,\ell}(t) \lesssim t^\ell
\end{equation}
for $n\geq0$ and for every $\ell = 0, ..., n+2$.\\

\textbf{Step 3:} Refined estimates of  $\mfG^{n+2,\ell}(t)$ for $\ell = 1, ..., n+2$ 

We can now insert the polynomial bounds \eqref{Dk} into the bound \eqref{Dng-estimate}  for $n+2$ derivatives of $g^\alpha$. This will lead to an improvement within each of these estimates. 

We know that for $\ell=0$ the bound for $\mfG^{n+2,\ell}(t)$ is uniform in time. Thus, for $\ell \geq 1$, we set out to refine the estimates of $\mfG^{n+2,\ell}(t)$ for $\ell = 1,\dots,n+1$. We will do so recursively, by assuming $\mfG^{n+2,\ell-1}(t)\lesssim 1$ and taking $\mfG^{n+2,\ell+1}(t)\lesssim t^{\ell+1}$ from \eqref{Dk}. 

Starting from \eqref{eq:g_(n+2)_l}, we apply $\mfG^{n+2,\ell-1}(t)\lesssim 1$ and the decay rates in \eqref{eq:fieldRates} to simplify the expression. Then Gr\"onwall's inequality provides the following bound,
\begin{align*}
	\mfG^{n+2,\ell}(t)
    &\lesssim 1+\int_1^t \!\!\Big( s^{\ell+1}\mcK_{n+2}(s)+s\big(\mcK_0(s)+s\mcK_1(s)\big)+\big(\mcK_0(s)+s\mcK_1(s)\big)\mfG^{n+2,\ell}(s)+\mcK_1(s)\mfG^{n+2,\ell+1}(s)  \Big)\,ds  \\
    &\lesssim \rndP{1 + \int_1^t \mcK_1(s)\rndP{s^2+\mfG^{n+2,\ell+1}(s)}\,ds }\exp\rndP{\int_1^t s\mcK_1(s) \,ds}.
\end{align*}

Applying
\eqref{Dk}, we find for all $\ell=1,...,n$ ($n\geq 1$),
	\begin{equation*}
	\mfG^{n+2,\ell}(t) 
    \lesssim 1 + \int_1^t  s^{-n-3}\Big(s^2+\mfG^{n+2,\ell+1}(s)\Big)\, ds\lesssim 1 + \int_1^t  s^{\ell-n-2}\,ds \lesssim 1.
	\end{equation*}
Then considering the term with $\ell=n+1$, we have the initial estimate
\begin{align}\label{Dnp2np1}
    \mfG^{n+2,n+1}(t) \lesssim 
        1+ \int_1^t  \mcK_1(s)\Big(s^2+\mfG^{n+2,n+2}(s)\Big)\, ds\lesssim 1 + \int_1^t  s^{-1}\,ds\lesssim \ln(t).
\end{align}
Using \eqref{eq:K1estimate}, this holds for all $n\in \N_0$ and will be improved below.
Finally, using \eqref{Dnp2np1} within \eqref{eq:g_(n+2)_(n+2)}, we find 
\begin{align*}
    \mfG^{n+2,n+2}(t) &\lesssim 1+\int_1^t \bigg(s^{n+3}\mcK_{n+2}(s)+s\big(\mcK_0(s)+s\mcK_1(s)\big)\ln(s)+\big(\mcK_0(s)+s\mcK_1(s)\big)\mfG^{n+2,n+2}(s)\bigg)\, ds.
\end{align*}
Using the bound on $\mcK_{0}$ from \eqref{eq:fieldRates}, an application of Gr\"onwall's inequality then leads to
	\[\mfG^{n+2,n+2}(t) \lesssim
	 \left(1 + \int_1^t \Big\{s^{n+3} \mcK_{n+2}(s)+s^2\mcK_1(s)\ln(s)\Big\} \, ds \right) \exp \left( \int_1^t s\mcK_1(s) \, ds \right).  
	 \]
Then applying the decay of $\mcK_{n+2}(t)$ from \eqref{eq:fieldRates} and $\mcK_1(t)$ from \eqref{eq:K1estimate}, we find
\begin{align*}
    \mfG^{n+2,n+2}(t) \lesssim 1+\int_1^t \Big( s^{-1}\ln(s)+ s^{-n-1}\ln(s)\Big)\,ds \lesssim \ln^2(t)
\end{align*}
for all $n\in\N_0$.
Using the equivalence of the norms, we find
\begin{equation}
\label{eq:prel-G-n+2,n+2}
 \mfG^{n+2,n+2}(t) \lesssim \mcG_p^{n+2}(t) \lesssim (n+2)\mfG^{n+2,n+2}(t) \lesssim \ln^2(t) 
 \end{equation}
as a preliminary estimate, which we can then insert into \eqref{Dnp2np1}. 
Doing so yields a uniform bound on $\mfG^{n+2,n+1}(t)$, where in the case $n=0$ we invoke the stronger estimate on $\mcK_1(t)$ from \eqref{eq:fieldRates}. With this, we find
\begin{align*}
    \mfG^{n+2,n+1}(t)
    \lesssim \begin{cases} \displaystyle 1 + \int_1^t  s^{-4}\ln^2(s)\Big(s^2+\ln^2(s)\Big)\, ds\lesssim 1+\int_1^t s^{-2}\ln^2(s)\,ds\lesssim 1 & n=0\\
    \displaystyle 1 + \int_1^t  s^{-n-3}\Big(s^2+\ln^2(s)\Big)\, ds\lesssim 1+\int_1^t s^{-n-1}\,ds\lesssim 1 & n\geq 1
    \end{cases}
\end{align*}
so that $\mfG^{n+2,n+1}(t)\lesssim 1$ for all $n\in\N_0$.
Therefore we have for all $\ell=1,...,n+1$,
	\begin{equation}\label{eq:prel-G-n+2,ell}
	\mfG^{n+2,\ell}(t) \lesssim \mcG_{x,p}^{n+2}(t) \lesssim 1
	\end{equation}
and 
\[
	\mcG_p^{n+2}(t) \lesssim  1+\int_1^t s^{n+3}\mcK_{n+2}\,ds.
\]
for all $n\in\N_0$.

\subsection*{Estimates of $(n+3)$rd order derivatives of $g^\alpha$.}

Taking derivatives of order $|\beta_x + \beta_p| = n+3$ within the Vlasov equation now implies $|\gamma_x+\gamma_p|\leq n+2$ and the integral bound in \eqref{DngOrig} becomes
\begin{align}
\label{Dng_np3}
    \|D_p^{\beta_p} D_x^{\beta_x} g^\alpha(t)\|_\infty
	&\lesssim 1+\!\int_1^t \sum_{j=0}^{n+2}\sum_{\substack{|\gamma_x+\gamma_p|=j \\ \gamma_x \preceq\beta_x, \gamma_p \preceq\beta_p}}
\sum_{i=0}^{|\beta_p-\gamma_p|}\!s^i\mcK_{i+|\beta_x-\gamma_x|}(s)\!\left(s\mfG^{j+1, |\gamma_p|}(s)+\mfG^{j+1,|\gamma_p|+1}(s)\right)ds.
\end{align}
As done above, we will take the supremum over \eqref{Dng_np3} for each $|\beta_p|=\ell\in\{0,...,n+3\}$ to construct estimates of every $\mfG^{n+3,\ell}(t)$. 
\\

\textbf{Step 1:} Initial estimates of $\mfG^{n+3,\ell}(t)$ for each $\ell=0,...,n+3$

We separate the terms within the sum over $j$ in the integrand of \eqref{Dng_np3} as follows:
$$ \sum_{j=0}^{n+2} = \sum_{j=0}^{0} +\sum_{j=1}^{1}+\sum_{j=2}^{n}+\sum_{j=n+1}^{n+1}+\sum_{j=n+2}^{n+2}  = I + II + III + IV + V.$$
In the base case when $n=0$, the sum over $j=0,1,2$ is given by $I+II+V$. In the case when $n\geq 1$, we will show that terms $III$ and $IV$ integrate to $\mathcal{O}(1)$ given the decay rates in \eqref{eq:fieldRates}. 

\textbf{Term $I$ ($j=0$).} As $j=0$ implies $|\gamma_x|=|\gamma_p|=0$, we use \eqref{eq:g-derivs-1} and \eqref{eq:g-derivs-2} to deduce
\begin{align*}
    I= \sum_{i=0}^{|\beta_p|}s^{i} \mcK_{i+|\beta_x|}(s)\big(s\mfG^{1,0}(s)+\mfG^{1,1}(s)\big) \lesssim s^{-2}\ln^2(s) + s^{|\beta_p|}\mcK_{n+2}(s) + s^{|\beta_p|+1}\mcK_{n+3}(s).
\end{align*}

\textbf{Term $II$ ($j=1$).} Here $|\beta_p-\gamma_p|+|\beta_x-\gamma_x|= n+2$ and $|\gamma_p|\leq 1$. 
{From \eqref{eq:prel-G-n+2,n+2} and \eqref{eq:prel-G-n+2,ell}, we have the estimates $\mfG^{2,1}(t)\lesssim \mcG_{x,p}^2(t) \lesssim 1$, and $\mfG^{2,2}(t)\lesssim \mcG_p^2(t)\lesssim \ln^2(t)$.} Thus for any $|\gamma_p|=0,1$, we have \[s\mfG^{2,|\gamma_p|}(s)+\mfG^{2,|\gamma_p|+1}(s) \lesssim s.\]
Thus, we find
\begin{align*}
    II&=\sum_{\substack{|\gamma_x+\gamma_p|=1 \\ \gamma_x \preceq\beta_x, \gamma_p \preceq\beta_p}}\sum_{i=0}^{|\beta_p-\gamma_p|}s^{i} \mcK_{i+|\beta_x-\gamma_x|}(s)\big(s\mfG^{2,|\gamma_p|}(s)+\mfG^{2,|\gamma_p|+1}(s)\big)\\
    &\lesssim \sum_{\substack{|\gamma_p|=1 \\ \gamma_p \preceq\beta_p,\gamma_x=0}} \hspace{-0.5cm} s^{|\beta_p|} \mcK_{n+2}(s) + \sum_{\substack{|\gamma_x|=1, \\\gamma_p=0, \gamma_x \preceq\beta_x}}  \hspace{-0.5cm} s^{|\beta_p|+1} \mcK_{n+2}(s) \\
    &\lesssim s^{|\beta_p|} \mcK_{n+2}(s)\mathds{1}_{1\leq |\beta_p| \leq n+3}(s)+ s^{|\beta_p|+1} \mcK_{n+2}(s)\mathds{1}_{0\leq |\beta_p| \leq n+2}(s).
\end{align*}

\textbf{Term $III$ ($j =2,...,n$ with $n\geq 2$).} We notice that $|\beta_p-\gamma_p|+|\beta_x-\gamma_x|= n+3-j $ and apply both \eqref{eq:g-derivs-1} and  \eqref{eq:g-derivs-2} to find
\begin{align*}
    III&=\sum_{j=2}^{n} \sum_{\substack{|\gamma_x+\gamma_p|=j \\ \gamma_x \preceq\beta_x, \gamma_p \preceq\beta_p}} \sum_{i=0}^{|\beta_p-\gamma_p|} s^{i} \mcK_{i+|\beta_x-\gamma_x|}(s) \big(s\mfG^{j+1, |\gamma_p|}(s)+\mfG^{j+1,|\gamma_p|+1}(s)\big) \\ &\lesssim \sum_{j=2}^{n} s^{n+4-j}\mcK_{n+3-j}(s)
    \lesssim s^{n+2}\mcK_{n+1}(s)+\sum_{i=3}^{n} s^{i+1}\mcK_{i}(s) .
\end{align*}
Estimating using the field estimates in \eqref{eq:fieldRates}, we find
\[III \lesssim s^{n+2}s^{-n-4}\ln^2(s)+\sum_{i=3}^{n} s^{i+1}s^{-n-3} \lesssim s^{-2}\ln^2(s),\]
and $III$ is integrable.

\textbf{Term $IV$ ($j=n+1$ with $n\geq 1$).} For this term we have $ |\beta_p-\gamma_p|+|\beta_x-\gamma_x|=2$.  Because we have $\mfG^{n+2,n+2}(t)\lesssim \mcG_p^{n+2}(t)\lesssim \ln^2(t)$ from \eqref{eq:prel-G-n+2,n+2} and $\mfG^{n+2,i}(t)\lesssim \mcG_{x,p}^{n+2}(t) \lesssim 1$ for all $i=0,...,n+1$ from \eqref{eq:g-derivs-1}, we find
\[s\mfG^{n+2,|\gamma_p|}(s)+\mfG^{n+2,|\gamma_p|+1}(s) \lesssim s\]
for all values of $|\gamma_p| = 0,...,n+1$.
This yields
\begin{align*}
    IV &= \sum_{\substack{|\gamma_x+\gamma_p|=n+1 \\ \gamma_x \preceq\beta_x, \gamma_p \preceq\beta_p}} \sum_{i=0}^{|\beta_p-\gamma_p|} s^{i} \mcK_{i+|\beta_x-\gamma_x|}(s)\Big(s\mfG^{n+2,|\gamma_p|}+\mfG^{n+2,|\gamma_p|+1}\Big)\\
    &\lesssim  \sum_{\substack{|\beta_x-\gamma_x|=2 \\ \beta_p=\gamma_p,\gamma_x \preceq\beta_x}}  s \mcK_{2}(s)  + \sum_{\substack{|\beta_p-\gamma_p|=1, |\beta_x-\gamma_x|=1 \\ \gamma_p \preceq\beta_p, \gamma_x \preceq\beta_x}} \sum_{i=0}^{1} s^{i+1} \mcK_{i+1}(s) + \sum_{\substack{|\beta_p-\gamma_p|=2 \\ \gamma_p \preceq\beta_p,\beta_x=\gamma_x}} \sum_{i=0}^{2} s^{i+1} \mcK_{i}(s) \\
    &\lesssim s\mcK_2(s)\mathds{1}_{0\leq |\beta_p|\leq n+1}(s) + \Big(s\mcK_1(s)+s^2\mcK_2(s)\Big)\mathds{1}_{1\leq |\beta_p|\leq n+2}(s)\\
    &\qquad + \Big(s\mcK_0(s)+s^2\mcK_1(s)+s^3\mcK_2(s)\Big) \mathds{1}_{2\leq |\beta_p|\leq n+3}(s).
\end{align*}
Using the field estimates in \eqref{eq:fieldRates}, this reduces to
\begin{align*}
    IV\lesssim 
    \begin{cases}
     s^{-3}+s^{-2}+ s^{-2}\ln^2(s)\lesssim s^{-2}\ln^2(s) & \text{ for } n=1\\
     s^{-n-2}+s^{-n-1}+s^{-n}\lesssim s^{-n} & \text{ for } n\geq 2.
    \end{cases}
\end{align*}
In both cases, $IV$ is integrable.

\textbf{Term $V$ ($j=n+2$).} As $|\gamma_x+\gamma_p|=n+2$, we conclude $|\beta_p-\gamma_p|+|\beta_x-\gamma_x|=1$, and thus
\begin{align*}
    V&= \sum_{\substack{|\gamma_x+\gamma_p|=n+2\\ \gamma_x \preceq\beta_x, \gamma_p \preceq\beta_p}} \sum_{i=0}^{|\beta_p-\gamma_p|} s^{i} \mcK_{i+|\beta_x-\gamma_x|}(s) \Big(s\mfG^{n+3,|\gamma_p|}(s)+\mfG^{n+3,|\gamma_p|+1}(s)\Big)\\
    &\lesssim \sum_{\substack{|\gamma_x+\gamma_p|=n+2 \\ \gamma_x =\beta_x, \gamma_p \prec\beta_p}} \hspace{-0.55cm}   \rndP{\mcK_0(s)+s\mcK_{1}(s)} \!\left(s\mfG^{n+3,|\gamma_p|}(s)+\mfG^{n+3,|\gamma_p|+1}(s)\right)+\sum_{\substack{|\gamma_x+\gamma_p|=n+2 \\ \gamma_x \prec\beta_x, \gamma_p =\beta_p}} \hspace{-0.55cm} \mcK_{1}(s)\!\left(s\mfG^{n+3,|\gamma_p|}(s)+\mfG^{n+3,|\gamma_p|+1}(s)\right)\\
    &\lesssim \rndP{\mcK_0(s)+s\mcK_1(s)} \rndP{s\mfG^{n+3,|\beta_p|-1}+\mfG^{n+3,|\beta_p|}(s)} \mathds{1}_{1\leq |\beta_p| \leq n+3}(s)\\
    &\qquad + \mcK_1(s)\rndP{s\mfG^{n+3,|\beta_p|}(s)+\mfG^{n+3,|\beta_p|+1}(s)}\mathds{1}_{0\leq |\beta_p| \leq n+2}(s).
\end{align*}

With these estimates in place, we recombine the terms $I$-$V$ to find
\begin{equation}
\label{eq:Dg_np3}
    \begin{aligned}
        \infnorm{D_p^{\beta_p} D_x^{\beta_x} g^\alpha(t)}
	    &\lesssim 1+ \int_1^t \bigg\{s^{|\beta_p|+1}\mcK_{n+3}(s) \\
        &\hspace{-0.75cm} + \Big[ s^{|\beta_p|+1}\mcK_{n+2}(s)+\mcK_1(s)\rndP{s\mfG^{n+3,|\beta_p|}(s)+\mfG^{n+3,|\beta_p|+1}(s)} \Big]\mathds{1}_{0\leq |\beta_p| \leq n+2}(s)\\
        &\hspace{-0.75cm} + \Big[ s^{|\beta_p|}\mcK_{n+2}(s) +   \rndP{\mcK_0(s)+s\mcK_1(s)}\rndP{s\mfG^{n+3,|\beta_p|-1}+\mfG^{n+3,|\beta_p|}(s)} \Big] \mathds{1}_{1\leq |\beta_p|\leq n+3}(s) \bigg\}\,ds.
    \end{aligned}
\end{equation}
We then take the supremum over each $|\beta_p|=\ell$, recalling the definition
\[ \mfG^{k,j}(t):=\max_{\alpha=1,...,N} \sup_{\substack{|\beta_p+\beta_x|=k\\|\beta_p|=j}} \infnorm{D_p^{\beta_p} D_x^{\beta_x} g^\alpha(t)}, \]
and express this inequality separately for $\ell = 0$, $\ell = 1,..., n+2$, and $\ell = n+3$, which yields
\begin{subequations}
 \begin{align}
    \mfG^{n+3,0}(t) &\lesssim 1+\int_1^t\! \Big( s\mcK_{n+3}(s) + s \mcK_{n+2}(s)+ \mcK_1(s)\rndP{s\mfG^{n+3,0}(s)+\mfG^{n+3,1}(s)}\Big)\,ds,\label{eq:Gnp20_primary} \\
    \mfG^{n+3,\ell}(t) &\lesssim 1+\int_1^t \!\Big( s^{\ell+1}\mcK_{n+3}(s) + s^{\ell+1}\mcK_{n+2}(s)+\rndP{s\mcK_0(s)+s^2\mcK_1(s)}\mfG^{n+3,\ell-1}(s)\label{eq:Gnp2ell_primary} \\
	&\qquad +\rndP{\mcK_0(s)+2s\mcK_1(s)}\mfG^{n+3,\ell}(s)+\mcK_1(s)\mfG^{n+3,\ell+1}(s) \Big)\,ds \quad \forall \ell\in\{1,...,n+2\},\nonumber\\
    \mfG^{n+3,n+3}(t) &\lesssim 1+\int_1^t \!\Big( s^{n+4}\mcK_{n+3}(s) + s^{n+3}\mcK_{n+2}(s) \label{eq:Gnp3_primary}\\
    & \qquad + \rndP{\mcK_0(s)+s\mcK_1(s)}\rndP{s\mfG^{n+3,n+2}(s)+\mfG^{n+3,n+3}(s)}  \Big)\,ds. \nonumber
 \end{align}
\end{subequations}
As in the previous subsection, we combine the cases $n=0$ and $n\geq 1$ within the remaining steps by using \eqref{eq:K1estimate}.
 \\
 
\textbf{Step 2:} Uniform bound on $\mfG^{n+3}(t)$ and polynomial bound on $\mfG^{n+3,\ell}(t)$

As in the section above, we use equations \eqref{eq:Gnp20_primary}-\eqref{eq:Gnp3_primary} to construct polynomial growth estimates with a Gr\"onwall argument using the definition of $\mfG^{n+3}(t)$. 
Begin by constructing the first two terms of $\mfG^{n+3}$ in \eqref{eq:Gnp20_primary} and use that the first two terms are bounded by the entire sum. The field terms will be estimated using the decay rates in \eqref{eq:fieldRates}.
\begin{align*}
    \mfG^{n+3,0}(t) &\lesssim 1+\int_1^t \Big( s\mcK_{n+3}(s) + s \mcK_{n+2}(s)+ s\mcK_1(s)\rndP{\mfG^{n+3,0}(s)+s^{-1}\mfG^{n+3,1}(s)}\Big)\,ds\\
    &\lesssim 1+\int_1^t \Big( s^{-n-4}\ln(s)+s^{-n-3}\ln(s)+ s^{-n-2}\rndP{\mfG^{n+3,0}(s)+s^{-1}\mfG^{n+3,1}(s)}\Big)\,ds\\
    &\lesssim 1+\int_1^t  s^{-n-2} \mfG^{n+3}(s)\,ds.
\end{align*}
Similarly, multiplying \eqref{eq:Gnp2ell_primary} by $t^{-j}$ for every $j=1,...,n+2$, using $s\leq t$, and \eqref{eq:fieldRates}, we find
\begin{align*}
    t^{-j}\mfG^{n+3,j}(t) &\lesssim 1+\int_1^t \Big( s\mcK_{n+3}(s) + s\mcK_{n+2}(s)+s^{-(j-1)}\rndP{\mcK_0(s)+s\mcK_1(s)}\mfG^{n+3,j-1}(s) \\
	&\qquad +s^{-j}\rndP{\mcK_0(s)+2s\mcK_1(s)}\mfG^{n+3,j}(s)+s^{-j}\mcK_1(s)\mfG^{n+3,j+1}(s) \Big)\,ds\\ 
    %
    %
    &\lesssim 1+\int_1^t s^{-n-2}\Big( s^{-(j-1)}\mfG^{n+3,j-1}(s)+s^{-j}\mfG^{n+3,j}(s)+s^{-(j+1)}\mfG^{n+3,j+1}(s) \Big)\,ds\\ 
    &\lesssim 1+\int_1^t s^{-n-2}\mfG^{n+3}(s)\,ds.
\end{align*}
Finally, multiplying the last term \eqref{eq:Gnp3_primary} by $t^{-n-3}$ and performing the same operations, we have
\begin{align*}
    t^{-n-3}\mfG^{n+3,n+3}(t) &\lesssim 1+\int_1^t \Big( s\mcK_{n+3}(s) + \mcK_{n+2}(s) + s^{-(n+2)}\big(\mcK_0(s)+s\mcK_1(s)\big) \mfG^{n+3,n+2}(s)\\
    &\qquad  +s^{-(n+3)}\big(\mcK_0(s)+s\mcK_1(s)\big) \mfG^{n+3,n+3}(s)\Big)\,ds\\
    %
    %
    &\lesssim 1+\int_1^t s^{-n-2}\rndP{s^{-(n+2)}\mfG^{n+3,n+2}(s)+s^{-(n+3)}\mfG^{n+3,n+3}(s)} \,ds\\
    &\lesssim 1+\int_1^t s^{-n-2}\mfG^{n+3}(s) \,ds
\end{align*}
Summing these estimates yields
\[ \mfG^{n+3}(t)=1+\max_{\alpha=1,...,N}\sum_{j=0}^{n+3}t^{-j}\mfG^{n+3,j}(t) \lesssim 1+\int_1^t s^{-n-2}\mfG^{n+3}(s)\,ds \]
and a straightforward Gr\"onwall argument leads to 
\[\mfG^{n+3}(t)\lesssim 1.\]
Thus, for $n\in\N_0$ we have
\begin{equation}
\label{eq:GFpoly}
    \mfG^{n+3,\ell}(t)\lesssim t^\ell \qquad \forall \ell =0,...,n+3,
\end{equation}
which we then refine iteratively.
\\

\textbf{Step 3:} Refined estimates of $\mfG^{n+3,\ell}(t)$ for $\ell=1,...,n+3$

As $\mfG^{n+3,0}(t)\lesssim 1$ from \eqref{eq:GFpoly}, we estimate $\mfG^{n+3,\ell}(t)$ for any $\ell\in\{1,...,n+1\}$ recursively while assuming $\mfG^{n+3,\ell-1}(t)\lesssim 1$ 

in Equation \eqref{eq:Gnp2ell_primary}.
When applying the decay rates in \eqref{eq:fieldRates} and the polynomial bound for $\mfG^{n+3,\ell+1}(s)$ in \eqref{eq:GFpoly}, Gr\"{o}nwall's inequality implies
$$\mfG^{n+3,\ell}(t)\lesssim1+\int_1^t \mcK_1(s)\mfG^{n+3,\ell+1}(s) \,ds \lesssim 1+\int_1^t s^{\ell-n-2}\,ds \lesssim 1 \qquad \forall \ell=1,...,n$$

for $n\geq 1$, and
\begin{equation}\label{eq:GFnp1}
    \mfG^{n+3,n+1}(t)\lesssim 1+\int_1^t \Big\{ s^{n+2}\mcK_{n+2}(s)+{\mcK_1(s)\mfG^{n+3,n+2}(s)} \Big\}\,ds \lesssim 1+\int_1^t s^{-1}\,ds \lesssim \ln(t)
\end{equation}
for all $n\in\N_0$.

With this, we estimate the next term with $\ell=n+2$ in Equation \eqref{eq:Gnp2ell_primary} to find 
\begin{align*}
    \mfG^{n+3,n+2}(t) 
    &\lesssim 1+\int_1^t \Big( s^{n+3}\mcK_{n+3}(s) + s^{n+3}\mcK_{n+2}(s)+\rndP{s\mcK_0(s)+s^2\mcK_1(s)}\ln(s) \\ & \qquad \qquad +\rndP{\mcK_0(s)+s\mcK_1(s)}\mfG^{n+3,n+2}(s)+\mcK_1(s)\mfG^{n+3,n+3}(s)\Big)\,ds 
\end{align*}
Upon applying the decay rates in \eqref{eq:fieldRates}, and the polynomial bound for $\mfG^{n+3,n+3}(s)$, Gr\"{o}nwall's inequality then yields
\begin{equation} 
\label{eq:GFnp2}
    \mfG^{n+3,n+2}(t) \lesssim 1+\int_1^t \Big\{s^{n+3}\mcK_{n+2}(s)+ {\mcK_1(s)\mfG^{n+3,n+3}(s)} \Big\}\,ds  \lesssim 1+\int_1^t 1\,ds \lesssim t.
\end{equation}

Finally we estimate the last term with all $p$-derivatives on the fields, $\ell=n+3$, so that
\begin{align*}
    \mfG^{n+3,n+3}(t) &\lesssim 1+\int_1^t \Big( s^{n+4}\mcK_{n+3}(s) + s^{n+3}\mcK_{n+2}(s)+ \rndP{\mcK_0(s)+s\mcK_1(s)}\rndP{s\mfG^{n+3,n+2}(t) +\mfG^{n+3,n+3}(s)}  \Big)ds 
\end{align*}
Utilizing the decay rates in \eqref{eq:fieldRates} and the result from \eqref{eq:GFnp2}, Gr\"{o}nwall's inequality implies 
\begin{equation}
\label{eq:GFnp3}
    \mfG^{n+3,n+3}(t)\lesssim 1+\int_1^t \Big\{ {s^{n+4}\mcK_{n+3}(s)}+{s^{n+3}\mcK_{n+2}(s)} \Big\}\,ds \lesssim 1+\int_1^t s^{-1}\ln(s)\,ds \lesssim \ln^2(t).
\end{equation}
This estimate is optimal until both $\mcK_{n+2}(t)$ and $\mcK_{n+3}(t)$ are known to possess stronger decay rates.

Now that $\mfG^{n+3,n+3}(t)$ has an improved bound, we refine the estimate on $\mfG^{n+3,n+2}(t)$ using Equation \eqref{eq:GFnp2} to find
\[\mfG^{n+3,n+2}(t) \lesssim 1+\int_1^t \Big\{{s^{n+3}\mcK_{n+2}(s)} + \mcK_1(s)\ln^2(s) \Big\}\,ds \lesssim 1+\int_1^t s^{-1}\ln(s)\,ds \lesssim \ln^2(t).\]
In summary, we have obtained
$$ \mcG_{x,p}^{n+3}(t)\lesssim 1 + \ln(t) + \mfG^{n+3,n+2}(t) \lesssim \ln^2(t) \qquad \text{and} \qquad \mcG_{p}^{n+3}(t)\lesssim \mfG^{n+3,n+3}(t)\lesssim \ln^2(t)$$
for all $n\in\N_0$, as desired and the proof is complete.

\end{proof}

\appendix
\section{External Lemmas}

First, we recall the previously-known growth of momentum derivatives of $g^\alpha$.
\begin{lemma}\cite[Lemma 3.5]{Pan-BA2025}
\label{Dpg}
We have
$$\mcG_p^1(t) \lesssim 1 + \int_1^t \left (s \mcK_0(s) + s^2 \mcK_1(s)  \right ) ds.$$
\end{lemma}

Next, we recall an estimate on second-order derivatives of the translated distribution function using the decay of field derivatives.
\begin{lemma}\cite[Lemma 4.1]{Pan-BA2025}
\label{D2g}
We have
$$\mcG_p^2(t) \lesssim 1 + \int_1^t \left (s^3 \mcK_2(s) + \ln^2(s) \left [s^2 \mcK_1(s) + s \mcK_0(s) \right ] \right )\,ds,$$
and
$$\mcG_{x,p}^2(t) \lesssim 1 + \int_1^t s^2 \mcK_1(s)\,ds.$$
Hence, from the initial field decay we have the preliminary growth estimates
$$\Vert \nabla^2_x g(t) \Vert_\infty \lesssim 1, \qquad \Vert \nabla_p \nabla_x g(t) \Vert_\infty \lesssim \ln^2(t),  \qquad \mathrm{and} \qquad \Vert \nabla^2_p g(t) \Vert_\infty \lesssim \ln^4(t).$$
\end{lemma}



\begin{thebibliography}{}

\bibitem{AbramSteg} Abramowitz, M. and Stegun, I., Handbook of Mathematical Functions with Formulas, Graphs, and Mathematical Tables. Dover (9th Ed) {\bf 1964}.

\bibitem{BD} Bardos, C. and Degond, P., Global existence for the {V}lasov-{P}oisson equation in {$3$} space variables with small initial data. Ann. Inst. H. Poincar\'e Anal. Non Lin\'eaire {\bf 1985}, 2(2): 101-118.

\bibitem{BKR} Batt, J., Kunze, M., and Rein, G., On the asymptotic behavior of a one-dimensional, monocharged plasma and a rescaling method. Advances in Differential Equations {\bf 1998}, 3: 271-292.


\bibitem{BCP1} Ben-Artzi, J., Calogero, S., and Pankavich, S., Arbitrarily large solutions of the Vlasov-Poisson system. SIAM J. Math. Anal. {\bf 2018}, 50(4): 4311-4326.

\bibitem{BCP2} Ben-Artzi, J., Calogero, S., and Pankavich, S., Concentrating solutions of the relativistic Vlasov-Maxwell system. Commun. Math. Sci. {\bf 2019}, 17(2): 377-392.

\bibitem{BMP} Ben-Artzi, J., Morisse, B., and Pankavich, S., Asymptotic Growth and Decay of Two-dimensional Symmetric Plasmas. Kinetic and Related Models {\bf 2024}, 17(1): 29-51.


\bibitem{Bigorgne} Bigorgne, L., Global existence and modified scattering for the solutions to the Vlasov-Maxwell system with a small distribution function, Analysis \& PDE {\bf 2025}, 18(3): 629-714.

\bibitem{BigorgneScattering} Bigorgne, L., Scattering map for the Vlasov-Maxwell system around source-free electromagnetic fields, arXiv:2312.12214.

\bibitem{BigorgneRVM} Bigorgne, L., Sharp Asymptotic Behavior of Solutions of the 3d Vlasov-Maxwell System with Small Data, Comm. Math. Phys. {\bf 2020}, 376, 893-992.

\bibitem{Bigorgne-Ruiz} Bigorgne, L. and Ruiz R. V., Late-time asymptotics of small data solutions for the Vlasov-Poisson system, Nonlinearity  {\bf 2026}, 39(4): 045007.


\bibitem{Breton} Breton E., Modified Scattering for Small Data Solutions to the Vlasov-Maxwell System: A Short Proof, Asymptotic Analysis {\bf 2026}, 148(2): 707-725.

\bibitem{Breton2} Breton E., A note on the non-$L^1$ asymptotic completeness of the Vlasov-Maxwell system, arxiv.org/abs/2509.04025.

\bibitem{Flynn} Flynn, P., Ouyang, Z., Pausader, B. et al., Scattering Map for the Vlasov-Poisson System, Peking Math J. {\bf 2023} 6: 365-392. 


\bibitem{Glassey} Glassey, R., The Cauchy Problem in Kinetic Theory.  SIAM: {\bf 1996}.

\bibitem{GPS}  Glassey, R., Pankavich, S., and Schaeffer, J., Decay in Time for a One-Dimensional, Two Component Plasma. Math. Meth Appl. Sci. {\bf 2008}, 31:2115-2132.

\bibitem{GPS2} Glassey, R., Pankavich, S., and Schaeffer, J., On long-time behavior of monocharged and neutral plasma in one and one-half dimensions. Kinetic and Related Models {\bf 2009}, 2: 465-488.

\bibitem{GPS4} Glassey, R., Pankavich, S., and Schaeffer, J., Time Decay for Solutions to the One-dimensional Equations of Plasma Dynamics. Quarterly of Applied Mathematics {\bf 2010}, 68: 135-141.

\bibitem{GPS5} Glassey, R., Pankavich, S., and Schaeffer, J., Large Time Behavior of the Relativistic Vlasov-Maxwell System in Low Space Dimension. Differential \& Integral Equations {\bf 2010} 23: 61-77. 


\bibitem{GS} Glassey, R., Strauss, W., Absence of shocks in an initially dilute collisionless plasma. Comm. Math. Phys. {\bf 1987}, 113, 191-208.

\bibitem{GS2} Glassey, R., Strauss, W., Singularity formation in a collisionless plasma could occur only at high velocities. Archive for rational mechanics and analysis {\bf 1986}, 92, 59-90.



\bibitem{Horst} Horst, E., Symmetric plasmas and their decay. Comm. Math. Phys. {\bf 1990}, 126: 613-633.

\bibitem{HRV} Hwang, H., Rendall, A., and Velazquez, J., Optimal gradient estimates and asymptotic behaviour for the Vlasov-Poisson system with small initial data. Archive for rational mechanics and analysis {\bf 2011}, 200: 313-360.


\bibitem{Ionescu} Ionescu, A., Pausader, B., Wang, X., Widmayer, K., On the asymptotic behavior of solutions to the Vlasov-Poisson system. International Mathematics Research Notices {\bf 2021}, 155, https://doi.org/10.1093/imrn/rnab155.





\bibitem{Mattingly2025} Mattingly, G., Pankavich, S. and Ben-Artzi, J.,  Arbitrary Polynomial Decay Rates of Neutral, Collisionless Plasmas. arXiv:2404.05812 {\bf 2024}.

\bibitem{Pankavich2020} Pankavich, S., Exact Large Time Behavior of Spherically-Symmetric Plasmas. SIAM J. Math. Anal. {\bf 2021} 53(4): 4474-4512.

\bibitem{Pankavich2021} Pankavich, S., Asymptotic Dynamics of Dispersive Plasmas, Communications in Mathematical Physics {\bf 2022}, 391: 455-493.

\bibitem{Pankavich2022} Pankavich, S., Scattering and Asymptotic Behavior of Solutions to the Vlasov-Poisson System in High Dimension, SIAM J. Math. Anal., {\bf 2023}, 55(5): 4727-4750.

\bibitem{Pan-BA2025} Pankavich, S. and Ben-Artzi, J., Modified Scattering of Solutions to the Relativistic Vlasov-Maxwell System Inside the Light Cone, J. London Math. Soc., {\bf 2025}, 112: e70346.



Equations, {\bf 2007}, Eds. C. M. Dafermos and E. Feireisl, Elsevier: 383-479.

\bibitem{Sch} Schaeffer, J., Large-time behavior of a one-dimensional monocharged plasma. Diff. and Int. Equations {\bf 2007}, 20(3): 277-292.


\bibitem{Schlue-Taylor} Schlue, V. and Taylor, M., Inverse modified scattering and polyhomogeneous expansions for the Vlasov--Poisson system. arXiv:2404.15885 {\bf 2024}.

\bibitem{Smulevici} Smulevici, J., Small data solutions of the Vlasov-Poisson system and the vector field method.  Ann. PDE {\bf 2016} 2(2).




\end{thebibliography}
\end{document}